\documentclass[12pt]{article}
\usepackage[margin=0.56in]{geometry}
\usepackage{amssymb,amsmath,amsthm}
\usepackage{graphicx,epstopdf,color}
\usepackage{esint}

\allowdisplaybreaks
\usepackage[alphabetic]{amsrefs}

\usepackage[latin1]{inputenc}
\newcommand{\R}{\mathbb{R}}

\newcommand{\N}{\mathbb{N}}

\usepackage{miama}
\usepackage[T1]{fontenc}
\usepackage{mathtools}

\numberwithin{equation}{section}

\newtheorem{thm}{Theorem}[section]
\newtheorem{cor}[thm]{Corollary}
\newtheorem{lem}[thm]{Lemma}
\newtheorem{prop}[thm]{Proposition}
\newtheorem{defn}[thm]{Definition}
\newtheorem{rem}[thm]{Remark}

\renewcommand{\leq}{\leqslant}

\renewcommand{\geq}{\geqslant}

\usepackage{authblk}

\begin{document}

\title{DENSITY ESTIMATES
FOR A (NON)LOCAL VARIATIONAL MODEL
 WITH DEGENERATE DOUBLE-WELL POTENTIAL}

\author{Serena Dipierro, Alberto Farina, Giovanni Giacomin and Enrico Valdinoci
\thanks{Serena Dipierro, Giovanni Giacomin and Enrico Valdinoci:
Department of Mathematics and Statistics,
University of Western Australia, 35 Stirling Highway,
WA6009 Crawley, Australia.\\
Alberto Farina: LAMFA, UMR CNRS 7352, 
	Universit\'e Picardie Jules Verne
	33, rue St Leu, 80039 Amiens, France.\\
{\tt serena.dipierro@uwa.edu.au, alberto.farina@u-picardie.fr
giovanni.giacomin@research.uwa.edu.au,
enrico.valdinoci@uwa.edu.au\\
SD, GG and EV are members of Australian Mathematical Society.
Supported by the Australian Future Fellowship
FT230100333 ``New perspectives on nonlocal equations''
and the Australian Laureate Fellowship FL190100081 ``Minimal
surfaces, free boundaries and partial differential equations''. 
Helpful discussions with Francesco De Pas and Jack Thompson are acknowledged. AF and GG would like to thank the Institute Henri Poincar\'e where part of this research was performed.}}}

\maketitle

\begin{abstract}
In this paper
we provide density estimates for a class of functions which includes all the minimizers of the energy 
\begin{equation*}
\mathcal{E}_s^p(u,\Omega):=(1-s)\left(\frac{1}{2}\int_{\Omega}\int_{\Omega}\frac{\left|u(x)-u(y)\right|^p}{\left|x-y\right|^{n+sp}}\,dx\,dy +\int_{\Omega}\int_{\R^n\setminus \Omega}\frac{\left|u(x)-u(y)\right|^p}{\left|x-y\right|^{n+sp}}\,dx\,dy\right)+\int_{\Omega}W(u(x))\,dx,
\end{equation*}
where~$p\in (1,+\infty)$, $s \in \left(0,1\right)$ and~$W$ is a double-well potential with polynomial growth~$m\in \left[p,+\infty\right)$ from the minima. The nonlocal estimates obtained are uniform as~$s\to1$.

Moreover, making use of a $\Gamma$-convergence result for $\mathcal{E}_s^p$ as $s\to 1$, we obtain density estimates for the minimizers  of the limit energy functional, which takes the form
\begin{equation*}
\mathcal{E}_1^p(u,\Omega):=\frac{K_{n,p}}{2p}\int_{\Omega} \left|\nabla u(x)\right|^p+\int_{\Omega} W(u(x))\,dx,
\end{equation*} 
for a suitable $K_{n,p}\in (0,+\infty)$.

\end{abstract}

\section{Introduction}
Phase separation is a cross-disciplinary field, which attracts researchers from  physics~\cite{weber2019physics,ter2013collected,MR0523642,gurtin1985theory, ginzburg1958theory,cahn1958free,allen1972ground,MR1002633}, biology~\cite{brangwynne2009germline,alberti2017phase,mitrea2016phase} and social sciences~\cite{harding2020population,castellano2009statistical}. This model is used to describe a broad variety of systems characterized by the interaction of two (or more) components, which are associated to different values of a suitable state parameter. Such a parameter is often expressed by a function $u:\Omega\subset\R^n \to [-1,1]$, where the pure phases are given by $-1$ and $1$.  

In the classical framework, a prototype of energy for phase separation is given by the  non-scaled energy functional
\begin{equation}\label{NSGL}
\mathcal{E}^p(u,\Omega):= \frac{1}{p}\int_{\Omega}\left|\nabla u(x)\right|^p\,dx+\int_{\Omega} W(u(x))\,dx, 
\end{equation}
where $p\in (1,+\infty)$. The function $W:[-1,1]\to \R^+$ is a double-well potential satisfying
\begin{equation*}
W(t)>0\quad\mbox{for all}\quad t\in (-1,1)\quad\mbox{and}\quad W(\pm 1)=0. 
\end{equation*}

In this framework, phase separation is induced by the minimization of $\mathcal{E}^p$. More precisely,
on the one hand, the action of the potential $W$ pushes the minimizers of $\mathcal{E}^p$ to attain values as close as possible to its zeros, namely the pure phases $\pm 1$. Furthermore, the presence of the gradient term of $L^p$-type discourages the formation of subregions of $\Omega$ where the state parameter jumps from one pure phase to the other. These regions, which are characterized by intermediate values of the state parameter, are known as ``interfaces''. 

A common matter of investigation in physics, biology, mathematical physics and analysis is the description of these interfaces. An interesting approach to this problem is provided by density estimates.    

The aim of density estimates is to establish lower bounds for the measure of the portion of the domain where the state parameter attains values close to the pure phases. The first result in this direction was obtained by L. Caffarelli and A. C\'ordoba. In~\cite{MR1310848} they proved that if $u:\Omega\to [-1,1]$ is a minimizer for $\mathcal{E}^2$, then for every $\theta\in (-1,1)$ and $r\in (0,+\infty)$ large enough it holds that $\left\lbrace \left|u\right|>\theta \right\rbrace\cap B_r$ is comparable to $r^n$. On the other hand, they showed that the
measure of the interface $\left\lbrace \left|u\right|\leq \theta\right\rbrace\cap B_r$ behaves like $r^{n-1}$.

Since the pioneering work of L. Caffarelli and A. C\'ordoba, density estimates have been established for minimizers and quasiminimizers of a larger class of energy functionals associated to phenomena of phase separation. In~\cite{MR2126143} and~\cite{6} density estimates have been proved for minimizers of $\mathcal{E}^p$ with $p\in (1,+\infty)$. Also, $\chi$-shaped potentials (i.e., piecewise constant potentials) are considered, which are associated to models for fluid jets. In~\cite{MR2413100} density estimates are proved for the same energies but under the weaker assumption of quasiminimality.  

Nonlocal counterparts of the energies studied in~\cite{MR1310848,MR2126143,6,MR2413100} have already been considered in the literature. A common approach to generalize phase separation in a nonlocal framework is to replace the gradient term of $L^p$-type with the Gagliardo seminorm. With such a choice, the energy used to describe nonlocal phase separation becomes 
\begin{equation}\label{fvdcteju6543}
\mathcal{E}_s^p(u,\Omega):=\frac{1}{2}\int_{\Omega}\int_{\Omega}\frac{\left|u(x)-u(y)\right|^p}{\left|x-y\right|^{n+sp}}\,dx\,dy+\int_{\Omega}\int_{\R^n\setminus \Omega}\frac{\left|u(x)-u(y)\right|^p}{\left|x-y\right|^{n+sp}}\,dx\,dy+\int_{\Omega}W(u(x))\,dx,
\end{equation} 
where $s\in (0,1)$ and $p\in (1,+\infty)$. 

Density estimates for the energy functional in~\eqref{fvdcteju6543} with $p=2$ were proved in~\cite{MR2873236} for every $s\in \left(0,\frac{1}{2}\right)$, making use of the Sobolev inequality. Moreover, in~\cite{MR3133422}, the same authors generalized the result for every $s\in (0,1)$.      

In all the aforementioned papers on density estimates, the potential $W$ is taken such that its growth from the minima is comparable to a polynomial of degree $m$, where $m\in \left(1,p\right]$ and $p\in (1,+\infty)$ is the homogeneity of the gradient term. Potentials $W$ that detach from their minima with $m\geq p$ are known as \textit{degenerate potentials}, and to the best of our knowledge these have been previously studied in the framework of density estimates only in two others  research articles, namely~\cite{1} and~\cite{DFVERPP}.     

In the local setting, see~\cite{1} and the recent~\cite{savin2025density}, density estimates for degenerate double-well potentials have been proved for quasiminimizers of $\mathcal{E}^p$. For our purposes, it is also important to recall that the result in~\cite{1} came with an additional restriction
on the growth of the potential with respect to its wells. Indeed, in Theorem~1.1 of~\cite{1}, density estimates are showed as far as $p\in (1,+\infty)$, $m\in \left[p,+\infty\right)$ and $n$ satisfy 
\begin{equation}\label{vernacu}
\frac{pm}{m-p}>n,  
\end{equation}
see~(1.7) in~\cite{1}. In the most recent paper~\cite{savin2025density}, density estimates for quasiminimizers have been showed without the additional assumption in~\eqref{vernacu}, but with stronger conditions on the quasiminimizer, see Remark~\ref{removiult}.

In the nonlocal setting, see~\cite{DFVERPP}, density estimates for degenerate double-well potentials have been proved for minimizers of $\mathcal{E}_s^p$ for every $p\in (1,+\infty)$ and $s\in \left(0,\frac{1}{p}\right)$. This result was obtained through the Sobolev inequality, with an approach similar to the one used by E. Valdionci and O. Savin in~\cite{MR2873236}. Nevertheless, the presence of a degenerate double-well potential together with a homogeneity $p\neq 2$ required the introduction of a new barrier bespoke for the specifics of the problem under consideration, see Theorem~3.1 in~\cite{DFVERPP}.    
   
In this paper, we consider the renormalized non-scaled free energy 
\begin{equation}\label{holo-polo}
\mathcal{E}_s^p(u,\Omega):=(1-s)\left(\frac{1}{2}\int_{\Omega}\int_{\Omega}\frac{\left|u(x)-u(y)\right|^p}{\left|x-y\right|^{n+sp}}\,dx\,dy+\int_{\Omega}\int_{\R^n\setminus \Omega}\frac{\left|u(x)-u(y)\right|^p}{\left|x-y\right|^{n+sp}}\,dx\,dy\right)+\int_{\Omega}W(u(x))\,dx, 
\end{equation} 
with $W$ being a degenerate double-well potential. In particular, we complete the work begun in~\cite{DFVERPP}, and establish, for every $s\in (0,1)$ and $p\in (1,+\infty)$, density estimates for a class of state parameters which includes all minimizers of $\mathcal{E}_s^p$, see Theorems~\ref{th:fracp>=2} and~\ref{CapVieABall} below.  

The presence of $(1-s)$ in front of the kinetic term is needed to guarantee that as $s\to 1$ the energy $\mathcal{E}_s^p$ $\Gamma$-converges to an energy proportional to the one in~\eqref{NSGL}, see Theorem~\ref{12c-dgbnbc5543} below. From this we will be able to obtain density estimates for minimizers of~\eqref{NSGL} starting from the nonlocal density estimates obtained in Theorem~\ref{CapVieABall}. As a byproduct, we are able to show that the condition in~\eqref{vernacu} for minimizers is not sharp, and density estimates hold for the local and degenerate case for every $n\geq 1$, $p\in (1,+\infty)$ and $m\in \left[p,+\infty\right)$, see Theorem~\ref{viredghloprec45}. 

The precise hypotheses on $W$, the statement of Theorems~\ref{th:fracp>=2},~\ref{CapVieABall},~\ref{12c-dgbnbc5543} and~\ref{viredghloprec45} and their corollaries are all discussed in the following section.

\section{Mathematical framework and main results}
In what follows we let $n\in\N$, $\Omega\subset\R^n$ open, $s\in(0,1)$ and $p\in(1,+\infty)$. We consider a potential $W:[-1,1]\to\R^+$ such that 
\begin{equation}\label{po12345}
W\mbox{  is lower semicontinuous}.
\end{equation} 
Also, we suppose that, for every $x\in [-1,1]$,
\begin{equation}\label{potential}
\lambda\chi_{(-\infty,\theta]}(x)\left|1+x\right|^{m}   \leq W(x)\leq \Lambda\left|1+x\right|^m,
\end{equation}
for some $\lambda\in (0,1]$, $\Lambda\in [1,+\infty)$, $\theta\in (-1,1)$ and~$m\in \left[p,+\infty\right)$.

Moreover, we assume that, for all $-1\leq r\leq t\leq -1+q$,
\begin{equation}\label{potntiald2}
W(t)-W(r)\geq c_1(t-r)\left|1+r\right|^{m-1}+c_1(t-r)^k,
\end{equation}
for some $q\in (0,1)$ and $c_1\in (0,+\infty)$,
where 
\begin{equation}\label{key?}
k:=\begin{dcases}
m\quad &\mbox{if}\quad m\in \N\\
\left[m\right]+1 \quad &\mbox{if}\quad m \in(1,+\infty)\setminus \N. 
\end{dcases}
\end{equation}

Given~$m\in (1,+\infty)$, the prototype of double-well potential that we have in mind has the form
\begin{equation*}
W_m(x):=\frac{(1-x^2)^m}{2m}, 
\end{equation*}
since one can show the validity of condition~\eqref{potntiald2} for this model case (see
In Proposition~\ref{colllise}).

For every measurable $u:\R^n\to \R$ we define
\begin{equation}\label{kinetic}
\mathcal{K}_s^p(u,\Omega):=\frac{1}{2}\int_{\Omega}\int_{\Omega}\frac{\left|u(x)-u(y)\right|^p}{\left|x-y\right|^{n+sp}}\,dx\,dy +\int_{\Omega}\int_{\R^n\setminus \Omega}\frac{\left|u(x)-u(y)\right|^p}{\left|x-y\right|^{n+sp}}\,dx\,dy
\end{equation} 
and rewrite the energy in~\eqref{holo-polo} as
\begin{equation}\label{energy}
\mathcal{E}_{s}^p(u,\Omega):= (1-s)\mathcal{K}_s^p(u,\Omega)+\int_{\Omega}W(u)\,dx.
\end{equation}
The normalizing constant $(1-s)$ in front of $\mathcal{K}_s^p$ is needed to obtain
stable estimates as~$s\to1$ and
recover density estimates for the local energy in~\eqref{NSGL} starting from the nonlocal ones. We will achieve such a result, see Theorem~\ref{viredghloprec45}, by proving first that $\mathcal{E}_s^p$ $\Gamma$-converges as $s\to 1$ to an energy proportional to the one in~\eqref{NSGL}, see Theorem~\ref{12c-dgbnbc5543}. This point will be discussed more in detail in Section~\ref{ere-optre}. 

A suitable function space to study~$\mathcal{K}_s^p$ is
\begin{equation}
\widehat{H}^{s,p}(\Omega):=\left\lbrace u:\R^n\to \R\mbox{  measurable s.t.  } \mathcal{K}_{s}^p(u,\Omega)<+\infty  \right\rbrace.
\end{equation}
Since the potential $W$ is defined in the interval $[-1,1]$, the suitable subspace of $\widehat{H}^{s,p}(\Omega)$ needed to study $\mathcal{E}_s^p$ is  
\begin{equation}\label{AdSet}
X^{s,p}(\Omega):=\left\lbrace u\in \widehat{H}^{s,p}(\Omega)\mbox{  s.t.  }\left\| u\right\|_{L^\infty(\R^n)}\leq 1\right\rbrace.
\end{equation}
Moreover, it will be useful to consider the subspace of $X^{s,p}(\Omega)$ of functions with a prescribed trace. Hence, given any measurable $f:\R^n\to [-1,1]$  we denote the subspace of $X^{s,p}(\Omega)$ with trace $f$, i.e. 
\begin{equation*}
X_f^{s,p}(\Omega):=\left\lbrace u\in X^{s,p}(\Omega)\mbox{  s.t.  }u=f\mbox{  a.e. in  }\R^n\setminus \Omega   \right\rbrace.
\end{equation*}
Then, we have the following definition:
\begin{defn}[$\epsilon$-minimizer]\label{epsilonmimimizers}
If $\Omega\subset\R^n$ is bounded, given $\epsilon\in [0,+\infty)$, we say that $u\in X^{s,p}(\Omega)$ is an $\epsilon$-minimizer for $\mathcal{E}_s^p$ if, for every $v\in X_u^{s,p}(\Omega)$,
\begin{equation*}
\mathcal{E}_s^p(u,\Omega)\leq \epsilon+\mathcal{E}_s^p(v,\Omega).
\end{equation*}
If $\Omega\subset\R^n$ is unbounded, we say that $u\in X^{s,p}(\Omega)$ is an $\epsilon$-minimizer for $\mathcal{E}_s^p(\cdot,\Omega)$ if for any $\Omega'\subset \subset \Omega$ it is a minimizer of $\mathcal{E}_s^p(\cdot,\Omega')$. 
\end{defn}
\begin{rem}\label{lfvc}
Note that if $u$ is an $\epsilon$-minimizer for $\mathcal{E}_s^p$ in $\Omega\subset\R^n$ open and bounded, then $u$ $\epsilon$-minimizes $\mathcal{E}_s^p$ in every measurable subdomain $\Omega'\subset\Omega$. 
\end{rem}
We observe that for $\epsilon=0$ one recovers the definition of minimizer of $\mathcal{E}_s^p$, see for e.g. Definition~1.1 in~\cite{DFVERPP}. It is also interesting to compare the notion of $\epsilon$-minimizer with the notion of quasiminimizer. This was introduced in the classical variational framework by E. Giusti and M. Giaquinta in~\cite{MR0666107,giaquinta1984quasi}. 

We provide here a definition of quasiminimizer for the nonlocal setting of the energy $\mathcal{E}_s^p$:
\begin{defn}[$Q$-minimizer for $\mathcal{E}_s^p$]
For every open set $\Omega\subset\R^n$ and $Q\in [1,+\infty)$ we say that $u\in X^{s,p}(\Omega)$ is a $Q$-minimizer for $\mathcal{E}_s^p$ in $\Omega$ if for every open set $A\subset\Omega$ and $v\in X_{u}^{s,p}(A)$ it holds that 
\begin{equation}
\mathcal{E}_s^p(u,A)\leq Q\mathcal{E}_s^p(v,A). 
\end{equation}
\end{defn}
Then, in the following remark, we show that for every $Q$ and $\epsilon$ there is a $\epsilon$-minimizer which is not a $Q$-minimizer.
\begin{rem}
In this remark we show that for every $Q\in (1,+\infty)$ and $\tilde{\epsilon}\in (0,1)$, there exists some $\tilde{\epsilon}$-minimizer of $\mathcal{E}_s^p$ in $B_1$ which is not a $Q$-minimizer. To do so, we introduce the function
\begin{equation}\label{cwbf-86yh}
u(x):=\begin{dcases}
1\quad &\mbox{for all}\quad x\in B_{e^{-2}}\\
\left|\ln(\left|x\right|)\right|-1\quad &\mbox{for all}\quad x\in B_1\setminus B_{e^{-2}}\\
-1\quad &\mbox{for all}\quad x\in \R^n\setminus B_1. 
\end{dcases}
\end{equation} 
We observe that 
\begin{equation}\label{cwbf-86yh}
\left\| \nabla u \right\|_{L^\infty(\R^n)}\leq e^2\quad\mbox{and}\quad \left\|u\right\|_{L^\infty(\R^n)}\leq 1.
\end{equation}
From this it follows that 
\begin{equation}\label{locrebgttt5}
\begin{split}
\mathcal{K}_s^p\left(u,B_{1}\right) &\leq \int_{B_1}\int_{B_2}\frac{\left|u(x)-u(y)\right|^p}{\left|x-y\right|^{n+sp}}\,dx\,dy+\int_{\R^n\setminus B_2}\int_{B_1}\frac{\left|u(x)-u(y)\right|^p}{\left|x-y\right|^{n+sp}}\,dx\,dy\\
&\leq \left\|\nabla u \right\|_{L^\infty(\R^n)}^p \int_{B_1}\int_{B_2}\frac{dx\,dy}{\left|x-y\right|^{n+sp-p}}+2^p\left\|u\right\|_{L^\infty(\R^n)}^p\int_{\R^n\setminus B_2}\int_{B_1} \frac{dx\,dy}{\left|x-y\right|^{n+sp}} \\
&\leq e^{2p} \int_{B_1}\int_{B_2}\frac{dx\,dy}{\left|x-y\right|^{n+sp-p}}+2^p\int_{\R^n\setminus B_2}\int_{B_1} \frac{dx\,dy}{\left|x-y\right|^{n+sp}} \\
&\leq C_{n,s,p},
\end{split}
\end{equation}
for some $C_{n,s,p}\in (0,+\infty)$. In particular, thanks to~\eqref{cwbf-86yh} and this last equation we deduce that $u\in X^{s,p}(B_1)$. Also, we notice that there exists some constant $c_{n,s,p}>0$ such that
\begin{equation}\label{ioprefvv}
\mathcal{K}_s^p(u,B_1)\geq c_{n,s,p}. 
\end{equation}
Now, we rescale $u$ as 
\begin{equation}
u_\epsilon(x):=u\left(\frac{x}{\epsilon}\right). 
\end{equation}
and we claim that 
\begin{equation}\label{im-btre54778}
\mbox{for every $\tilde{\epsilon}\in (0,1)$ there exists some $\epsilon>0$ such that $u_\epsilon$ is a $\tilde{\epsilon}$-minimizer in $B_1$}. 
\end{equation}
First, we observe that thanks to~\eqref{potential} we have that 
\begin{equation}\label{ppppp7777734}
\int_{B_1}W\left(u_{\epsilon}(x)\right)\,dx=\int_{B_\epsilon}W\left(u_{\epsilon}(x)\right)\,dx 
=\epsilon^n\int_{B_1} W(u(x))\,dx
\leq \epsilon^n \Lambda 2^m \mathcal{L}^n\left(B_1\right).   
\end{equation}
Moreover, by scaling and making use of~\eqref{locrebgttt5} together with the fact that $supp(u+1)\subset B_1$ we obtain that
\begin{equation}\label{invefdrt54}
\mathcal{K}_s^p \left(u_\epsilon,B_1\right)=\epsilon^{n-sp} \mathcal{K}_s^p\left(u,B_{\frac{1}{\epsilon}}\right)=\epsilon^{n-sp} \mathcal{K}_s^p(u,B_1)
\leq  \epsilon^{n-sp} C_{n,s,p}. 
\end{equation}
In particular, from equation~\eqref{ppppp7777734} and~\eqref{invefdrt54}, and choosing $\epsilon$ small enough we obtain that  
\begin{equation*}
\begin{split}
\mathcal{E}_s^p(u_\epsilon,B_1)&=(1-s)\mathcal{K}_s^p(u_\epsilon,B_1)+\int_{B_1}W\left(u_\epsilon\right)\,dx\\
&\leq \epsilon^{n-sp}C_{n,s,p}(1-s)+ \epsilon^n \Lambda 2^m\mathcal{L}^n(B_1)\\
&=\epsilon^{n-sp}\left(C_{n,s,p}(1-s)+\epsilon^{sp} \Lambda 2^m\mathcal{L}^n(B_1)\right)\\
&\leq \tilde{\epsilon}. 
\end{split}
\end{equation*}
From this, it follows immediately that for every $v\in X_u^{s,p}(B_1)$ it holds that 
\begin{equation*}
\mathcal{E}_s^p(u,B_1)\leq \tilde{\epsilon}+\mathcal{E}_s^p(v,B_1)
\end{equation*}
proving claim~\eqref{im-btre54778}. Furthermore, we observe that as a consequence of~\eqref{ioprefvv} 
\begin{equation}\label{olt-uterev}
\mathcal{E}_s^p(u_\epsilon,B_\epsilon)\geq (1-s)\mathcal{K}_s^p(u_\epsilon,B_\epsilon)=(1-s)\epsilon^{n-s} \mathcal{K}_s^p(u,B_1)\geq c_{n,s,p}(1-s)\epsilon^{n-sp}.
\end{equation} 
Now, that we proved claim~\eqref{im-btre54778}, we observe that there exists no $Q\in [1,+\infty)$ such that $u_\epsilon$ is a $Q$-minimizer in $B_1$. As a matter of fact, we observe that if $v\equiv 1\in X_{u_\epsilon}^{s,p}(B_\epsilon)$, using~\eqref{olt-uterev} we obtain 
\begin{equation*}
\mathcal{E}_s^p(v,B_\epsilon)=0< (1-s)c_{s,n,p}\epsilon^{n-sp}\leq \mathcal{E}_s^p(u_\epsilon,B_\epsilon). 
\end{equation*}
\end{rem}
We also observe that every $Q$-minimizer $u$ in $\Omega$ is an $\epsilon$-minimizer, for an $\epsilon$ which depends on $\mathcal{E}_s^p(u,\Omega)$. More precisely: 
\begin{rem}\label{EnQBo-omnre}
Suppose that $Q\in (1,+\infty)$ and that $u\in X^{s,p}(\Omega)$ is a $Q$-minimizer for~$\mathcal{E}_s^p$ in $\Omega\subset\R^n$. Furthermore, we let $E_0\in (0,+\infty)$ such that 
\begin{equation*}
\mathcal{E}_s^p(u,\Omega)\leq E_0. 
\end{equation*}
Then, we have that 
\begin{equation*}
\mbox{$u$ is an $\epsilon$-minimizer for $\mathcal{E}_s^p$ for  $\epsilon:=(Q-1)E_0$.} 
\end{equation*}
To show this, we need to prove that for every $v\in X_u^{s,p}(\Omega)$ it holds that 
\begin{equation}\label{cfamlbctey}
\mathcal{E}_s^p\left(u,\Omega\right)\leq \epsilon+\mathcal{E}_s^p\left(v,\Omega\right). 
\end{equation}
Indeed, we observe that if $v\in X_u^{s,p}(\Omega)$, then we  have that either 
\begin{equation}\label{gdue}
\mathcal{E}_s^p(u,\Omega)<\mathcal{E}_s^p(v,\Omega)
\end{equation}
or
\begin{equation}\label{guno}
\mathcal{E}_s^p(u,\Omega)\geq \mathcal{E}_s^p(v,\Omega).
\end{equation}
If~\eqref{gdue} holds true than~\eqref{cfamlbctey} follows immediately. Now, we assume that~\eqref{guno} holds true. Then, in this case
\begin{equation*}
\begin{split}
\mathcal{E}_s^p(u,\Omega) &\leq Q\mathcal{E}_s^p(v,\Omega)\\
&\leq \mathcal{E}_s^p(v,\Omega)+(Q-1)\mathcal{E}_s^p(v,\Omega)\\
&=\mathcal{E}_s^p(v,\Omega)+(Q-1)\mathcal{E}_s^p(u,\Omega)\\
&\leq \mathcal{E}_s^p(v,\Omega)+(Q-1)E_0\\
&=\mathcal{E}_s^p(v,\Omega)+\epsilon.
\end{split}
\end{equation*}
\end{rem}
For every $u\in C_{\textit{loc}}^\alpha(\Omega)$, we define the quantity 
\begin{equation}\label{gliorteitbergf}
\left[u\right]_{B,\alpha}:=\sup\left\lbrace \left[u\right]_{C^\alpha (B_1(x_0))}\mbox{  s.t.  } x_0\in B_R\mbox{ and }B_{R+2}\subset \Omega \right\rbrace
\end{equation}
and we consider the subspace of $C_{\textit{loc}}^\alpha(\Omega)$ given by 
\begin{equation*}
C_B^\alpha(\Omega):=\left\lbrace u\in C_{\textit{loc}}^\alpha(\Omega)\mbox{  s.t.  } \left[u\right]_{B,\alpha}<+\infty \right\rbrace.
\end{equation*}
Note that for every $\Omega\neq \R^n$  we have that $C_{\textit{loc}}^\alpha(\Omega)=C_B^\alpha(\Omega)$. 

Now we introduce the mathematical setting necessary to study the local contribution to the energy functional in~\eqref{NSGL}. To do so, we let
\begin{equation*}
K_{n,p}:=\int_{\partial B_1}\left|\omega\cdot e_1\right|^p\,dH_{\omega}^{n-1}
\end{equation*}
and, for every $u\in W^{1,p}(\Omega)$,
\begin{equation*}
\mathcal{K}_1^p(u,\Omega):=\frac{K_{n,p}}{2p}\int_{\Omega} \left|\nabla u (x)\right|^p\,dx.
\end{equation*}
Furthermore, we define
\begin{equation*} 
X^{1,p}(\Omega):=\left\lbrace u\in W^{1,p}(\Omega)\mbox{  s.t.  }\left\| u\right\|_{L^\infty(\Omega)}\leq 1 \right\rbrace
\end{equation*}
and, for every $u \in X^{1,p}(\Omega)$, we consider the energy functional
\begin{equation}\label{localene}
\mathcal{E}_1^p(u,\Omega):=\mathcal{K}_1^p(u,\Omega)+\int_{\Omega}W(u(x))\,dx.
\end{equation}
Also, as customary we denote by $W_0^{1,p}(\Omega)$ the closure of $C_c^\infty(\Omega)$ with respect to the  $\left\|\cdot\right\|_{W^{1,p}(\Omega)}$-norm and we recall the definition of minimizer for $\mathcal{E}_1^p$: 
\begin{defn}[Minimizer for $\mathcal{E}_1^p$]
If $\Omega\subset\R^n$ is open and bounded, we say that $u\in X^{1,p}(\Omega)$ is a minimizer for $\mathcal{E}_1^p$ in $X^{1,p}(\Omega)$ if for every $w\in X^{1,p}(\Omega)$ such that $u-w\in W_0^{1,p}(\Omega)$ 
\begin{equation*}
\mathcal{E}_1^p(u,\Omega)\leq \mathcal{E}_1^p(w,\Omega). 
\end{equation*} 
\end{defn}

\subsection{Nonlocal density estimates}
We are now ready to state the main result of this paper: 

\begin{thm}[Density estimates for $\epsilon$-minimizers]\label{th:fracp>=2}
Let $\theta_1,\theta_2\in (-1,\theta]$ and 
\begin{equation}\label{liibfdcvre5555}
\theta_*:=\min \left\lbrace \theta_1,\theta_2,-1+q \right\rbrace\quad\mbox{and}\quad \theta^*:=\max\left\lbrace  \theta_1,\theta_2,-1+q\right\rbrace.
\end{equation}
Also, let $u\in X^{s,p}(\Omega)$ be such that
\begin{equation}\label{Fcon-812345}
\mathcal{L}^{n}(B_{r_0}\cap \left\lbrace u>\theta_1 \right\rbrace)>c_0,
\end{equation}
for some $c_0,r_0\in (0,+\infty)$ satisfying $B_{r_0}\subset \Omega$.

Then, there exists~$\delta:=\delta_{s,n,p,m,c_1,\theta_*,c_0}\in(0,+\infty)$ such that, if $\epsilon\in [0,\delta)$ and $u$ is an $\epsilon$-minimizer for $\mathcal{E}_s^p$, the following holds true:

There exist $R^*:=R_{s,n,p,m,c_1,\theta_*,r_0}^*\in [r_0,+\infty)$, $\tilde{c}:=\tilde{c}_{s,n,p,m,\Lambda,c_1,\theta_*,r_0,c_0}\in (0,1)$ and $c:=c_{m,\theta_*}\in (0,+\infty)$ such that, for any $r\in \left[R^*,+\infty\right)$ satisfying $B_{3r}\subset \Omega$,
\begin{equation}\label{benz}
c\int_{B_r\cap \left\lbrace \theta_{*}<u\leq \theta^* \right\rbrace}\left|1+u(x)\right|^m\,dx+\mathcal{L}^n(B_r\cap \left\lbrace u>\theta_2 \right\rbrace)>\tilde{c}r^n.
\end{equation}
\end{thm}

Here above and in the rest of the paper, $\mathcal{L}^n(E)$ denotes the Lebesgue measure of a set $E \subset \R^n$. We observe that if $p\in (1,+\infty)$, $s\in \left(0,\frac{1}{p}\right)$ and $\epsilon=0$ then Theorem~\ref{th:fracp>=2} coincides with Theorem~1.3 in~\cite{DFVERPP}. However, if $p\neq 2$ and $m\geq p$ or $p=2$ and $m>2$, even the cases $\epsilon=0$ and $s\in\left[\frac{1}{p},1\right)$ are new in the literature. Moreover, if $\epsilon>0$, then Theorem~\ref{th:fracp>=2} is new for every $s\in (0,1)$, $p\in (1,+\infty)$ and $m\in \left[p,+\infty\right)$.

\medskip

We also observe that, making use of the same proof used to show Theorem~\ref{th:fracp>=2}, we can also obtain Theorem~\ref{ibcverty6} below. The difference is that in Theorem~\ref{th:fracp>=2}, given a function $u$ satisfying the condition in~\eqref{Fcon-812345} for some $c_0,r_0\in (0,+\infty)$, we determine a range of $\epsilon$ such that if $u$ is an $\epsilon$-minimizer, then $u$ satisfies density estimates. 

On the other hand, in Theorem~\ref{ibcverty6}, we show that for every $\epsilon$ there exists some $c_0,r_0$ such that if $u$ satisfies~\eqref{Fcon-812345} for such $c_0,r_0$ and it is and $\epsilon$-minimizer, then density estimates hold.

\begin{thm}\label{ibcverty6}
Let $\theta_1,\theta_2\in (-1,\theta]$ and $\theta_*$, $\theta^*$ as in~\eqref{liibfdcvre5555}. Moreover, let $\epsilon\in (0,+\infty)$ and $u\in X^{s,p}(\Omega)$ be an $\epsilon$-minimizer for~$\mathcal{E}_{s}^p$. Then, there exists some $c_0,r_0\in (0,+\infty)$, depending on $\epsilon$, such that if $B_{r_0}\subset \Omega$
and
\begin{equation*}
\mathcal{L}^{n}(B_{r_0}\cap \left\lbrace u>\theta_1 \right\rbrace)>c_0,
\end{equation*}
then, there exist $R^*:=R_{s,n,p,m,c_1,\theta_*,r_0}^*\in [r_0,+\infty)$, $\tilde{c}:=\tilde{c}_{s,n,p,m,\Lambda,c_1,\theta_*,r_0,c_0}\in (0,1)$ and $c:=c_{m,\theta_*}\in (0,+\infty)$ such that, for any $r\in \left[R^*,+\infty\right)$ satisfying $B_{3r}\subset \Omega$,
\begin{equation*}
c\int_{B_r\cap \left\lbrace \theta_{*}<u\leq \theta^* \right\rbrace}\left|1+u(x)\right|^m\,dx+\mathcal{L}^n(B_r\cap \left\lbrace u>\theta_2 \right\rbrace)>\tilde{c}r^n.
\end{equation*}
\end{thm}

Now, we recall that in Remark~\ref{EnQBo-omnre} we showed that every $Q$-minimizer is an $\epsilon$-minimizer, for an $\epsilon$ which depends on $Q$ and on the energy of the quasiminimizer. Making use of such a result together with Theorem~\ref{th:fracp>=2}, we obtain density estimates for quasiminimizers of $\mathcal{E}_s^p$ with uniform energy bound. More precisely, we have the following result: 

\begin{thm}
Let $\theta_1,\theta_2\in (-1,\theta]$  and $\theta_*$, $\theta^*$ as in~\eqref{liibfdcvre5555}. Furthermore, we assume that $u\in X^{s,p}(\Omega)$ satisfies~\eqref{Fcon-812345} for some $c_0,r_0\in (0,+\infty)$ satisfying~$B_{r_0}\subset\Omega$. 
Then, for every $E_0\in (0,+\infty)$ there exists some $Q_0\in (1,+\infty)$ such that if $Q\in \left[1,Q_0
\right)$,  
\begin{equation*}
\mathcal{E}_s^p(u,\Omega)\leq E_0
\end{equation*}
and $u$ is a $Q$-minimizer of~$\mathcal{E}_s^p$ in $\Omega$, then there exist $R^*:=R_{s,n,p,m,c_1,\theta_*,r_0}^*\in [r_0,+\infty)$, $\tilde{c}:=\tilde{c}_{s,n,p,m,\Lambda,c_1,\theta_*,r_0,c_0}\in (0,1)$ and $c:=c_{m,\theta_*}\in (0,+\infty)$ such that, for any $r\in \left[R^*,+\infty\right)$ satisfying $B_{3r}\subset \Omega$,
\begin{equation*}
c\int_{B_r\cap \left\lbrace \theta_{*}<u\leq \theta^* \right\rbrace}\left|1+u(x)\right|^m\,dx+\mathcal{L}^n(B_r\cap \left\lbrace u>\theta_2 \right\rbrace)>\tilde{c}r^n.
\end{equation*}
\end{thm}

\begin{rem}
We remark that in the framework of density estimates for minimizers and quasiminimizers, the condition in~\eqref{Fcon-812345} is often replaced by a pointwise condition, see for instance~(1.6) in~\cite{MR3133422}. More precisely, for a given point $x_0\in \Omega$, which for convenience is chosen to be $x_0=0$, one can assume that   
\begin{equation}\label{pointcondi}
u(0)>\theta_1. 
\end{equation}
Then, from~\eqref{pointcondi}, as a consequence of the continuity for minimizers and quasiminimizers of~$\mathcal{E}_s^p$, see e.g.~\cite{MR0666107,giaquinta1984quasi},  one obtains the existence of some $r_0,c_0\in (0,+\infty)$ such that~\eqref{Fcon-812345} is satisfied. 

\medskip
We observe that in the framework of $\epsilon$-minimizers it is not possible to replace~\eqref{Fcon-812345} with~\eqref{pointcondi}. To show this, we assume that $sp\in (0,1)$, and given $\eta\in (0,1)$ we consider the discontinuous function 
\begin{equation*}
u_\eta(x):=\begin{dcases}
-1\quad &\mbox{for all}\quad x\in \R^n\setminus  B_{\eta}\\
1\quad &\mbox{for all}\quad x\in B_{\eta}.
\end{dcases}
\end{equation*} 
Then, we claim that for every $R>\mu$ and
\begin{equation}\label{klmaim}
\mbox{for every $\epsilon\in (0,1)$ there exists some $\eta\in (0,1)$ such that $u_\eta$ is an $\epsilon$-minimizer for $\mathcal{E}_{s,p}$ in $B_R$.}
\end{equation}
Indeed, we can compute that 
\begin{equation}\label{pl1234}
\int_{B_R}W\left(u_\eta(x)\right)\,dx=0.
\end{equation}
Also, we have that 
\begin{equation}\label{pm5678}
\begin{split}
\mathcal{K}_s^p\left(u_\eta,B_R\right)&=\frac{1}{2}\int_{B_R}\int_{B_R}\frac{\left|u_\eta(x)-u_\eta(y)\right|^p}{\left|x-y\right|^{n+sp}}\,dx\,dy+\int_{B_R}\int_{\R^n\setminus B_R}\frac{\left|u_\eta(x)-u_\eta(y)\right|^p}{\left|x-y\right|^{n+sp}}\,dx\,dy\\
&=2^{p-1}\int_{B_R\setminus B_\eta}\int_{B_\eta}\frac{dx\,dy}{\left|x-y\right|^{n+sp}}+2^p\int_{\R^n\setminus B_R}\int_{B_\eta}\frac{dx\,dy}{\left|x-y\right|^{n+sp}}\\
&\leq 2^p\int_{\R^n\setminus B_\eta}\int_{B_\eta}\frac{dx\,dy}{\left|x-y\right|^{n+sp}}\\
&=c_{s,n,p} \eta^{1-sp},
\end{split}
\end{equation}
for some $c_{s,n,p}>0$. From~\eqref{pl1234} and~\eqref{pm5678} it follows that for every $\epsilon\in (0,1)$ there exists some $\eta\in (0,1)$ such that for every $v\in X_{u_\eta}^{s,p}(\Omega)$
\begin{equation*}
\mathcal{E}_s^p\left(u_\eta,B_R\right)\leq (1-s)c_{n,p}\eta^{1-sp} = \epsilon\leq \epsilon+ \mathcal{E}_s^p(v,\Omega), 
\end{equation*} 
proving claim~\eqref{klmaim}. 

Now, on one hand, we notice that~$u_\eta$ satisfies~\eqref{pointcondi}. On the other hand, we observe that for every $\eta\in (0,1)$ and $R>0$ there is no $R^*$ and $\tilde{c}$ such that~\eqref{benz} holds true in $B_R$. This is a consequence of the fact that even though $u_\eta$ satisfies~\eqref{Fcon-812345} for some $c_0,r_0\in (0,+\infty)$, the dependence of $\delta$ on $c_0$ is such that $\epsilon>\delta$.   

In other words, the parameter $c_0$ for which~\eqref{Fcon-812345} holds true determines the magnitude of the $\epsilon$ for which $u$ must be an $\epsilon$-minimizer so that it satisfies density estimates. 
\end{rem}

In the following result, we establish density estimates for $\epsilon$-minimizers that are stable with respect to $s\to 1$. To be more precise, for $s_0\in \left(\frac{1}{p},1\right)$ and $s\in \left[s_0,1\right)$, we show that density estimates hold with structural constants which can be chosen independently from $s$, but depending on $s_0$.  The proof of this result relies on the additional regularity assumption that the $\epsilon$-minimizer is H\"older continuous. The precise statement of the result is the following:

\begin{thm}[Stable density estimates for $\epsilon$-minimizers]\label{CapVieABall}
Let $s_0\in \left(\frac{1}{p},1\right)$, $s\in \left[s_0,1\right)$, $\theta_1,\theta_2\in (-1,\theta]$ and $\theta_*,\theta^*$ be defined as in~\eqref{liibfdcvre5555}. Also, let $\alpha\in (0,1)$ and $u\in C_B^\alpha(\Omega)\cap X^{s,p}(\Omega)$ such that there exist $c_0,r_0\in (0,+\infty)$ satisfying $B_{r_0}\subset\Omega$ and~\eqref{Fcon-812345}.

Then, there exists~$\delta:=\delta_{s_0,n,p,m,c_1,\theta_*,c_0,\alpha,\left[u\right]_{B,\alpha}}\in (0,+\infty)$ such that, if $\epsilon\in \left[0,\delta\right)$ and $u$ is an $\epsilon$-minimizer for $\mathcal{E}_s^p$ the following holds true:

There exist~$R^{*}:= R_{s_0,n,p,m,c_1,\theta_*,r_0}^{*}\in \left[r_0,+\infty\right)$, $\tilde{c}:=\tilde{c}_{s_0,n,p,m,\Lambda, c_1,\theta_*,r_0,c_0,\alpha,\left[u\right]_{B,\alpha}}\in (0,1)$ and $c:=c_{m,\theta_*}\in (0,+\infty)$ such that, for any $r\in \left[R^*,+\infty\right)$ satisfying $B_{4r}\subset \Omega$,
\begin{equation}\label{benz.1}
c\int_{B_r\cap \left\lbrace \theta_{*}<u\leq \theta^* \right\rbrace}\left|1+u(x)\right|^m\,dx+\mathcal{L}^n(B_r\cap \left\lbrace u>\theta_2 \right\rbrace)>\tilde{c}r^n.
\end{equation}
\end{thm}

From Theorem~\ref{CapVieABall}, one can establish density estimates for minimizers of $\mathcal{E}_s^p$ that are stable as $s\to 1$ and do not require any additional assumption on the regularity of the minimizer. 

Indeed, it is well known that minimizers for $\mathcal{E}_s^p$ are H\"older continuous, see for instance~\cite{MR3593528,MR3861716,MR3630640} and the references therein. In particular in~\cite{MR3630640}, using fractional De Giorgi classes, M. Cozzi proved H\"older estimates for minimizers of $\mathcal{E}_s^p$ that do not depend on $s$ whenever $s_0\in \left(0,1\right)$, $s\in \left[s_0,1\right)$ and $n\geq 2$, see Theorem~6.4 and Proposition~7.5 in~\cite{MR3630640} or Theorem~2.2 in~\cite{DFVERPP}. Therefore, as a byproduct of Theorem~\ref{CapVieABall}, we obtain the stability of density estimates for minimizers of~$\mathcal{E}_s^p$ as $s\to 1$. The precise statement of this result goes as follows:

\begin{cor}[Stable density estimates for minimizers]\label{bmrlsadcbl}
Let $p\in (1,+\infty)$, $s_0\in \left(\frac{1}{p} ,1\right)$ and $s\in \left[s_0,1\right)$. Also, let $\theta_1,\theta_2\in (-1,\theta]$ and $\theta_*$, $\theta^*$ be defined as in~\eqref{liibfdcvre5555}. 

Moreover, we assume that $u\in X^{s,p}(\Omega)$ is a minimizer for $\mathcal{E}_s^p$ and that there exist~$c_0,r_0\in (0,+\infty)$ such that $B_{r_0}\subset\Omega$ and~\eqref{Fcon-812345} is satisfied. 

Then, there exist~$R^*:=R_{s_0,n,m,p,c_1,\theta_*,r_0}^*\in\left[r_0,+\infty\right)$ , $\tilde{c}:=\tilde{c}_{s_0,n,p,m,\Lambda,c_1,\theta_*,r_0,c_0}\in (0,1)$ and $c:=c_{m,\theta_*}\in(0,+\infty)$ such that, for any $r\in \left[R^*,+\infty\right)$ satisfying $B_{4r}\subset \Omega$,
\begin{equation}\label{lietoca}
c_{m,\theta_*}\int_{B_r\cap \left\lbrace \theta_{*}<u\leq \theta^* \right\rbrace}\left|1+u(x)\right|^m\,dx+\mathcal{L}^n(B_r\cap \left\lbrace u>\theta_2 \right\rbrace)>\tilde{c}r^n.
\end{equation}
\end{cor}

The presence of the first term in the left-hand side of~\eqref{benz},~\eqref{benz.1} and~\eqref{lietoca} is a consequence of the lower bound in~\eqref{potential}. To be more precise, under some additional assumptions on the lower bound for $W$ (see~\eqref{condiWzione} below) it is possible to reabsorb such term into the right-hand side of the inequality, obtaining the ``full density estimates''. This is made possible thanks to the following upper bound for the energy $\mathcal{E}_s^p$.  

\begin{thm}\label{qteanldeeo}
Let $s\in (0,1)$ and $u$ be a minimizer of $\mathcal{E}_s^p$ in $X^{s,p}(B_{R+2})$ with $R\geq 2$. Then, there exists $\bar{C}:=\bar{C}_{n,p,m,\Lambda}\in(0,+\infty)$ such that 
\begin{equation}\label{BFAOTE}
\mathcal{E}_s^p(u,B_R)\leq \frac{\bar{C}}{s}\begin{dcases}
\frac{R^{n-sp}}{1-sp}\quad &\mbox{if}\quad s\in \left(0,\frac{1}{p}\right)\\
R^{n-1}\log(R)\quad &\mbox{if}\quad s=\frac{1}{p}\\
\frac{R^{n-1}}{sp-1}\quad &\mbox{if}\quad s\in \left(\frac{1}{p},1\right).
\end{dcases}
\end{equation} 
\end{thm}

This result was originally proved in~\cite{MR3133422}
for $p=2$ and without the explicit dependence of the constant with respect to $s$.  

From Theorem~\ref{th:fracp>=2}, Corollary~\ref{bmrlsadcbl} and Theorem~\ref{qteanldeeo} we obtain the following density estimate: 

\begin{cor}\label{coro-09}
Let $\theta_1,\theta_2\in (-1,1)$, $\theta_*$, $\theta^*$ as in~\eqref{liibfdcvre5555} and $u\in X^{s,p}(\Omega)$ be a minimizer for $\mathcal{E}_s^p$.  Moreover,  let $c_0,r_0\in (0,+\infty)$ be such that $B_{r_0}\subset \Omega$ and~\eqref{Fcon-812345} is satisfied.   

Also, we assume that for every $\mu\in (-1,1)$ there exists~$\lambda_\mu\in (0,+\infty)$ such that, for every $x\in (-1,1)$,
\begin{equation}\label{condiWzione}
\lambda_\mu \chi_{\left(-\infty,\mu\right]}(x)\left|1+x\right|^m  \leq W(x).
\end{equation}
Then, there exist~$R^*:=R_{s,n,m,p,c_1,\theta_*,r_0}^*\in \left[r_0,+\infty\right)$, $\tilde{c}:=\tilde{c}_{s,n,p,m,\Lambda,c_1,\theta_*,r_0,c_0}\in (0,1)$ and $c:=c_{m,\theta_*}\in(0,+\infty)$ such that, for any $r\in \left[R^*,+\infty\right)$ satisfying $B_{4r}\subset \Omega$,
\begin{equation}\label{benzina}
\mathcal{L}^n(B_r\cap \left\lbrace u>\theta_2 \right\rbrace)>\tilde{c}r^n.
\end{equation}
Furthermore, if $s_0\in \left(\frac{1}{p},1\right)$ and $s\in \left[s_0,1\right)$, the quantities $R^{*}$ and~$\tilde{c}$ do not depend on~$s$, namely
\begin{equation}\label{cala}
R^{*}:= R_{s_0,n,m,p,c_1,\theta_*,r_0}^{*}\quad\mbox{and}\quad \tilde{c}:=\tilde{c}_{s_0,n,p,m,\Lambda, c_1,\theta_*,r_0,c_0}. 
\end{equation}
\end{cor}

\subsection{Local density estimates}\label{ere-optre}

In this section we state the main result concerning density estimates for minimizers of  $\mathcal{E}_1^p$. To do so, we first need to state the following $\Gamma$ -convergence result. In particular, we show that $\mathcal{E}_s^p$ $\Gamma$-converges to $\mathcal{E}_1^p$ as $s\to 1$. This result will be employed to prove density estimates for minimizers of $\mathcal{E}_1^p$ (see Theorem~\ref{viredghloprec45} below), starting from the density estimates for $\epsilon$-minimizers of $\mathcal{E}_s^p$ discussed  in Theorem~\ref{CapVieABall}. 

\begin{thm}[$\Gamma$-convergence of $\mathcal{E}_s^p$]\label{12c-dgbnbc5543}
Let $\left\lbrace s_k\right\rbrace_{k\in\N}\subset (0,1)$ such that $s_k\to 1$. Also, let $p\in (1,+\infty)$ and $\Omega\subset\R^n$ be open, bounded and Lipschitz. Then, we have the following:

\begin{itemize}
\item[(i)] Let $\left\lbrace u_{s_k} \right\rbrace_{k\in\N} \subset X^{s_k,p}(\Omega)$ and $M\in (0,+\infty)$ be such that
\begin{equation}\label{kkkkfvetrre}
\sup_{k\in\N} (1-s_k)\mathcal{K}_{s_k}^p(u_{s_k},\Omega) \leq M.
\end{equation}
Then, up to a subsequence, $u_{s_k}\to u$ in $L^p(\Omega)$ for some $u\in  W^{1,p}(\Omega)$. 

Moreover, if $u_{s_k}\in X_0^{s_k,p}(\Omega)$ for every $k$, we obtain that $u\in W_0^{1,p}(\Omega)$. 
\item[(ii)] For any $u\in L^p(\Omega)$ and $\left\lbrace u_{s_k} \right\rbrace_{k\in\N} \subset X^{s_k,p}(\Omega)$ such that $u_{s_k}\to u$ in $L^p(\Omega)$ it holds that
\begin{equation}\label{utrc5342}
\mathcal{E}_1^p(u,\Omega)\leq \liminf_{k\to +\infty} \mathcal{E}_{s_k}^p(u_{s_k},\Omega).
\end{equation}
\item[(iii)] For any $u\in L^p(\Omega)$ satisfying~$\left|u\right|\leq 1$ a.e. in $\Omega$ there exists~$\left\lbrace u_{s_k}\right\rbrace_{k\in\N} \subset X^{s_k,p}(\Omega)$ such that $u_{s_k}\to u$ in $L^p(\Omega)$ and
\begin{equation*}
\mathcal{E}_1^p(u,\Omega)\geq \limsup_{k\to +\infty} \mathcal{E}_{s_k}^p(u_{s_k},\Omega). 
\end{equation*}
Furthermore, if $u\in X^{1,p}(\Omega)$, there exists~$v\in X^{s,p}(\Omega)$ for every $s\in(0,1)$ such that $u\equiv v$ in $\Omega$ and 
\begin{equation}\label{glordvendalat}
\mathcal{E}_1^p(u,\Omega)=\limsup_{k\to +\infty} \mathcal{E}_{s_k}^p(v,\Omega). 
\end{equation}
\end{itemize} 
\end{thm}

As aforementioned, making use of Theorem~\ref{CapVieABall} and Theorem~\ref{12c-dgbnbc5543} we can establish density estimates for minimizers of $\mathcal{E}_1^p$. The statement of the result is the following:

\begin{thm}\label{viredghloprec45}
Let $\Omega\subset\R^n$ be open and bounded, $p\in(1,+\infty)$, $m\in[p,+\infty)$ and $\theta_1,\theta_2\in (-1, \theta]$. Also, we set $\theta_{*}$ and $\theta^*$ as in~\eqref{liibfdcvre5555}. Moreover, we assume that $u$ is a minimizer for $\mathcal{E}_1^p$ in $X^{1,p}(\Omega)$ and  for some $c_0,r_0\in (0,+\infty)$  satisfying $B_{r_0}\subset\Omega$ 
\begin{equation}\label{xldvaxknterv543}
\mathcal{L}^n(B_{r_0}\cap \left\lbrace u>\theta_1\right\rbrace)\geq c_0.
\end{equation} 
Then, there exist~$R^*:=R_{n,m,p,c_1,\theta_*,c_0}^*\in\left[r_0,+\infty\right)$, $\tilde{c}:=\tilde{c}_{n,p,m,\Lambda,c_1,\theta_*,r_0,c_0}\in(0,1)$ and $c_{m,\theta_*}\in(0,+\infty)$ such that,
for every $r\in \left[R^*,+\infty\right)$ for which $B_{4r}\subset\Omega$,
\begin{equation*}
c_{m,\theta_*}\int_{B_r\cap \left\lbrace \theta_{*}<u\leq \theta^* \right\rbrace}\left|1+u(x)\right|^m\,dx+\mathcal{L}^n(B_r\cap \left\lbrace u>\theta_2 \right\rbrace)>\tilde{c}r^n
\end{equation*}
\end{thm}
We recall that such a result was already proved in~\cite{1} for quasiminimizers (and therefore also for minimizers) under a suitable constraint on $n$, $p$ and $m$. More precisely, in Theorem~1.1 in~\cite{1}, it is stated that density estimates hold for quasiminimizers of~$\mathcal{E}_1^p$ if $n$, $p$ and $m$ satisfy 
\begin{equation}\label{scbdgetoygt6}
\frac{pm}{m-p}>n.
\end{equation} 
As stated in Remark~1.7 in~\cite{1}, it is not known whether the condition in~\eqref{scbdgetoygt6} is sharp or not. It is important to mention that it has been recently proved in~\cite{savin2025density} that with an additional assumption on the quasiminimizer the condition in~\eqref{scbdgetoygt6} can be removed, see Remark~\ref{removiult}. The novelty in Theorem~\ref{viredghloprec45} is to prove that~\eqref{scbdgetoygt6} is not sharp. Indeed, we show that density estimates hold for minimizers of $\mathcal{E}_1^p$ for every $n\in\N$, $p\in (1,+\infty)$ and $m\in \left[p,+\infty\right)$. 

\begin{rem}\label{removiult}
We notice that in the recent paper~\cite{savin2025density}, it has been showed that the condition in~\eqref{scbdgetoygt6} can be removed, under an additional assumption on the quasiminimizer. 

More precisely, given any $n\in\N$, $p\in (1,+\infty)$ and $m\in \left[p,+\infty\right)$, they establish that for every $\epsilon>0$ there exists some $\rho_0>0$ such that density estimates hold for a quasiminimizers $u$ of~$\mathcal{E}_1^p$  as far as 
\begin{equation}\label{bevcgdhya66}
\mathcal{L}^n\left( \left\lbrace u\geq 0 \right\rbrace    \cap B_{\rho} \right)\geq \epsilon \rho^n
\end{equation}
for some $\rho\geq \rho_0$, see equation~(1.5) in Theorem~1.1 in~\cite{savin2025density}. 

We observe that the condition in~\eqref{bevcgdhya66} is much stronger than the one in~\eqref{xldvaxknterv543} of Theorem~\ref{viredghloprec45}.

As a matter of fact, on one hand,  in Theorem~1.1 in~\cite{savin2025density}  they show that once the structural parameters (such as $n$, $p$ and $m$) are fixed, there exists some $\rho_0$ such that if~\eqref{bevcgdhya66} is satisfied then density estimates hold. 

On the other hand, in Theorem~\ref{viredghloprec45}, we prove that independently from the choice of $c_0,r_0\in (0,+\infty)$ in~\eqref{xldvaxknterv543}, as far as $B _{r_0}\subset\Omega$, density estimates hold for every minimizer.
\end{rem}

\medskip

The rest of the paper is structured as follows. In Section~\ref{7hltgb5454546} we prove Theorems~\ref{th:fracp>=2},~\ref{CapVieABall} and Corollaries~\ref{bmrlsadcbl},~\ref{coro-09}. These results are a consequence of a preliminary estimate, Lemma~\ref{gliorigamieriopachi} in  Section~\ref{gi-i-gio}, which holds true for every  $\epsilon$-minimizer of $\mathcal{E}_s^p$ with $\epsilon\in (0,+\infty)$ and satisfying~\eqref{Fcon-812345}. The proofs of Theorems~\ref{th:fracp>=2} and~\ref{CapVieABall} are contained in Section~\ref{kovbreclo865r}. Section~\ref{hyg-yh-654} is devoted to the proofs of  Corollaries~\ref{bmrlsadcbl} and~\ref{coro-09}. The proofs of Theorems~\ref{12c-dgbnbc5543} and~\ref{viredghloprec45} are provided in Section~\ref{nmmdqcstngiga}. Finally, in Section~\ref{motifenzi} we prove Theorem~\ref{qteanldeeo}.

\section{Proof of Theorems~\ref{th:fracp>=2}~\&~\ref{CapVieABall} and Corollaries~\ref{bmrlsadcbl}~\&~\ref{coro-09}}\label{7hltgb5454546}

In this section we prove Theorems~\ref{th:fracp>=2} \&~\ref{CapVieABall} and Corollaries~\ref{bmrlsadcbl} \&~\ref{coro-09}. These proofs are based on a preliminary estimate, which holds true for any $\epsilon$-minimizer $u\in X^{s,p}(\Omega)$  satisfying equation~\eqref{Fcon-812345}. Such an estimate is the content of Lemma~\ref{gliorigamieriopachi} in Section~\ref{gi-i-gio} below. In Section~\ref{kovbreclo865r}, we use the inequality stated in Lemma~\ref{gliorigamieriopachi} to prove the main results, namely Theorem~\ref{th:fracp>=2} and Theorem~\ref{CapVieABall}. Finally, in Section~\ref{hyg-yh-654} we provide the proofs of Corollaries~\ref{bmrlsadcbl} and~\ref{coro-09}.      

\begin{rem}\label{plothgnebvdf} {\rm
We notice that in~\eqref{Fcon-812345} we can assume without loss of generality that $c_0\in (1,+\infty)$. More precisely, we show that if the results in either Theorem~\ref{th:fracp>=2} or Theorem~\ref{CapVieABall} hold true with $c_0\in (1,+\infty)$, then they hold true also with $c_0\in (0,1]$. 

Therefore, in what follows we assume that
\begin{equation}\label{kilimoner}\begin{split}&
{\mbox{Theorem~\ref{th:fracp>=2} and Theorem~\ref{CapVieABall} hold true}}\\&{\mbox{under the additional assumption that $c_0\in (1,+\infty)$.}} \end{split}
\end{equation}
Then, let $c_0\in (0,1]$
(and note that we are now not entitled in principle to use
Theorem~\ref{th:fracp>=2} and Theorem~\ref{CapVieABall} for this~$c_0$). Let also~$\tilde{c}_0\in (1,+\infty)$  and define $r:=\left(\frac{c_0}{\tilde{c}_0}\right)^\frac{1}{n}$. 

Also, if $f\in X^{s,p}(\Omega)$ we consider its rescaling $f_r\in X^{s,p}\left(\Omega_\frac{1}{r}\right)$ given by
\begin{equation*}
f_r(x):=f(r x).
\end{equation*}
Thus, let $u\in X^{s,p}(\Omega)$ satisfy~\eqref{Fcon-812345} and define $\tilde{r}_0:=\frac{r_0}{r}$.

In this way,
\begin{equation}\label{efvecdere}
\begin{split}
\mathcal{L}^n\left(B_{\tilde{r}_0} \cap \left\lbrace u_r>\theta_1 \right\rbrace  \right)&=\frac{1}{r^n} \mathcal{L}^n\left(B_{r_0}\cap \left\lbrace u>\theta_1 \right\rbrace\right)\\
&\geq \frac{c_0}{r^n}\\
&=\tilde{c}_0.
\end{split}
\end{equation} 
Also, for all $\epsilon\in (0,+\infty)$,
\begin{equation*}
\tilde{W}:=r^{sp}W,\quad \tilde{\epsilon}:=\epsilon r^{sp-n}\quad \mbox{and}\quad \tilde{\mathcal{E}}_s^p\left(\cdot,\Omega_\frac{1}{r}\right):=(1-s)\mathcal{K}_s^p\left(\cdot,\Omega_\frac{1}{r}\right)+\left\|\tilde{W}\left(\cdot\right)\right\|_{L^1\left(\Omega_\frac{1}{r}\right)}.
\end{equation*}
If~$u\in X^{s,p}(\Omega)$ is an $\epsilon$-minimizer, then, for every $w_r\in X_{u_r}^{s,p}\left(\Omega_\frac{1}{r}\right)$,
\begin{equation}\label{lower-case}
\begin{split}
\tilde{\mathcal{E}}_s^p(w_r,\Omega)+\tilde{\epsilon}&=r^{sp-n} \left(\mathcal{E}_s^p(w,\Omega)+\epsilon \right)\\
&\geq r^{sp-n}\mathcal{E}_s^p(u,\Omega)\\
&=\tilde{\mathcal{E}}_s^p\left(u_r,\Omega_\frac{1}{r}\right).
\end{split}
\end{equation} 
In particular, we obtain that $u_r$ is a $\tilde{\epsilon}$-minimizer for $\tilde{\mathcal{E}}_s^p\left(\cdot,\Omega_\frac{1}{r}\right)$. 

Then, as a consequence of~\eqref{kilimoner},~\eqref{efvecdere},~\eqref{lower-case} and Theorem~\ref{th:fracp>=2} we obtain that there exists~$\delta:=\delta_{s,n,m,p,c_1,\theta_*,\tilde{c}_0}\in (0,+\infty)$ such that if $\tilde{\epsilon}\in \left[0,\delta\right)$ then
there exist~$R^{*}:=R_{s,n,m,p,r^{sp}c_1,\theta_*,\tilde{r}_0}^{*}\in \left[\tilde{r}_0,+\infty\right)$, $\tilde{c}:=\tilde{c}_{s,n,p,m,r^{sp}\Lambda,r^{sp}c_1,\theta_*,\tilde{r}_0,\tilde{c}_0}\in (0,1)$ and $c:=c_{m,\theta_*}\in(0,+\infty)$ such that, for every $t\in \left[R^*,+\infty\right)$ satisfying $B_{3t}\subset \Omega_\frac{1}{r}$,
\begin{equation}\label{redi-opredf}
\begin{split}
\tilde{c}t^n &<\mathcal{L}^n(B_t\cap \left\lbrace u_r>\theta_2 \right\rbrace)+c_{m,\theta_*}\int_{B_t\cap \left\lbrace \theta_{*}< u_r\leq \theta^* \right\rbrace}\left|1+u_r(x)\right|^m\,dx\\
\end{split}
\end{equation} 
Now, we define 
\begin{equation*}
\delta^{(1)}:= r^{n-sp}\delta \quad\mbox{and}\quad  R^{(1)}:=r R^{*}
\end{equation*}
Hence, by scaling, we obtain from~\eqref{redi-opredf} that if $\epsilon\in \left[0,\delta^{(1)}\right)$, then, for every $t\in \left[R^{(1)},+\infty\right)$ such that $B_{3t}\subset \Omega$,
\begin{equation*}
\mathcal{L}^n\left( B_{t}\cap \left\lbrace u>\theta_2 \right\rbrace\right)+ c_{m,\theta_*}\int_{B_{t}\cap \left\lbrace \theta_{*}< u \leq \theta^* \right\rbrace}\left|1+u(x)\right|^m\,dx \geq \tilde{c} t^n. 
\end{equation*}
Therefore, we have showed that if Theorem~\ref{th:fracp>=2} holds true with $c_0\in (1,+\infty)$, then it holds true with $c_0\in (0,1]$.

Now, we  assume that $s_0\in \left(\frac{1}{p},1\right)$, $s\in\left[s_0,1\right)$ and also $u\in X^{s,p}(\Omega)\cap C_B^\alpha(\Omega)$. 

Then, as a consequence of~\eqref{kilimoner},~\eqref{efvecdere},~\eqref{lower-case} and Theorem~\ref{CapVieABall} we obtain that there exists $\delta:=\delta_{s_0,n,p,m,c_1,\theta_*,\tilde{c}_0,\alpha,[u]_{B,\alpha}}\in (0,+\infty)$ such that, if $\tilde{\epsilon}\in \left[0,\delta\right)$, then
there exist~$R^*:=R_{s_0,n,p,m,c_1,\theta_*,\tilde{r}_0}^*\in \left[\tilde{r}_0,+\infty\right)$, $\tilde{c}:=\tilde{c}_{s_0,n,p,m,\Lambda,c_1,\theta_*,\tilde{r}_0,\tilde{c}_0,\alpha,[u]_{B,\alpha}}\in (0,1)$ and $c:=c_{m,\theta_*}\in (0,+\infty)$ such that, for any $t\in \left[R^*,+\infty\right)$ satisfying $B_{4t}\subset \Omega_{\frac{1}{r}}$,
\begin{equation}\label{opredf-lorevgt}
\begin{split}
\tilde{c}t^n &<\mathcal{L}^n(B_t\cap \left\lbrace u_r>\theta_2 \right\rbrace)+c_{m,\theta_*}\int_{B_t\cap \left\lbrace \theta_{*}< u_r\leq \theta^* \right\rbrace}\left|1+u_r(x)\right|^m\,dx.
\end{split}
\end{equation} 
Now we define 
\begin{equation*}
\delta^{(2)}:=\delta \min_{s\in \left[s_0,1\right]}r^{n-sp}\quad\mbox{and}\quad R^{(2)}:=r R^*. 
\end{equation*}
Therefore, by scaling, we obtain from~\eqref{opredf-lorevgt} that if $\epsilon\in \left[0,\delta^{(2)}\right)$, then, for every $t\in \left[R^{(2)},+\infty\right)$ such that $B_{4t}\subset \Omega$,
\begin{equation*}
\mathcal{L}^n\left( B_{t}\cap \left\lbrace u>\theta_2 \right\rbrace\right)+ c_{m,\theta_*}\int_{B_{t}\cap \left\lbrace \theta_{*}< u \leq \theta^* \right\rbrace}\left|1+u(x)\right|^m\,dx \geq \tilde{c} t^n. 
\end{equation*}
Hence, we have showed that if Theorem~\ref{CapVieABall} holds true with $c_0\in (1,+\infty)$, then it holds true with $c_0\in (0,1]$.
}
\end{rem}

\subsection{A preliminary estimate}\label{gi-i-gio}
In this section we state and prove a preliminary estimate for $\epsilon$-minimizers $u\in X^{s,p}(\Omega)$. This result will be employed in the proof of both Theorems~\ref{th:fracp>=2} and~\ref{CapVieABall}. Before stating and commenting such a result, we first need to set some notation.  

First, given any $p\in (1,+\infty)$ and $s\in (0,1)$ we let
\begin{equation}\label{holeo}
\bar{p}:=\max\left\lbrace 2,p \right\rbrace,\quad l:=\min\left\lbrace 2,p \right\rbrace\quad\mbox{and}\quad \bar{s}:=\frac{sl}{2}.
\end{equation}
Moreover, for every $\tau\in (0,+\infty)$ we recall the quantities 
\begin{equation}\label{latuse-invefr43}
C_\tau:=C_{s,\tau,n,p,m}\quad\mbox{and}\quad \bar{R}_{\tau}:=\bar{R}_{s,\tau,n,p,m}
\end{equation} 
defined respectively in~(3.4) and~(3.1) of~\cite{DFVERPP}. Also, from now on we introduce the quantity 
\begin{equation}\label{muuuuuuu}
\mu:=\frac{c_1}{(1-s)p}
\end{equation} 
where $c_1\in (0,+\infty)$ is the structural constant in~\eqref{potntiald2}, and we choose the barrier $w\in C((-1,1],\R^n)$ as in Theorem~3.1 in~\cite{DFVERPP} with $\tau=\mu$. Furthermore, we define the constant 
\begin{equation*}
K_0:=\left[C_\mu\left(\frac{2}{1+\theta_*}\right)\right]^{\frac{m-1}{ps}}-1, 
\end{equation*} 
and we notice that, according to~(3.4) in~\cite{DFVERPP},
\begin{equation}\label{KpgrdiR}
K_0\geq C_\mu^\frac{m-1}{ps}\geq \max\left\lbrace 1, \bar{R}_\mu^{\frac{ps}{m-1}}\right\rbrace^\frac{m-1}{ps}\geq \bar{R}_\mu.
\end{equation}
For later convenience, we also observe that, thanks to~(3.5) in Theorem~3.1 of~\cite{DFVERPP}, for all~$R\in(K_0,+\infty)$, all~$K\in (K_0,R)$ and all $x\in B_{R-K}$,
\begin{equation}\label{ukulele}
w(x)\leq -1+C_\mu(K+1)^{-\frac{ps}{m-1}}<-1+\frac{1+\theta_*}{2}.
\end{equation} 
Now, given $u\in X^{s,p}(\Omega)$ and $R\in \left[\bar{R}_\mu,+\infty\right)$ such that $B_R\subset \Omega$, we define 
\begin{equation}\label{vdefini-tion}
v:=\min\left\lbrace u,w \right\rbrace.
\end{equation} 
We notice that, as a consequence of~(3.2) in~\cite{DFVERPP}, 
\begin{equation}\label{eq:boundcond}
v\equiv u \quad\mbox{in}\quad \R^n \setminus  B_R.
\end{equation}
Then, we define
\begin{equation}\label{adierre}
\begin{split}
a_R&:=\left\lbrace u-v\geq \frac{1+\theta_*}{4} \right\rbrace\quad\mbox{and}\quad b_R:= \left\lbrace  \frac{1+\theta_*}{8}<u-v<\frac{1+\theta_*}{4}\right\rbrace
\end{split}
\end{equation}
and, using~\eqref{ukulele} and~\eqref{adierre}, we conclude that, for every $R\in \left(K_0,+\infty\right)$ and $K\in (K_0,R)$,
\begin{equation}\label{takanaka}
a_R \supset B_{R-K}\cap \left\lbrace u>\theta_{*} \right\rbrace.
\end{equation}
Also, according to equation~\eqref{eq:boundcond}, the function~$u-v$ has compact support. In particular, the symmetric decreasing rearrangement of $u-v$, denoted by $(u-v)^*$, is well defined. Hence, we can define
\begin{equation}\label{astarebstar}
\begin{split}
a_R^*:= \left\lbrace (u-v)^*\geq \frac{1+\theta_*}{4} \right\rbrace, & \quad b_R^*:= \left\lbrace  \frac{1+\theta_*}{8}<(u-v)^*<\frac{1+\theta_*}{4}\right\rbrace\\
\mbox{and}\quad d_R^*&:=\R^n\setminus \left(a_R^*\cup b_R^*\right).
\end{split}
\end{equation}
Now we state and prove the following facts regarding $a_R^*$, $b_R^*$ and $d_R^*$. 

\begin{prop}\label{jamio-le}
Let $\theta_1,\theta_2\in (-1,\theta]$ and $\theta_*$, $\theta^*$ be defined as in~\eqref{liibfdcvre5555}. Moreover, we let $u\in X^{s,p}(\Omega)$ and $c_0,r_0\in (0,+\infty)$ such that $B_{r_0}\subset \Omega$ and~\eqref{Fcon-812345} is satisfied. 

Then, for every $R\in \left(K_0+r_0,+\infty\right)$,
\begin{equation}\label{iann-uytb}
\left|a_R^*\right|\geq c_0
\end{equation}
and there exist~$r_2\in (0,+\infty)$ and $r_1\in [0,+\infty)$ such that
\begin{equation}\label{levelsetpalle4563}
a_R^*=B_{r_2}, \quad b_R^*=B_{r_1+r_2}\setminus B_{r_2}\quad\mbox{and}\quad d_R^*=\R^n\setminus B_{r_2+r_1}.
\end{equation}
\end{prop}

\begin{proof}
First, from~\eqref{Fcon-812345} and~\eqref{takanaka} we infer that, for every~$R\in\left(K_0+r_0,+\infty\right)$ and $K\in (K_0,R-r_0)$,
\begin{equation*}
\begin{split}
\left|a_R^*\right| & =\left|a_R\right|\\ 
&\geq \left|B_{R-K}\cap \left\lbrace u>\theta_{*} \right\rbrace\right|\\
&\geq \left|B_{R-K}\cap \left\lbrace u>\theta_1 \right\rbrace\right|\\
&\geq \left|B_{r_0}\cap \left\lbrace u>\theta_1 \right\rbrace\right|\\
&\geq c_0. 
\end{split}
\end{equation*}
This concludes the proof of~\eqref{iann-uytb}. 

Now we recall that the superlevel sets of the symmetric decreasing rearrangement of a function are balls in $\R^n$. Thus, there exist~$r_2,r_1\in [0,+\infty)$ such that 
\begin{equation*}
a_R^*=B_{r_2}, \quad b_R^*=B_{r_1+r_2}\setminus B_{r_2}\quad\mbox{and}\quad d_R^*=\R^n\setminus B_{r_2+r_1}.
\end{equation*}
Moreover, thanks to~\eqref{iann-uytb} we obtain that $r_2> \left(\omega_n c_0\right)^\frac{1}{n}$, which completes the proof of~\eqref{levelsetpalle4563}. 
\end{proof}

We now set
\begin{equation}\label{red-g-poi}
\hat{c}_p:=\begin{dcases}
2^{1-p}\quad &\mbox{if}\quad p\in \left[2,+\infty\right)\\
\frac{3p(p-1)}{4^{4-p}}\quad &\mbox{if}\quad p\in (1,2). 
\end{dcases}
\end{equation} 
Also, given any $K\in (0,+\infty)$, we define the following constants 
\begin{equation}\label{c2-34566}
\begin{split}
&c_2:=c_1\left(\frac{1+\theta_*}{8}\right)^k, \quad\quad c_3:=c_2(1+K)^p\\
&\bar{c}_4:= \frac{1}{2} \left(\frac{1+\theta_*}{8}\right)^{\bar{p}}\quad\mbox{and}\quad c_5:=C_{\mu}^{m-1} \frac{2c_1+\Lambda}{\hat{c}_p}+\frac{c_3}{\hat{c}_p},  
\end{split}
\end{equation}
where $k$ has been introduced in~\eqref{key?}. 

Moreover, for every $t\in (0,+\infty)$ such that $B_t\subset \Omega$, we let 
\begin{equation}\label{rfdlopcbgft5}
V(t):=\mathcal{L}^n(B_t\cap \left\lbrace u>\theta_* \right\rbrace). 
\end{equation}
Finally, given any measurable subsets $A,D\subset \R^n$ we set 
\begin{equation}\label{ellead}
L(A,D):=\int_{A}\int_{D}\frac{dx\,dy}{\left|x-y\right|^{n+sp}}. 
\end{equation}

With this, we are ready to state the main result of this section: 
\begin{lem}\label{gliorigamieriopachi}
Let $\epsilon\in \left[0,+\infty\right)$, $\theta_1,\theta_2\in (-1,\theta]$ and $\theta_*,\theta^*$ be defined as in~\eqref{liibfdcvre5555}. Also, let $u\in X^{s,p}(\Omega)$ be an $\epsilon$-minimizer for $\mathcal{E}_s^p$ and $c_0,r_0\in (0,+\infty)$ be such that $B_{r_0}\subset \Omega$ and~\eqref{Fcon-812345} is satisfied.  

Then, for every $R\in \left(K_0,+\infty\right)$ such that $B_R\subset\Omega$ and $K\in (K_0,R)$,
\begin{equation}\label{iattoll-90843}
\begin{split}
&(1-s)\bar{c}_4 L(a_R^*,d_R^*) +\frac{c_2}{\hat{c}_p}\left|b_R^*\right| \leq \frac{\epsilon }{\hat{c}_p}+c_5 \int_{0}^R \left(R-t+1\right)^{-sp}V'(t)\,dt.    
\end{split}
\end{equation}
\end{lem}

As aforementioned at the beginning of Section~\ref{7hltgb5454546}, Lemma~\ref{gliorigamieriopachi} is used to prove both Theorem~\ref{th:fracp>=2} and~\ref{CapVieABall}. Indeed, we observe that any $\epsilon$-minimizer satisfying the hypothesis of either Theorem~\ref{th:fracp>=2} or~\ref{CapVieABall} satisfies automatically all the conditions required to apply Lemma~\ref{gliorigamieriopachi}. 

 \begin{proof}[Proof of Lemma~\ref{gliorigamieriopachi}]
For every $R\in \left[\bar{R}_\mu,+\infty\right)$ such that $B_R\subset \Omega$, we set
\begin{equation}\label{akkappaerre}
A_k(R):=c_1\int_{B_R\cap \left\lbrace w<u\leq \theta_* \right\rbrace}(u-w)^k\,dx.
\end{equation}
We claim that, for every $R\in \left[\bar{R}_\mu,+\infty\right)$ satisfying $B_R\subset\Omega$,
\begin{equation}\label{Ineq.ecm}
\hat{c}_p(1-s) \mathcal{K}_{\bar{s}}^{\bar{p}}(u-v, B_R) \leq \epsilon + C_\mu^{m-1}(2c_1+\Lambda)\int_{B_R\cap \left\lbrace u> \theta_* \right\rbrace } \left(R-\left|x\right|+1\right)^{-sp} \,dx-A_k(R).
\end{equation}

In order to show~\eqref{Ineq.ecm}, we recall equation~(3.8) in~\cite{DFVERPP}, according to which, for every $R\in \left[\bar{R}_\mu,+\infty\right)$ such that $B_R\subset \Omega$,
\begin{equation*}
\begin{split}
(1-s)\hat{c}_p \mathcal{K}_{\bar{s}}^{\bar{p}}(u-v, B_R)&\\
\leq (1-s)\mathcal{K}_s^p &(u,B_R)-(1-s)\mathcal{K}_{s}^p(v,B_R)-(1-s)p\int_{B_R\cap \left\lbrace u>w\right\rbrace}(u(x)-v(x))(-\Delta)_p^s w(x)\,dx.
\end{split}
\end{equation*}
From this, we deduce that
\begin{equation*}
\begin{split}&\!\!\!\!\!\!\!\!\!\!\!\!
(1-s)\hat{c}_p \mathcal{K}_{\bar{s}}^{\bar{p}}(u-v, B_R) \\ \leq \,&(1-s) \mathcal{K}_s^p(u,B_R)-(1-s)\mathcal{K}_{s}^p(v,B_R) +\int_{B_R} W(u)-W(v)\,dx+\int_{B_R} W(v)-W(u)\,dx\\
&\qquad  -(1-s)p\int_{B_R\cap \left\lbrace u>w\right\rbrace}(u(x)-v(x))(-\Delta)_p^s w(x)\,dx\\
= \,&\mathcal{E}_s^p(u,B_R)-\mathcal{E}_s^p(v,B_R)+\int_{B_R} W(v)-W(u)\,dx\\
& \qquad -(1-s)p\int_{B_R\cap \left\lbrace u>w\right\rbrace}(u(x)-v(x))(-\Delta)_p^s w(x)\,dx.
\end{split}
\end{equation*}

Consequently, recalling the $\epsilon$-minimality of $u$ and~\eqref{eq:boundcond} we gather that 
\begin{equation}\label{kerat.9034}
\begin{split}&\!\!\!\!\!\!
(1-s)\hat{c}_p \mathcal{K}_{\bar{s}}^{\bar{p}}(u-v, B_R) \\ &\leq \epsilon+\int_{B_R} W(v)-W(u)\,dx-(1-s)p\int_{B_R\cap \left\lbrace u>w\right\rbrace}(u(x)-v(x))(-\Delta)_p^s w(x)\,dx\\
&=\epsilon+\int_{B_R \cap \left\lbrace u>w \right\rbrace} W(w)-W(u)\,dx-(1-s)p\int_{B_R\cap \left\lbrace u>w\right\rbrace}(u(x)-v(x))(-\Delta)_p^s w(x)\,dx.
\end{split}
\end{equation}
Now, using~\eqref{potential} and~\eqref{potntiald2} we obtain that, for $k\in\N$ as in~\eqref{key?},
\begin{equation}\label{gdfwe}
\begin{split}
& \int_{B_R\cap \left\lbrace u>w \right\rbrace} W(w)-W(u)\,dx\\
=& \int_{B_R\cap \left\lbrace w<u\leq \theta_* \right\rbrace} W(w)-W(u)\,dx+\int_{B_R\cap \left\lbrace u>\max\left\lbrace w,\theta_*\right\rbrace \right\rbrace} W(w)-W(u)\,dx\\
\leq & -c_1\int_{B_R\cap \left\lbrace w<u\leq \theta_* \right\rbrace}\left|1+w\right|^{m-1}(u-w)\,dx -c_1\int_{B_R\cap \left\lbrace w<u\leq \theta_* \right\rbrace}(u-w)^k\,dx \\  
&\qquad+ \Lambda \int_{B_R\cap \left\lbrace u>\max\left\lbrace \theta_*,w \right\rbrace   \right\rbrace} \left|1+w\right|^m \,dx.
\end{split}
\end{equation}
Also, we infer from~\eqref{kerat.9034} and~\eqref{gdfwe} that, for all $R\in\left[\bar{R}_\mu,+\infty \right)$ such that $B_R\subset \Omega$,
\begin{equation*}
\begin{split}
&\hat{c}_p(1-s)\mathcal{K}_{\bar{s}}^{\bar{p}}(u-v, B_R)\\ 
\leq & \epsilon-c_1\int_{B_R\cap \left\lbrace w<u\leq \theta_* \right\rbrace}\left|1+w\right|^{m-1}(u-w)\,dx+ \Lambda \int_{B_R\cap \left\lbrace u>\max\left\lbrace \theta_*,w \right\rbrace   \right\rbrace} \left|1+w\right|^{m-1} \,dx\\
&-(1-s)p\int_{B_R\cap \left\lbrace u>w\right\rbrace}(u(x)-v(x))(-\Delta)_p^s w(x)\,dx-c_1\int_{B_R\cap \left\lbrace w<u\leq \theta_* \right\rbrace}(u-w)^k\,dx.
\end{split}
\end{equation*}
Making use of this inequality,~\eqref{akkappaerre} here and~(3.3) in~\cite{DFVERPP}, we obtain that
\begin{equation*}
\begin{split}
&\hat{c}_p(1-s)\mathcal{K}_{\bar{s}}^{\bar{p}}(u-v,B_R)\\
\leq &\epsilon -c_1\int_{B_R\cap \left\lbrace w<u\leq \theta_* \right\rbrace}\left|1+w\right|^{m-1}(u-w)\,dx+ \Lambda \int_{B_R\cap \left\lbrace u>\max\left\lbrace \theta_*,w \right\rbrace   \right\rbrace} \left|1+w\right|^{m-1} \,dx\\
&+c_1\int_{B_R\cap \left\lbrace w<u\leq \theta_*\right\rbrace}(u(x)-v(x))\left|1+w\right|^{m-1}\,dx+c_1\int_{B_R\cap \left\lbrace u>\max\left\lbrace \theta_*,w \right\rbrace   \right\rbrace}(u(x)-v(x))\left|1+w\right|^{m-1}\,dx\\
&-c_1\int_{B_R\cap \left\lbrace w<u\leq \theta_* \right\rbrace}(u-w)^k\,dx\\
\leq & \epsilon+\left(2c_1+\Lambda\right)\int_{B_R\cap \left\lbrace u>\max\left\lbrace \theta_*,w \right\rbrace   \right\rbrace} \left|1+w\right|^{m-1} \,dx-c_1\int_{B_R\cap \left\lbrace w<u\leq \theta_* \right\rbrace}(u-w)^k\,dx\\
= &\epsilon+\left(2c_1+\Lambda\right)\int_{B_R\cap \left\lbrace u>\max\left\lbrace \theta_*,w \right\rbrace   \right\rbrace} \left|1+w\right|^{m-1} \,dx-A_k(R).
\end{split}
\end{equation*}
Finally, thanks to~(3.5) in~\cite{DFVERPP}, 
\begin{equation*}
\hat{c}_p(1-s)\mathcal{K}_{\bar{s}}^{\bar{p}}(u-v,B_R) \leq \epsilon+ C_\mu^{m-1}(2c_1+\Lambda)\int_{B_R\cap \left\lbrace u> \theta_*  \right\rbrace} \left(R-\left|x\right|+1\right)^{-sp} \,dx-A_k(R).
\end{equation*}
This concludes the proof of claim~\eqref{Ineq.ecm}. 

From~\eqref{Ineq.ecm} and the coarea formula, we infer that, for every $R\in \left[\bar{R}_\mu,+\infty\right)$ such that $B_R\subset\Omega$,
\begin{equation}\label{Slimshady}
\begin{split}
\frac{A_k(R)}{\hat{c}_p} &+(1-s) \mathcal{K}_{\bar{s}}^{\bar{p}}(u-v,B_R) \\
&\leq \frac{\epsilon}{\hat{c}_p}+ C_\mu^{m-1}\frac{2c_1+\Lambda}{\hat{c}_p}\int_{0}^R \left(R-\left|x\right|+1\right)^{-sp}\int_{\partial B_t} \chi_{\left\lbrace u>\theta_* \right\rbrace}(x)\,dH_x^{n-1}\,dt\\
&=\frac{\epsilon}{\hat{c}_p}+C_\mu^{m-1}\frac{2c_1+\Lambda}{\hat{c}_p}\int_{0}^R \left(R-\left|x\right|+1\right)^{-sp}V'(t)\,dt.
\end{split}
\end{equation} 
Also, from~(3.1) in~\cite{DFVERPP} and~\eqref{liibfdcvre5555} here,
\begin{equation}\label{suK0tntddr}
K_0\geq \left[2^{\frac{p}{m-1}+1}\right]^\frac{m-1}{ps}-1\geq 2^\frac{1}{s}-1\geq 1.
\end{equation}

We now claim that, for every $R\in(K_0,+\infty)$ and $K\in (K_0,R)$,
\begin{equation}\label{Stimistopot}
A_k(R) \geq c_2\left|b_R^{*} \right|-c_3\int_{0}^R(R+1-t)^{-ps}V'(t)\,dt,
\end{equation}
and, for every $R\in \left[\bar{R}_\mu,+\infty\right)$,
\begin{equation}\label{Stimistocine}
\begin{split}
2\mathcal{K}_{\bar{s}}^{\bar{p}}\left(u-v,B_R\right) \geq \left(\frac{1+\theta_*}{8}\right)^{\bar{p}}L(a_R^*,d_R^*).
\end{split}
\end{equation}
We observe that if claims~\eqref{Stimistopot} and~\eqref{Stimistocine} hold true, then  from~\eqref{Stimistopot},~\eqref{Stimistocine},~\eqref{Slimshady} and~\eqref{KpgrdiR} it follows that, for every $R\in \left(K_0,+\infty\right)$ such that $B_R\subset\Omega$ and $K\in (K_0,R)$,
\begin{equation*}
\begin{split}
&\frac{(1-s)}{2}\left(\frac{1+\theta_*}{8}\right)^{\bar{p}} L(a_R^*,d_R^*) +\frac{c_2}{\hat{c}_p}\left|b_R^*\right|-\frac{c_3}{\hat{c}_p}\int_{0}^R (R+1-t)^{-sp} V'(t)\,dt\\ 
\leq & (1-s) \mathcal{K}_{\bar{s}}^{\bar{p}}(u-v,B_R) +\frac{A_k(R)}{\hat{c}_p}\\
\leq &\frac{\epsilon}{\hat{c}_p}+C_\mu^{m-1} \frac{2c_1+\Lambda}{\hat{c}_p} \int_{0}^R \left(R-t+1\right)^{-sp}V'(t)\,dt,    
\end{split}
\end{equation*}
which shows~\eqref{iattoll-90843}. Hence, in order to conclude the proof of Lemma~\ref{gliorigamieriopachi} it is only left to show claims~\eqref{Stimistopot} and~\eqref{Stimistocine}. 

To prove~\eqref{Stimistopot} we first observe that
\begin{equation*}
b_R\cap\left\lbrace u\leq \theta_* \right\rbrace \subset B_R\cap \left\lbrace w<u\leq \theta_*\right\rbrace
\end{equation*}
and therefore
\begin{equation}\label{hntecopeeewd}
\begin{split}
A_{k}(R)=c_1\int_{B_R\cap \left\lbrace w<u\leq \theta_*\right\rbrace} (u-w)^k\,dx &\geq c_1\int_{b_R\cap \left\lbrace u\leq \theta_* \right\rbrace}(u-w)^k\,dx\\
&\geq c_1\left(\frac{1+\theta_*}{8}\right)^k\left|b_R\cap \left\lbrace u\leq \theta_*\right\rbrace\right|. 
\end{split}
\end{equation}

In addition, we  notice that 
\begin{equation}\label{DSP-YHNRE45}
\begin{split}
V(R)-V(R-K)&=\int_{R-K}^R V'(t)\,dt\\
&\leq (1+K)^{ps}\int_{R-K}^R \left(R+1-t\right)^{-ps}V'(t)\,dt\\
&\leq (1+K)^{ps}\int_{0}^R \left(R+1-t\right)^{-ps}V'(t)\,dt.
\end{split}
\end{equation} 
Also, from~\eqref{ukulele} we evince that
\begin{equation}\label{glio-7hvcdr4}
b_{R}\cap B_{R-K} \subset B_{R-K}\cap \left\lbrace u\leq \theta_* \right\rbrace.
\end{equation}
From~\eqref{DSP-YHNRE45} and~\eqref{glio-7hvcdr4} it follows that 
\begin{equation}\label{xucvedlgbvcr51}
\begin{split}
\left|b_R\right|& = \left|b_R\cap \left\lbrace u\leq \theta_* \right\rbrace \right|+\left|b_R\cap \left\lbrace u>\theta_*\right\rbrace\right|\\
&=\left|b_R\cap \left\lbrace u\leq \theta_* \right\rbrace \right|+\left|\left(b_R\setminus  B_{R-K}\right)  \cap \left\lbrace u>\theta_*\right\rbrace\right|\\
&\leq \left|b_R\cap \left\lbrace u\leq \theta_* \right\rbrace \right|+\left|\left(B_R\setminus B_{R-K}\right)  \cap \left\lbrace u>\theta_*\right\rbrace\right|\\
&= \left|b_R\cap \left\lbrace u\leq \theta_* \right\rbrace \right|+ V(R)-V(R-K)\\
&\leq \left|b_R\cap \left\lbrace u\leq \theta_* \right\rbrace \right|+ (1+K)^{ps}\int_{0}^R \left(R+1-t\right)^{-ps}V'(t)\,dt.
\end{split}
\end{equation}
Moreover, from~\eqref{hntecopeeewd} and~\eqref{xucvedlgbvcr51} we deduce that 
\begin{equation}\label{perthgtbgf}
\begin{split}
A_k(R)&\geq c_1\left(\frac{1+\theta_*}{8}\right)^k \left( \left|b_R\right|-(1+K)^{ps}\int_{0}^R(R+1-t)^{-ps}V'(t)\,dt\right)\\ 
&=c_2\left|b_R \right|-c_3\int_{0}^R(R+1-t)^{-ps}V'(t)\,dt\\
&=c_2\left|b_R^*\right|-c_3\int_{0}^R (R+1-t)^{-ps}V'(t)\,dt. 
\end{split}
\end{equation}
This concludes the proof of claim~\eqref{Stimistopot}. 

It is only left to prove~\eqref{Stimistocine}. To show this, we observe that if $x\in a_R^*$ and $y\in d_R^*$ 
\begin{equation}\label{stacca-09}
\begin{split}
\left|(u-v)^*(x)-(u-v)^*(y)\right| &\geq (u-v)^*(x)-(u-v)^*(y)\\
&\geq \frac{1+\theta_*}{4}-\frac{1+\theta_*}{8}\\
&=\frac{1+\theta_*}{8}. 
\end{split}
\end{equation}
Then, making use of Theorem~9.2 in~\cite{MR1002633} and~\eqref{stacca-09} here, we obtain that
\begin{equation*}\label{EsOnThKiTe}
\begin{split}
2\mathcal{K}_{\bar{s}}^{\bar{p}}\left(u-v,B_R\right) & =\int_{\R^n}\int_{\R^n} \frac{\left|(u-v)(x)-(u-v)(y)\right|^{\bar{p}}}{\left|x-y\right|^{n+\bar{s}\bar{p}}}\,dy\,dx\\
& \geq \int_{\R^n}\int_{\R^n} \frac{\left|(u-v)^*(x)-(u-v)^*(y)\right|^{\bar{p}}}{\left|x-y\right|^{n+\bar{s}\bar{p}}}\,dy\,dx\\
&\geq \left(\frac{1+\theta_*}{8}\right)^{\bar{p}}\int_{a_R^*}\int_{d_R^*} \frac{dy\,dx}{\left|x-y\right|^{n+sp}}\\ 
&=\left(\frac{1+\theta_*}{8}\right)^{\bar{p}}L(a_R^*,d_R^*). 
\end{split}
\end{equation*}
This concludes the proof of~\eqref{Stimistocine}. 
\end{proof}

\subsection{Proof of Theorem~\ref{th:fracp>=2} \&~\ref{CapVieABall}}\label{kovbreclo865r}
In this section we prove Theorems~\ref{th:fracp>=2} and~\ref{CapVieABall}.

The first step  is to suitably estimate from below the left-hand side of~\eqref{iattoll-90843}. To do so, we need two different approaches for Theorems~\ref{th:fracp>=2} and~\ref{CapVieABall}, due to the presence of the multiplicative term $(1-s)$ in front of $L(a_R^*,d_R^*)$. 

Indeed, when we show Theorem~\ref{th:fracp>=2}, we can treat $(1-s)$ as any constant, being the aim of the result to show density estimates at any fixed $s\in (0,1)$. 

However, in Theorem~\ref{CapVieABall}, we want to obtain density estimates whose structural constants do not depend on $s\in [s_0,1)$ where $s_0\in \left(\frac{1}{p},1\right)$. To achieve this, when estimating the left-hand side of~\eqref{iattoll-90843}, we need to consider the fact that $(1-s)$ vanishes as $s\to 1$. As a consequence of this, it is not possible to utilize the same result that we use to prove Theorem~\ref{th:fracp>=2} also to obtain a lower bound for the left-hand side of~\eqref{iattoll-90843}.  

We begin by proving Theorem~\ref{th:fracp>=2}. The result employed to estimate the left-hand side of~\eqref{iattoll-90843} is the following: 

\begin{prop}\label{keleetele}
Let $s\in(0,1)$, $p\in(1,+\infty)$ and $r_1,r_2\in [0,+\infty)$ such that  $\left|B_{r_2}\right|\geq \kappa$ for some $\kappa\in (0,+\infty)$. Also, let 
\begin{equation}\label{thedefofl}
l(B_{r_2}):=\begin{dcases}
\left|B_{r_2}\right|^{\frac{1-sp}{n}}\quad &\mbox{if}\quad s\in \left(0,\frac{1}{p}\right)\\
\left|\ln \left(\left|B_{r_2}\right|\right)\right|\quad &\mbox{if}\quad  s=\frac{1}{p}\\
1\quad &\mbox{if}\quad s\in \left(\frac{1}{p},1\right).
\end{dcases}
\end{equation}
Then, there exists~$\hat{c}:=\hat{c}_{\kappa,n,p}\in (0,+\infty)$  such  that  
\begin{equation}\label{Lrvcbopref}
L(B_{r_2},\R^n\setminus B_{r_2+r_1})+\left|B_{r_2+r_1}\setminus B_{r_2}\right|\geq \hat{c}\left|B_{r_2}\right|^{\frac{n-1}{n}}l(B_{r_2}).
\end{equation} 
\end{prop}
Proposition~\ref{keleetele} is a direct consequence of the following result:  
\begin{thm}\label{SymverofTh}
Let $s\in (0,1)$, $p\in (1,+\infty)$, $\delta:=\min\left\lbrace 2^{-n}, 2^{-n+2-p} \right\rbrace$, $R\in \left(0,+\infty\right)$ and $r\in \left(0,\delta R\right)$. Then, 
\begin{equation}\label{Sonounoskianto}
L(B_R,\R^n\setminus B_{R+r})\geq \begin{dcases} 
\delta R^{n-sp}\quad &\mbox{if}\quad
s\in \left(0,\frac{1}{p}\right)\\
\delta R^{n-1}\log\left(\frac{R}{r}\right)\quad &\mbox{if}\quad s=\frac{1}{p}\\
\delta R^{n-sp} \left(\frac{r}{R}\right)^{1-sp}\quad &\mbox{if}\quad s\in \left(\frac{1}{p},1\right).
\end{dcases}
\end{equation}
\end{thm}

The proof of Theorem~\ref{SymverofTh} can be found in Appendix~\ref{Crocodilewalk}. 

We observe that a lower bound 
in the spirit of Theorem~\ref{SymverofTh}
was also established, with different (and more complicated)
methods, in Theorem~1.6 of~\cite{MR3133422}
for $L(A,D)$, where $A,D\subset\R^n$ are two generic measurable sets satisfying $A\cap D=\emptyset$ and not necessarily a ball $B_R$ and the complement of a ball $\R^n\setminus B_{R+r}$. Although it is possible to prove such a general result also for $p\neq 2$, for the purpose of this paper this further generality is not necessary: we can thereby restrict our attention to balls and complement of balls,
which provides a significant technical simplification, allowing the use of rearrangement techniques. Indeed, using the continuity of the Gagliardo seminorm with respect to the symmetric decreasing rearrangement in the proof of Lemma~\ref{gliorigamieriopachi}, we managed to obtain in the left-hand side of~\eqref{iattoll-90843} the function $L$ computed on the level sets $a_R^*$ and $d_R^*$ of the symmetric decreasing rearrangement $(u-v)^*$.

\begin{proof}[Proof of Proposition~\ref{keleetele}]
We begin by setting 
$$\delta:=\min\left\lbrace 2^{-n}, 2^{-n+2-p} \right\rbrace$$ 
and we study separately the case $r_1\in (0,\delta r_2)$ and $r_1\in [r_2\delta,+\infty)$.\\
\medskip
\textbf{(i) Case $r_1\in\left[\delta r_2,+\infty\right)$}. In this case, we observe that    
\begin{equation}\label{tbeczhntb6}
\begin{split}
\left|B_{r_2+r_1}\setminus B_{r_2}\right|&=\omega_n\left((r_2+r_1)^n-r_2^n\right)\\
&=\omega_n r_2^n \left(\left(1+\frac{r_1}{r_2}\right)^n-1\right)\\
&\geq \omega_n r_2^n \left(\left(1+\delta\right)^n-1\right)\\
&=\left|B_{r_2}\right|\left(\left(1+\delta\right)^n-1\right)\\
&=c_\delta\left|B_{r_2}\right|,
\end{split}
\end{equation}
where $c_\delta:=(1+\delta)^n-1$. 

Thus, we consider the map $\Psi:[\kappa,+\infty)\to (0,+\infty)$ defined by 
\begin{equation*}
\Psi(t):=\frac{t^{\frac{1}{n}}}{1+\left|\ln(t)\right|}.
\end{equation*} 
Also, we notice that
\begin{equation*}
\lim_{t\to +\infty}\Psi(t)=+\infty
\end{equation*}
and therefore 
\begin{equation}\label{krefoplere}
i_{\kappa}:=\inf_{t\in \left[\kappa,+\infty\right)}\psi(t)\in (0,+\infty).
\end{equation}
Then, if we set  
\begin{equation*}
c_2:=c_\delta \min\left\lbrace  \kappa^\frac{1}{n}, 1, i_{\kappa}  \right\rbrace,
\end{equation*}
making use~\eqref{krefoplere} and~\eqref{tbeczhntb6} we obtain that
\begin{equation*}
\begin{split}&\!\!\!\!\!\!\!\!\!\!
L(B_{r_2},\R^n\setminus B_{r_2+r_1})+\left|B_{r_2+r_1}\setminus B_{r_2}\right|\\ & \geq \left|B_{r_2+r_1}\setminus B_{r_2}\right|\\
&\geq c_\delta\left|B_{r_2}\right|\\
&=\begin{dcases}
c_\delta\left|B_{r_2}\right|^{\frac{sp}{n}}\left|B_{r_2}\right|^{\frac{n-sp}{n}}\geq c_\delta\kappa^{\frac{sp}{n}}\left|B_{r_2}\right|^\frac{n-sp}{n}\quad &\mbox{if}\quad s\in \left(0,\frac{1}{p}\right)\\
c_\delta\left|B_{r_2}\right|^{\frac{1}{n}}\left|B_{r_2}\right|^{\frac{n-1}{n}}\geq c_\delta i_{\kappa} \left(1+\left|\ln(\left|B_{r_2}\right|)\right|\right)\left|B_{r_2}\right|^{\frac{n-1}{n}} \quad &\mbox{if}\quad s\in \left[\frac{1}{p},1\right) 
\end{dcases}\\
&\geq c_2 \left|B_{r_2}\right|^{\frac{n-1}{n}} l(B_{r_2}).
\end{split}
\end{equation*}
This concludes the proof of~\eqref{Lrvcbopref} for the case $r_1\in \left[r_2\delta,+\infty\right)$.\\ 
\\
\textbf{(ii) Case $r_1\in [0,\delta r_2)$}. For $s\in \left(\frac{1}{p},1\right)$ we define $z:(0,+\infty)\to (0,+\infty)$ as
\begin{equation*}
z(t):= t^{1-sp}+t.
\end{equation*}
Then, we observe that $t_0:=(sp-1)^\frac{1}{sp}$ is a global minimizer for $z$, and henceforth
\begin{equation}\label{preppepe}
\inf_{t\in (0,+\infty)}z(t)= z(t_0)=sp(sp-1)^{\frac{1-sp}{sp}}.
\end{equation} 
Similarly, if $\alpha\in (0,+\infty)$ we introduce the map $g:(0,\alpha]\to (0,+\infty)$ given by 
\begin{equation*}
g(t):=\ln \left(\frac{\alpha}{t}\right)+t.
\end{equation*} 
Then, we have that $t_\alpha:=\min\left\lbrace 1,\alpha\right\rbrace$ is a global minimizer for~$g$, and consequently 
\begin{equation}\label{zp.mj.li-09}
\inf_{t\in (0,\alpha]}g(t)=g(t_1)=\begin{dcases}
\ln(\alpha)+1\quad &\mbox{if}\quad \alpha\in (1,+\infty)\\
\alpha \quad &\mbox{if}\quad \alpha\in (0,1]. 
\end{dcases} 
\end{equation} 
Notice that 
\begin{equation}\label{robach}
\begin{split}
\left|B_{r_2+r_1}\setminus B_{r_2}\right|&=\omega_n\left((r_2+r_1)^n-r_2^n\right)\\
&=\omega_n\left(\sum_{j=0}^n\begin{pmatrix} n\\j \end{pmatrix} r_1^j r_2^{n-j}-r_2^n\right)\\
&\geq n \omega_n r_1 r_2^{n-1}\\
&=\left|\partial B_1\right| r_1 r_2^{n-1}\\
&\geq r_1 r_2^{n-1}.    
\end{split}
\end{equation}
Then,  we use~\eqref{Sonounoskianto},~\eqref{preppepe},~\eqref{zp.mj.li-09}, \eqref{robach} and the fact that $\left|B_{r_2}\right|\geq \kappa$ to obtain  
\begin{equation}\label{phjh}
\begin{split}
\frac{1}{\delta}L(B_{r_2},\R^n\setminus B_{r_2+r_1})+\left|B_{r_2+r_1}\setminus B_{r_2}\right| &\geq \begin{dcases}
r_2^{n-sp} +\left|B_{r_2+r_1}\setminus B_{r_2}\right|\quad &\mbox{if}\quad s\in \left(0,\frac{1}{p}\right)\\
r_2^{n-1}\log\left(\frac{r_2}{r_1}\right)+\left|B_{r_2+r_1}\setminus B_{r_2}\right|\quad &\mbox{if}\quad s=\frac{1}{p}\\
r_2^{n-1} r_1^{1-sp}+\left|B_{r_2+r_1}\setminus B_{r_2}\right|\quad &\mbox{if}\quad s\in \left(\frac{1}{p},1\right)
\end{dcases}\\
&\geq \begin{dcases}
r_2^{n-sp} \quad &\mbox{if}\quad s\in \left(0,\frac{1}{p}\right)\\
r_2^{n-1}\left(\ln\left(\frac{r_2}{r_1}\right)+  r_1\right) \quad &\mbox{if}\quad s=\frac{1}{p}\\
r_2^{n-1} \left(r_1^{1-sp}+r_1\right) \quad &\mbox{if}\quad s\in \left(\frac{1}{p},1\right)
\end{dcases}\\
&\geq \begin{dcases}
r_2^{n-sp} \quad &\mbox{if}\quad s\in \left(0,\frac{1}{p}\right)\\
r_2^{n-1}g(t_{r_2}) \quad &\mbox{if}\quad s=\frac{1}{p}\\
r_2^{n-1} sp(sp-1)^{\frac{1-sp}{sp}} \quad &\mbox{if}\quad s\in \left(\frac{1}{p},1\right).
\end{dcases}\\
\end{split}
\end{equation}
In particular, there exists a constant $c_{n,\kappa,p}\in (0,+\infty)$ such  that 
\begin{equation*}
\frac{1}{\delta}L(B_{r_2},\R^n\setminus B_{r_2+r_1})+\left|B_{r_2+r_1}\setminus B_{r_2}\right|\geq c_{n,\kappa,p}\left|B_{r_2}\right|^{\frac{n-1}{n}}l(B_{r_2}).
\end{equation*}  
This concludes the proof of~\eqref{Lrvcbopref} for the case $r_1\in (0,\delta r_2)$. 
\end{proof}

Now, thanks to Lemma~\ref{gliorigamieriopachi} and Proposition~\ref{keleetele}, we are ready to prove Theorem~\ref{th:fracp>=2}. 

\begin{proof}[Proof of Theorem~\ref{th:fracp>=2}]
First, according to Remark~\ref{plothgnebvdf}, we can assume without loss of generality that 
\begin{equation*}
c_0\in (1,+\infty).
\end{equation*}

We introduce the constant
\begin{equation}\label{c5-0987ujgbt}
c_4:=\min \left\lbrace \frac{c_2}{\hat{c}_p}, (1-s)\bar{c}_4\right\rbrace,
\end{equation}
where $c_2$ and $\bar{c}_4$ are provided in~\eqref{c2-34566}. 

We notice that $u\in X^{s,p}(\Omega)$ satisfies all the hypothesis of Lemma~\ref{gliorigamieriopachi}. Therefore, thanks to~\eqref{iattoll-90843}, for every $R\in \left(K_0,+\infty\right)$ such that $B_R\subset\Omega$ and $K\in (K_0,R)$  
\begin{equation}\label{cdm:D}
L(a_R^*,d_R^*)+\left|b_R^*\right|\leq \frac{\epsilon}{c_4\hat{c}_p}+\frac{c_5}{c_4}\int_{0}^R(R-t+1)^{-sp}V'(t)\,dt.
\end{equation}
Moreover, we notice that $u\in X^{s,p}(\Omega)$ satisfies all the assumptions of Proposition~\ref{jamio-le}.  Hence, from~\eqref{iann-uytb} and~\eqref{levelsetpalle4563} we deduce that all the conditions in order to apply Proposition~\ref{keleetele} are met with the choice $\kappa=c_0$. 

Accordingly, making use of~\eqref{Lrvcbopref} in Proposition~\ref{keleetele} we obtain that 
\begin{equation*}
L(a_R^{*},d_R^{*})+\left|b_R^{*}\right| \geq \hat{c}\left|a_R^{*}\right|^{\frac{n-1}{n}} l(a_R^{*}),
\end{equation*}
where $l$ is provided in~\eqref{thedefofl} and $\hat{c}:=\hat{c}_{c_0,n,p}\in (0,+\infty)$. 

From this inequality,~\eqref{takanaka} and~\eqref{rfdlopcbgft5}, we obtain that for all~$R\in \left(K_0+r_0,+\infty\right)$ and $K\in (K_0,R-r_0)$,
\begin{equation}\label{side-p.l}
\begin{split}
L(a_R^{*},d_R^{*})+\left|b_R^{*}\right|& \geq \hat{c}\left|a_R\right|^{\frac{n-1}{n}} l(a_R)\\
&\geq \hat{c}\left|B_{R-K}\cap \left\lbrace u>\theta_*\right\rbrace\right|l\left(B_{R-K}\cap \left\lbrace u>\theta_*\right\rbrace\right)\\
&=\hat{c} V(R-K)^{\frac{n-1}{n}}l_{R-K},
\end{split}
\end{equation} 
where
\begin{equation*}
l_R:=l\left(B_R\cap \left\lbrace u>\theta_* \right\rbrace\right).
\end{equation*}

Note also that 
\begin{equation}\label{rfdvchytgbffff} \mbox{the map }
R\longmapsto l_R\mbox{  is non decreasing}.
\end{equation}
Now, from equation~\eqref{side-p.l} and~\eqref{cdm:D} we deduce that, for all~$R\in \left(K_0+r_0,+\infty\right)$ such that $B_R\subset\Omega$ and $K\in (K_0,R-r_0)$, 
\begin{equation}\label{quasi}
V(R-K)^{\frac{n-1}{n}}l_{R-K}\leq \frac{\epsilon}{c_4\hat{c}\hat{c}_p}+\frac{c_5}{c_4\hat{c}}\int_{0}^R(R-t+1)^{-sp}V'(t)\,dt.
\end{equation}

Also, if $\rho\geq 1$ then 
\begin{equation}\label{pvdcbgrfeee4}
\int_{t}^{\frac{3}{2}\rho} (R+1-t)^{-sp}\,dR\leq \begin{dcases}
\frac{1}{1-sp}\left(\frac{5}{2}\rho\right)^{1-sp}\quad &\mbox{if}\quad s\in \left(0,\frac{1}{p}\right)\\
\log\left(\frac{5}{2}\rho\right)\quad &\mbox{if}\quad s=\frac{1}{p}\\
\frac{1}{sp-1}\quad &\mbox{if}\quad s\in \left(\frac{1}{p},1\right).
\end{dcases}
\end{equation} 
Moreover, integrating both sides of~\eqref{quasi} in $R\in \left[\rho,\frac{3}{2}\rho\right]$, with $\rho\in \left(K_0+r_0,+\infty\right)$, $B_{\frac{3}{2}\rho}\subset\Omega$ and $K\in \left(K_0, \min\left\lbrace \rho-r_0,\frac{\rho}{2}\right\rbrace\right)$, we infer that 
\begin{equation}\label{mbfvgtrrrrf}
\begin{split}
\int_{\rho}^{\frac{3}{2}\rho} l_{R-K}V(R-K)^{\frac{n-1}{n}}\,dR &\leq \frac{\rho}{2}\frac{\epsilon}{c_4\hat{c}\hat{c}_p}+\frac{c_5}{c_4\hat{c}}\int_{\rho}^{\frac{3}{2}\rho}\int_{0}^R \left(R-t+1\right)^{-sp}V'(t)\,dt\,dR\\
& \leq  \frac{\rho}{2}\frac{\epsilon}{c_4\hat{c}\hat{c}_p}+\frac{c_5}{c_4\hat{c}}\int_{0}^{\frac{3}{2}\rho}\int_{t}^{\frac{3}{2}\rho} \left(R-t+1\right)^{-sp}V'(t)\,dR\,dt.
\end{split}
\end{equation}

Thus, making use of~\eqref{pvdcbgrfeee4},~\eqref{rfdvchytgbffff} and~\eqref{mbfvgtrrrrf}, we obtain that,
for all~$\rho\in \left(K_0+r_0,+\infty\right)$ such that $B_{\frac{3}{2}\rho}\subset\Omega$ and all~$K\in \left(K_0, \min\left\lbrace \rho-r_0,\frac{\rho}{2}\right\rbrace\right)$,
\begin{equation}\label{unnerrtuwvydiukuyg}
\rho l_{\rho-K}V(\rho-K)^\frac{n-1}{n}\leq \frac{\rho}{2}\frac{\epsilon}{c_4\hat{c}\hat{c}_p}+\frac{c_5}{c_4\hat{c}}\begin{dcases}
\frac{1}{1-sp}\left(\frac{5}{2}\rho\right)^{1-sp}V\left(\frac{3}{2}\rho\right)\quad &\mbox{if}\quad s\in \left(0,\frac{1}{p}\right)\\
\log\left(\frac{5}{2}\rho\right)V\left(\frac{3}{2}\rho\right)\quad &\mbox{if}\quad s=\frac{1}{p}\\
\frac{1}{sp-1}V\left(\frac{3}{2}\rho\right)\quad &\mbox{if}\quad s\in \left(\frac{1}{p},1\right).
\end{dcases}
\end{equation}
Now, we introduce the constants  
\begin{equation*}
C_{s,p}^{(1)}:=\frac{5 c_5}{c_4\hat{c}(1-sp)},\quad C_{s,p}^{(2)}:=\frac{2c_5}{c_4\hat{c}},\quad C_{s,p}^{(3)}:= \frac{c_5}{(sp-1)c_4\hat{c}}\quad\mbox{and}\quad C_{s,p}^{(4)}:=\frac{1}{c_4\hat{c}\hat{c}_p}.
\end{equation*}
Then, if we denote $r:=\frac{\rho}{2}$ and we assume that $r\geq 5$, we notice that $\rho-K\geq r$ and therefore, owing to~\eqref{unnerrtuwvydiukuyg},
\begin{equation*}
\begin{dcases}
2r^{sp} V(r)^{\frac{n-sp}{n}}\leq r^{sp} \epsilon C_{s,p}^{(4)}+C_{s,p}^{(1)} V(3r)\quad &\mbox{if}\quad s\in \left(0,\frac{1}{p}\right)\\
2r\frac{\log(V(r))}{\log(r)}V(r)^{\frac{n-1}{n}} \leq \frac{r}{\log(r)} \epsilon C_{s,p}^{(4)}+ C_{s,p}^{(2)}V(3r)\quad &\mbox{if}\quad s=\frac{1}{p}\\
2rV(r)^{\frac{n-1}{n}}\leq r \epsilon C_{s,p}^{(4)}+ C_{s,p}^{(3)} V(3r)\quad &\mbox{if}\quad  s\in \left(\frac{1}{p},1\right).
\end{dcases}
\end{equation*}
We can rewrite these inequalities as 
\begin{equation}\label{iopre}
\begin{dcases}
r^{sp} \left(2V(r)^{\frac{n-sp}{n}}-\epsilon C_{s,p}^{(4)}\right)\leq C_{s,p}^{(1)} V(3r)\quad &\mbox{if}\quad s\in \left(0,\frac{1}{p}\right)\\
\frac{r}{\log(r)} \left(2\log(V(r))V(r)^{\frac{n-1}{n}}-\epsilon C_{s,p}^{(4)}\right) \leq  C_{s,p}^{(2)}V(3r)\quad &\mbox{if}\quad s=\frac{1}{p}\\
r\left(2V(r)^{\frac{n-1}{n}}-\epsilon C_{s,p}^{(4)}\right) \leq  C_{s,p}^{(3)} V(3r)\quad &\mbox{if}\quad  s\in \left(\frac{1}{p},1\right).
\end{dcases}
\end{equation}

We also observe that, for every $r\in (r_0,+\infty)$ such that $B_r\subset \Omega$,
\begin{equation}\label{Lmadbur}
\begin{dcases}
\mbox{if}\quad \epsilon\leq  \frac{c_0^{\frac{n-sp}{n}}}{C_{s,p}^{(4)}}\quad &\mbox{then}\quad 2V(r)^{\frac{n-sp}{n}}-\epsilon C_{s,p}^{(4)}\geq V(r)^{\frac{n-sp}{n}}\\
\mbox{if}\quad  \epsilon\leq \frac{\log(c_0)c_0^{\frac{n-1}{n}}}{C_{s,p}^{(4)}}\quad &\mbox{then}\quad 2\log(V(r))V(r)^{\frac{n-1}{n}}-\epsilon C_{s,p}^{(4)}\geq \log(V(r))V(r)^{\frac{n-1}{n}}\\
\mbox{if}\quad \epsilon\leq \frac{c_0^\frac{n-1}{n}}{C_{s,p}^{(4)}} \quad &\mbox{then}\quad 2V(r)^{\frac{n-1}{n}}-\epsilon C_{s,p}^{(4)}\geq \log(V(r))V(r)^{\frac{n-1}{n}}. 
\end{dcases}
\end{equation}

We set 
\begin{equation*}
\delta :=\min \left\lbrace \frac{c_0^{\frac{n-sp}{n}}}{C_{s,p}^{(4)}},\frac{\log(c_0)c_0^{\frac{n-1}{n}}}{C_{s,p}^{(4)}}, \frac{c_0^\frac{n-1}{n}}{C_{s,p}^{(4)}}\right\rbrace \quad\mbox{and}\quad R_0:=\max\left\lbrace r_0, 5, \frac{K_0+r_0}{2} \right\rbrace.
\end{equation*}
In this way, we deduce from~\eqref{iopre} and~\eqref{Lmadbur}
that, for every~$\epsilon\in (0,\delta)$ and
every~$r\in [R_0,+\infty)$ such that $B_{3r}\subset \Omega$,   
\begin{equation*}
\begin{dcases}
r^{sp} V(r)^{\frac{n-sp}{n}}\leq C_{s,p}^{(1)} V(3r)\quad &\mbox{if}\quad s\in \left(0,\frac{1}{p}\right)\\
r\frac{\log(V(r))}{\log(r)}V(r)^{\frac{n-1}{n}} \leq  C_{s,p}^{(2)}V(3r)\quad &\mbox{if}\quad s=\frac{1}{p}\\
rV(r)^{\frac{n-1}{n}}\leq C_{s,p}^{(3)} V(3r)\quad &\mbox{if}\quad  s\in \left(\frac{1}{p},1\right).
\end{dcases}
\end{equation*}
From now on $K$ is fixed once and for all. Thanks to this and~\eqref{Fcon-812345} we notice that (C.1) and (C.2) in Lemma~C.1 in~\cite{DFVERPP} hold true with the following choices:
\begin{equation*}
\begin{dcases}
\sigma=sp,\,\nu=n,\,C=C_{s,p}^{(1)},\,\gamma=3,\,\mu=c_0 \quad &\mbox{if}\quad s\in \left(0,\frac{1}{p}\right)\\ 
\sigma=1,\,\nu=n,\,C=C_{s,p}^{(2)},\,\gamma=3,\,\mu=c_0 \quad &\mbox{if}\quad s=\frac{1}{p}\\
\sigma=1,\,\nu=n,\,C=C_{s,p}^{(3)},\,\gamma=3,\,\mu=c_0 \quad &\mbox{if}\quad s\in \left(\frac{1}{p},1\right).
\end{dcases}
\end{equation*}
Therefore, in virtue of~(C.3) in Lemma~C.1 in~\cite{DFVERPP}, we obtain that, for every $r\in \left[R^*,+\infty\right)$ satisfying $B_{3r}\subset\Omega$,
\begin{equation}\label{hndvcfgtr566}
V(r)\geq \tilde{c}r^n,
\end{equation} 
for a suitable $\tilde{c}\in (0,1)$,
depending only on $C$, $\sigma$, $\mu$, $\gamma$, $\nu$ and $R_0$, and a suitable $R^*\in\left[R_0,+\infty\right)$, depending only on $R_0$ and $\gamma$.

In particular, from~\eqref{hndvcfgtr566} we deduce that, when $r\in \left[R^*,+\infty\right)$ and $B_{3r} \subset\Omega$,
\begin{equation}\label{euno}
\tilde{c} r^n\leq \mathcal{L}^n(B_r\cap \left\lbrace u>\theta_*\right\rbrace)=\mathcal{L}^n(B_r\cap \left\lbrace u>\theta^*\right\rbrace)+\mathcal{L}^n(B_r\cap \left\lbrace \theta_*<u\leq \theta^*\right\rbrace).
\end{equation}

We also set $c_{m,\theta_*}:=(1+\theta_*)^{-m}$ and observe that
\begin{equation}\label{edue}
\begin{split}
\mathcal{L}^n\left(B_r\cap \left\lbrace \theta_{*}<u\leq \theta^* \right\rbrace\right)&=\int_{B_r\cap \left\lbrace \theta_{*}<u\leq \theta^* \right\rbrace}\,dx\\
& \leq c_{m,\theta_*}\int_{B_r\cap \left\lbrace \theta_{*}<u\leq \theta^* \right\rbrace}\left|1+u(x)\right|^m\,dx.
\end{split}
\end{equation}

Finally, we deduce from~\eqref{euno} and~\eqref{edue}
that, for every $r\geq R_*$  such that $B_{3r}\subset \Omega$, 
\begin{equation*}
\begin{split}
\mathcal{L}^n\left(B_r\cap \left\lbrace u>\theta^* \right\rbrace\right) \geq \;& \tilde{c}r^n-\mathcal{L}^n\left(B_r\cap \left\lbrace \theta_*<u\leq \theta^*\right\rbrace\right)\\
\geq \;& \tilde{c} r^n- c_{m,\theta_*}\int_{B_r\cap \left\lbrace \theta_{*}<u\leq \theta^* \right\rbrace}\left|1+u(x)\right|^m\,dx.\qedhere
\end{split}
\end{equation*}
\end{proof}

Now we prove Theorem~\ref{CapVieABall}. As aforementioned, in this case, to estimate the left-hand side of~\eqref{iattoll-90843}, it is not possible to use Proposition~\ref{keleetele}. Hence, to obtain a suitable lower bound for the left-hand side of~\eqref{iattoll-90843}, we make use of the additional H\"older regularity assumption on the $\epsilon$-minimizer $u\in X^{s,p}(\Omega)$.  

\begin{proof}[Proof of Theorem~\ref{CapVieABall}]
First, we notice that, according to Remark~\ref{plothgnebvdf}, we can assume without loss of generality that 
\begin{equation*}
c_0\in (1,+\infty).
\end{equation*}
Also, it follows
from~\eqref{levelsetpalle4563} that, for every $R\in  \left[\bar{R}_\mu,+\infty\right)$,
\begin{equation}\label{misuradibstar}
\left|b_R^*\right|\geq n\omega_n  r_2^{n-1}r_1.
\end{equation}
Then, we claim that there exists a constant $\tilde{C}:=\tilde{C}_{\alpha,s_0,n,m,p,c_1,\left[u\right]_{B,\alpha}}\in (0,+\infty)$ such that, when $R\in \left[K_0+r_0,+\infty\right)$ and $B_{R+2}\subset\Omega$,
\begin{equation}\label{boundforr_1}
r_1> \tilde{C},  
\end{equation}
where $\left[u\right]_{B,\alpha}$ has been defined in~\eqref{gliorteitbergf}.

To show~\eqref{boundforr_1}, we begin by setting the following notation. For every function $f\in C_c(\R^n)$ we denote by $\omega_f$ its modulus of continuity, namely, for every $\delta>0$,
\begin{equation*}
\omega_f(\delta):=\sup \left\lbrace \left|f(x)-f(y)\right| \mbox{  s.t.  } \left|x-y\right|<\delta \right\rbrace.
\end{equation*}
It is well known, see for instance Corollary 6.1 in~\cite{MR1695019}, that for every $\delta>0$ and $f\in C_c(\R^n)$, 
\begin{equation*}
f^*\in C_c(\R^n)\quad \mbox{and}\quad \omega_{f^*}(\delta)\leq \omega_{f}(\delta).
\end{equation*}

Since $u\in C_{\textit{loc}}^\alpha(\Omega)$, it follows from Proposition~\ref{hiolevcgfdb4t54}
that $u-v\in C_c^\alpha(\R^n)$ and therefore  
\begin{equation}\label{fvd-098}
(u-v)^*\in C_c^\alpha(\R^n)\quad\mbox{and}\quad \omega_{(u-v)^*}(\delta)\leq \omega_{u-v}(\delta).
\end{equation}
From this,~\eqref{iann-uytb} and the Mean Value Theorem, we obtain that for every $R\in\left(K_0+r_0,+\infty\right)$
there exists~$x_0\in b_R^*$ such that 
\begin{equation}\label{cre-8}
(u-v)^*(x_0)=\frac{1}{2}\left(\frac{1+\theta_*}{4}+\frac{1+\theta_*}{8}\right)= \frac{3}{16}\left(1+\theta_*\right).
\end{equation}
Now, for every $t\in (0,+\infty)$ we define the annulus 
\begin{equation*}
\mathcal{C}_{x_0}^t:=\left\lbrace x\in\R^n\mbox{ s.t. } \left| \left|x\right|-\left|x_0\right|\right|\leq t  \right\rbrace
\end{equation*}
Then, we claim that there exists~$d:=d_{\alpha,s_0,n,m,p,c_1,\left[u\right]_{B,\alpha}}\in (0,+\infty)$ such that 
\begin{equation}\label{colizionedelmascarpone}
\mathcal{C}_{x_0}^d \subset b_R^*.
\end{equation}
In particular, if claim~\eqref{colizionedelmascarpone} holds true, then   
\begin{equation*}
r_1\geq 2d,
\end{equation*}
proving~\eqref{boundforr_1}. 

For this reason, it only remains to check the validity of~\eqref{colizionedelmascarpone}. 
To do so, we notice that, according to~\eqref{tebegdfjuybgf}, there exists~$C^{(0)}:=C_{s_0,n,m,p,c_1}^{(0)}\in(0,+\infty)$ such that, for every $\delta\in(0,2)$,
\begin{equation}\label{gbvdy-76yh}
\omega_{u-v}(\delta) \leq \left[2C^{(0)}+4\left[u\right]_{B,\alpha}\right] \delta^\alpha.
\end{equation}

We also denote 
\begin{equation*}
C^{(1)}:=2C^{(0)}+4\left[u\right]_{B,\alpha}.
\end{equation*} 
Moreover, by construction, the function~$(u-v)^*$ is rotationally symmetric. In particular, for every $x\in \R^n$,
\begin{equation*}
(u-v)^* (x)=(u-v)^* \left(\left|x\right|\frac{x_0}{\left|x_0 \right|}   \right). 
\end{equation*} 

Hence, keeping in mind~\eqref{fvd-098} and~\eqref{gbvdy-76yh}, we obtain that, for every $\delta\in (0,2)$  and $x\in \mathcal{C}_{x_0}^\delta$,
\begin{equation}\label{slime-04fev}
\begin{split}
\left|(u-v)^*(x)-(u-v)^*(x_0)\right|&=\left|(u-v)^* \left(\left|x\right|\frac{x_0}{\left|x_0 \right|}\right)-(u-v)^*(x_0)\right|\\
&\leq \omega_{(u-v)^*}(\delta)\\
&\leq \omega_{(u-v)}(\delta)\\
&\leq C^{(1)}\delta^\alpha.
\end{split}
\end{equation}
Thus, if we define
\begin{equation*}
d:=\left(\frac{1+\theta_*}{16\, C^{(1)}}\right)^\frac{1}{\alpha},
\end{equation*} 
making use of~\eqref{cre-8} and~\eqref{slime-04fev} we evince that, for every $R\in \left[K_0+r_0,+\infty\right)$ such that $B_{R+2}\subset\Omega$,
\begin{equation*}
\mathcal{C}_{x_0}^d\subset b_{R}^*,
\end{equation*}
proving claim~\eqref{colizionedelmascarpone} and thus~\eqref{boundforr_1}.

Now, thanks to~\eqref{rfdlopcbgft5},~\eqref{takanaka},~\eqref{misuradibstar},~\eqref{levelsetpalle4563} and~\eqref{boundforr_1} we evince that, for every $R\in \left[K_0+r_0,+\infty\right)$ such that~$B_{R+2}\subset\Omega$ and for
every~$K\in \left(K_0,R-r_0\right)$,
\begin{equation}\label{njtg-bgd-0843}
\begin{split}
\left|b_{R}^*\right|&\geq n\omega_n 2d r_2^{n-1}=n\omega_n^{\frac{1}{n}} 2d \left|a_R^*\right|^{\frac{n-1}{n}}= n\omega_n^{\frac{1}{n}} 2d \left|a_R\right|^{\frac{n-1}{n}} \geq n\omega_n^{\frac{1}{n}} 2d V(R-K)^{\frac{n-1}{n}}.  
\end{split}
\end{equation}

Also, according to~\eqref{iattoll-90843}, for every $R\in\left( K_0,+\infty\right)$ such that $B_R\subset\Omega$ and for every~$K\in (K_0,R)$,
\begin{equation}\label{zfvcbghbf5643e}
\left|b_R^*\right|\leq \frac{\epsilon}{c_2}+ C^{(2)} \int_{0}^R  (R-t+1)^{-sp}V'(t)\,dt,
\end{equation}
where we have defined 
\begin{equation}\label{kdvgtecas45237777}
C^{(2)}:=\frac{c_3+C_\mu^{m-1}(2c_1+\Lambda)}{c_2}. 
\end{equation}
We recall that the constants $c_2$ and $c_3$ are defined in~\eqref{c2-34566}. Note also that $c_3$ depends on $K$.

Now, from~(3.1),~(3.3) in~\cite{DFVERPP}
and~\eqref{muuuuuuu} here, it is immediate to verify that
\begin{equation}\label{controdiconsta-1}
\bar{R}_{s_0}:=\sup_{s\in \left[s_0,1\right)} \bar{R}_\mu\in (0,+\infty) \quad\mbox{and}\quad C_{s_0}:=\sup_{s\in\left[s_0,1\right)}C_\mu\in(0,+\infty).
\end{equation} 
Then, recalling~\eqref{suK0tntddr}, we define the constant
\begin{equation*}
K_{s_0}:= \left[C_{s_0}\left(\frac{2}{1+\theta_*}\right)\right]^\frac{m-1}{p s_0}-1.
\end{equation*}
Moreover, making use of~\eqref{kdvgtecas45237777}  for every $K\in (K_{s_0},R-r_0)$ we set
\begin{equation*}
C^{(3)}:=\frac{C_{s_0}^{m-1}(2c_1+\Lambda)+c_3}{c_2}. 
\end{equation*}

Finally, in light of~\eqref{njtg-bgd-0843},~\eqref{zfvcbghbf5643e}  and~\eqref{controdiconsta-1}, we deduce that, when~$R\in \left(K_{s_0}+r_0,+\infty \right)$, $B_{R+4}\subset\Omega$ and $K\in (K_{s_0},R-r_0)$,
\begin{equation*}
V(R-K)^{\frac{n-1}{n}}\leq \frac{\epsilon}{c_2 n\omega_n^{\frac{1}{n}} 2d}+\frac{C^{(3)}}{n\omega_n^{\frac{1}{n}} 2d}\int_{0}^R \left(R-t+1\right)^{-sp}V'(t)\,dt.
\end{equation*}
Now, we integrate both sides of this inequality in $R\in \left[\rho,\frac{3}{2}\rho\right]$, with $\rho \in \left( K_{s_0}+r_0,+\infty\right)$, $B_{\frac{3}{2}\rho+4}$ and $K\in \left(K_{s_0}, \min \left\lbrace\rho-r_0,\frac{\rho}{2} \right\rbrace\right)$ and conclude that 
\begin{equation*}
\begin{split}
\int_{\rho}^{\frac{3}{2}\rho} V(R-K)^\frac{n-1}{n}\,dR \leq &\frac{\rho}{2}\frac{\epsilon}{c_2 n\omega_n^{\frac{1}{n}} 2d}+ \frac{C^{(3)}}{n\omega_n^\frac{1}{n} 2d}\int_{\rho}^{\frac{3}{2}\rho}\int_{0}^R \left(R-t+1\right)^{-sp}V'(t)\,dt\,dR\\
\leq & \frac{\rho}{2}\frac{\epsilon}{c_2 n\omega_n^{\frac{1}{n}} 2d}+\frac{C^{(3)}}{n\omega_n^\frac{1}{n} 2d}\int_{0}^{\frac{3}{2}\rho}\int_{t}^{\frac{3}{2}\rho} \left(R-t+1\right)^{-sp}V'(t)\,dR\,dt. 
\end{split}
\end{equation*}

Accordingly, recalling~\eqref{pvdcbgrfeee4},    
\begin{equation}\label{zknvzjnzfvre}
\begin{split}
\rho V(\rho-K) \leq & \frac{\rho}{2}\frac{\epsilon}{c_2 n\omega_n^{\frac{1}{n}} 2d}     +\frac{C^{(3)}}{n\omega_n^\frac{1}{n} 2d(sp-1)}V\left(\frac{3}{2}\rho\right)\\
\leq & \frac{\rho}{2}\frac{\epsilon}{c_2 n\omega_n^{\frac{1}{n}} 2d}     +\frac{C^{(3)}}{n\omega_n^\frac{1}{n} 2d(s_0 p-1)}V\left(\frac{3}{2}\rho\right).
\end{split}
\end{equation} 
Now, we introduce the constants
\begin{equation}\label{cevhyyy-op}
C^{(4)}:= \frac{C^{(3)}}{n\omega_n^\frac{1}{n} 2d(s_0 p-1)}\quad\mbox{and}\quad C^{(5)}:=\frac{1}{c_2 n\omega_n^{\frac{1}{n}} 2d}.   
\end{equation}
Then, if~$r:=\frac{\rho}{2}$, we assume that $r\in \left[5,+\infty\right)$ and we notice that $\rho-K\geq r$. It thereby follows
from~\eqref{zknvzjnzfvre} and~\eqref{cevhyyy-op} that 
\begin{equation}\label{Doechii}
r\left(2V(r)^{\frac{n-1}{n}}-\epsilon C^{(5)}\right) \leq  C^{(4)} V(3r). 
\end{equation}
In particular, if we choose 
\begin{equation*}
\delta:=\frac{c_0^{\frac{n-1}{n}}}{C^{(5)}}\quad\mbox{and}\quad R_0:=\max\left\lbrace r_0, 5,\frac{K_{s_0}+r_0}{2}\right\rbrace
\end{equation*}
we deduce 
from~\eqref{Fcon-812345} and~\eqref{Doechii} that, for every $\epsilon\in(0,\delta)$ and $r\in \left[R_0,+\infty\right)$ such that $B_{4r}\subset \Omega$,
\begin{equation*}
rV(r)^{\frac{n-1}{n}}\leq C^{(4)}V(3r).  
\end{equation*}
From now on $K$ is fixed once and for all. Thanks to the above inequality and~\eqref{Fcon-812345} we find that~(C.1) and~(C.2) in Lemma~C.1 in~\cite{DFVERPP} hold true with the following choices:
\begin{equation*}
\sigma=1,\quad \nu=n,\quad C=C^{(0)},\quad \gamma=3\quad\mbox{and}\quad \mu=c_0. 
\end{equation*}
Therefore, in virtue of~(C.3) in Lemma~C.1 in~\cite{DFVERPP}, we obtain that, for every $r\in \left[R^*,+\infty\right)$ such that $B_{4r}\subset\Omega$,
\begin{equation}\label{ghbdvfeee}
V(r)\geq \tilde{c}{r^n},
\end{equation}
for a suitable $\tilde{c}\in (0,1)$, depending only on $C$, $\sigma$, $\mu$, $\gamma$, $\nu$ and $R_0$, and a suitable $R^*\in [R_0,+\infty)$, depending only on $R_0$ and $\gamma$.

In particular, for all~$r\in \left[R^*,+\infty\right)$ such that $B_{4r} \subset\Omega$,
\begin{equation}\label{euno.1}
\tilde{c} r^n\leq \mathcal{L}^n(B_r\cap \left\lbrace u>\theta_*\right\rbrace)=\mathcal{L}^n(B_r\cap \left\lbrace u>\theta^*\right\rbrace)+\mathcal{L}^n(B_r\cap \left\lbrace \theta_*<u\leq \theta^*\right\rbrace).
\end{equation}

Now we set~$c_{m,\theta_*}:=(1+\theta_*)^{-m}$ and we see that
\begin{equation}\label{edue.2}
\begin{split}
\mathcal{L}^n\left(B_r\cap \left\lbrace \theta_{*}<u\leq \theta^* \right\rbrace\right)&=\int_{B_r\cap \left\lbrace \theta_{*}<u\leq \theta^* \right\rbrace}\,dx\\
& \leq c_{m,\theta_*}\int_{B_r\cap \left\lbrace \theta_{*}<u\leq \theta^* \right\rbrace}\left|1+u(x)\right|^m\,dx.
\end{split}
\end{equation}
Finally, from~\eqref{euno.1} and~\eqref{edue.2} we deduce that, for every $r\in\left[R^*,+\infty\right)$  such that $B_{4r}\subset \Omega$,
\begin{equation*}
\begin{split}
\mathcal{L}^n\left(B_r\cap \left\lbrace u>\theta^* \right\rbrace\right) \geq & \tilde{c}r^n-\mathcal{L}^n\left(B_r\cap \left\lbrace \theta_*<u\leq \theta^*\right\rbrace\right)\\
\geq & \tilde{c} r^n- c_{m,\theta_*}\int_{B_r\cap \left\lbrace \theta_{*}<u\leq \theta^* \right\rbrace}\left|1+u(x)\right|^m\,dx.\qedhere
\end{split}
\end{equation*}   
\end{proof}

\subsection{Proof of Corollaries~\ref{bmrlsadcbl} \&~\ref{coro-09}}\label{hyg-yh-654}

In this section we prove Corollary~\ref{bmrlsadcbl} and~\ref{coro-09}. We begin with Corollary~\ref{bmrlsadcbl}. To do so, we first recall the following Morrey's inequality. 
\begin{thm}[Lemmas~5.1,~2.7 and equation~8.2 in~\cite{brasco2024morrey}]\label{basdcfzrty65}
Let $s\in (0,1)$, $p\in (1,+\infty)$, $x_0\in\R^n$ and $u\in C\left(\overline{B_2(x_0)}\right)$. Then, if $sp>n$ there exists some constant $C_{n,p}\in (0,+\infty)$ such that
\begin{equation*}
\left[u\right]_{C^{s-\frac{n}{p}}(B_1(x_0))}^p\leq \frac{C_{n,p}(1-s)}{(sp-n)^{p-1}}\int_{B_2(x_0)}\int_{B_2(x_0)}\frac{\left|u(x)-u(y)\right|^p}{\left|x-y\right|^{n+sp}}\,dx\,dy.
\end{equation*}
\end{thm}

\begin{proof}[Proof of Corollary~\ref{bmrlsadcbl}]
First, we assume that $n\geq 2$. Since $u\in X^{s,p}(\Omega)$ is a minimizer of $\mathcal{E}_s^p$ we can apply Theorem~A.1 in~\cite{DFVERPP} and obtain that $u\in C_{\textit{loc}}^\alpha(\Omega)$, where $\alpha\in (0,1)$ depends only on $n$, $p$ and $s_0$. 

Furthermore, there exists a constant $C_{s_0,n,p}$ such that, for every $x_0\in B_R$,
\begin{equation}\label{Stimaholdercozzi}
\left[u\right]_{C^{\alpha}\left(B_{1}(x_0)\right)} \leq C_{s_0,n,p} \left(\left\|u\right\|_{L^\infty(B_4(x_0))}+\mbox{Tail}(u,x_0,4)+\left\|W\right\|_{L^\infty(B_8(x_0))}^\frac{1}{p}\right).
\end{equation} 
Using the fact that $\left\|u\right\|_{L^\infty(\R^n)}\leq 1$ we obtain that 
\begin{equation*}
\begin{split}
\mbox{Tail}(u,x_0,2) &\leq  \left[(1-s)2^{sp} \int_{\R^n\setminus B_2(x_0)}\frac{dy}{\left|y-x_0\right|^{n+sp}}\right]^\frac{1}{p-1}\\
&=\left[\frac{(1-s)}{sp}\omega_{n-1}\right]^\frac{1}{p-1}\\
&\leq \tilde{C}_{s_0,n,p}. 
\end{split}
\end{equation*}
Therefore, it follows from~\eqref{potential} and~\eqref{Stimaholdercozzi} that 
\begin{equation*}
\left[u\right]_{C^{\alpha}(B_1(x_0))} \leq C_{s_0,n,p}\left(1+\tilde{C}_{s_0,n,p}+2^\frac{m}{p}\Lambda^\frac{1}{p}\right) \leq C_{n,s_0,p,m,\Lambda}. 
\end{equation*}
In particular, bearing in mind~\eqref{gliorteitbergf}, we have that
\begin{equation*}
\left[u\right]_{B,\alpha}\leq C_{n,s_0,p,m,\Lambda}.
\end{equation*} 
Thanks to this last estimate and Theorem~\ref{CapVieABall} we conclude the proof for the case $n\geq 2$.

Now, it is only left to show the case $n=1$. We observe that $u\in W_{\textit{loc}}^{s,p}(\Omega)$ and $sp>1$. Then, thanks to this and Theorem~8.2 in~\cite{MR2944369}, we obtain that $u\in C_{\textit{loc}}^\alpha(\Omega)$ for $\alpha:=\frac{sp-1}{p}$.

In particular, we can use Theorem~\ref{basdcfzrty65} and the fact that for every $x,y\in B_1(x_0)$ it holds that $\left|x-y\right|\leq 2$ to obtain that 
\begin{equation}\label{iertopperett}
\begin{split}
\frac{1}{2}\left[u\right]_{C^{s_0-\frac{1}{p}}(B_1(x_0))}^p &\leq \left[u\right]_{C^{s-\frac{1}{p}}(B_1(x_0))}^p\\
&\leq \frac{C_{p}(1-s)}{(sp-1)^{p-1}}\int_{B_2(x_0)}\int_{B_2(x_0)}\frac{\left|u(x)-u(y)\right|^p}{\left|x-y\right|^{n+sp}}\,dx\,dy\\
&\leq \frac{C_{p}}{(sp-1)^{p-1}} \mathcal{E}_s^p\left(u,B_R\right).
\end{split}
\end{equation}
Now, we make use of the energy bound in Theorem~\ref{qteanldeeo}, and together with~\eqref{iertopperett} we obtain that 
\begin{equation*}
\frac{1}{2}\left[u\right]_{C^{s_0-\frac{1}{p}}(B_1(x_0))}^p\leq \frac{\bar{C}_1}{s(sp-1)^{p}}\leq \frac{\bar{C}_1}{s_0(s_0p-1)^p},
\end{equation*} 
where $\bar{C}_1$ is a constant depending only on $p$, $m$ and $\Lambda$. Again, bearing in mind~\eqref{gliorteitbergf} we obtain that
\begin{equation*}
\left[u\right]_{B,\alpha} \leq C_{s_0,p,m,\Lambda}. 
\end{equation*}
Thanks to this last estimate and Theorem~\ref{CapVieABall} we conclude. 
\end{proof}

Now we prove Corollary~\ref{coro-09}. 

\begin{proof}[Proof of Corollary~\ref{coro-09}]
For simplicity of notation, for every $r\in \left[2,+\infty\right)$ we denote 
\begin{equation}\label{EFFE-enac}
F(r):=\frac{\bar{C}}{s}\begin{dcases}
\frac{r^{n-sp}}{1-sp}\quad &\mbox{if}\quad s\in \left(0,\frac{1}{p}\right)\\
r^{n-1}\log(r)\quad &\mbox{if}\quad s=\frac{1}{p}\\
\frac{r^{n-1}}{sp-1}\quad &\mbox{if}\quad s\in \left(\frac{1}{p},1\right),
\end{dcases}
\end{equation}
where $\bar{C}:=\bar{C}_{n,p,m,\Lambda}\in(0,+\infty)$ is the constant introduced in~\eqref{BFAOTE}. 

Thanks to~\eqref{condiWzione} and~\eqref{BFAOTE} we have that, for every $r\in [2,+\infty)$ such that $B_{r+2}\subset \Omega$,
\begin{equation}\label{diorama-manna}
\begin{split}
\int_{B_r\cap \left\lbrace \theta_*<u\leq \theta^* \right\rbrace}\left|1+u(x)\right|^m\,dx &\leq \lambda_{\theta^*}^{-1} \int_{B_r\cap \left\lbrace \theta_*<u\leq \theta^* \right\rbrace}W(u(x))\,dx\\
&\leq \lambda_{\theta^*}^{-1}\mathcal{E}_s^p(u,B_r)\leq \lambda_{\theta*}^{-1}F(r). 
\end{split}
\end{equation}

Also, for every $r\geq 2$, 
\begin{equation}\label{rangeerre}
4r\geq r+2. 
\end{equation}
Thus, applying Theorem~\ref{th:fracp>=2} to $u$ with $\epsilon=0$ and making use of~\eqref{diorama-manna} and~\eqref{rangeerre}, we obtain that there exist~$R^*:=R_{s,n,m,p,c_1,\theta_*,r_0}^*\in [r_0,+\infty)$ and $\tilde{c}:=\tilde{c}_{s,n,m,p,c_1,\theta_*,r_0,c_0}\in (0,1)$ such that, for every $r\geq \max \left\lbrace 2, R^* \right\rbrace$ satisfying $B_{4r}\subset \Omega$, 
\begin{equation*}
\begin{split}
\tilde{c}r^n &\leq c \int_{B_r\cap \left\lbrace \theta_*<u\leq \theta^* \right\rbrace} \left|1+u(x)\right|^m\,dx+\mathcal{L}^n\left(B_r\cap \left\lbrace u>\theta_2\right\rbrace\right)\\
&\leq \lambda_{\theta_*}^{-1}cF(r)+ \mathcal{L}^n\left(B_r\cap \left\lbrace u>\theta_2\right\rbrace\right).
\end{split}
\end{equation*}
From this and equation~\eqref{EFFE-enac}, the claim in~\eqref{benzina} readily follows.  

Now, if $s_0\in \left(\frac{1}{p},1\right)$ and $s\in \left[s_0,1\right)$ we can apply Corollary~\ref{bmrlsadcbl} to $u$. Therefore, making use of~\eqref{lietoca},~\eqref{EFFE-enac},~\eqref{diorama-manna} and~\eqref{rangeerre},
we deduce that there exist $R^*:=R_{s_0,n,m,p,c_1,\theta_*,r_0}^*\in\left[r_0,+\infty\right)$, $\tilde{c}:=\tilde{c}_{s_0,n,p,m,\Lambda,c_1,\theta_*,r_0,c_0}\in (0,1)$ and $c:=c_{m,\theta_*}\in(0,+\infty)$ such that, for every $r\geq \max\left\lbrace 2,R^* \right\rbrace$ satisfying $B_{4r}\subset\Omega$,
\begin{equation*}
\begin{split}
\tilde{c}r^n &\leq c \int_{B_r\cap \left\lbrace \theta_*<u\leq \theta^* \right\rbrace} \left|1+u(x)\right|^m\,dx+\mathcal{L}^n\left(B_r\cap \left\lbrace u>\theta_2\right\rbrace\right)\\
&\leq \lambda_{\theta_*}^{-1}c\frac{\bar{C}r^{n-1}}{s_0(s_0p-1)}+ \mathcal{L}^n\left(B_r\cap \left\lbrace u>\theta_2\right\rbrace\right).
\end{split}
\end{equation*}
From this equation, claim~\eqref{cala} plainly follows. 
\end{proof}

\section{Proof of Theorems~\ref{12c-dgbnbc5543} \&~\ref{viredghloprec45}}\label{nmmdqcstngiga}

The proof of Theorem~\ref{12c-dgbnbc5543} relies the $\Gamma$-convergence of the Gagliardo seminorm, which has already been discussed in~\cite{MR2033060} and~\cite{MR4544090}. In particular, we follow the proof of Theorem~2.1 in~\cite{MR4544090}, where the $\Gamma$-convergence of $\mathcal{K}_s^2$ is showed for traceless functions. Here minor adjustments are needed due to the different homogeneity $p\in (1,+\infty)$ of the seminorm and the relaxed condition on the trace. 

In what follows we will adopt the following notation. For any $x\in\R^n$ and $A\subset\R^n$ we define 
\begin{equation*}
d(x,A):=\inf_{y\in A}\left|x-y\right|.
\end{equation*} 
Furthermore, for any $\epsilon\in (0,+\infty)$  we set
\begin{equation}\label{setsepsilon}
\Omega_{\epsilon}^{+}:=\left\lbrace x\in\R^n\mbox{  s.t.  } d(x,\Omega)\leq \epsilon \right\rbrace\quad\mbox{and}\quad\Omega_{\epsilon}^{-}:=\left\lbrace x\in \Omega\mbox{  s.t.  }d(x,\partial \Omega)\geq \epsilon \right\rbrace.
\end{equation}
Note that $\Omega_{\epsilon}^{-}$ might be empty for large values of $\epsilon$. Furthermore, given any measurable $g:\Omega\to \R$ we adopt the following notation for its zero extension in $\R^n\setminus \Omega$
\begin{equation*}
\bar{g}(x):=\begin{dcases}
g(x)\quad &\mbox{if}\quad x\in\Omega\\
0\quad &\mbox{if}\quad x\in \R^n\setminus \Omega.
\end{dcases}
\end{equation*}
Moreover, for every  $h\in\R^n$ and $f:\R^n\to \R$ measurable we consider the translation operator defined by 
\begin{equation*}
\tau_h f(\cdot):=f(\cdot+h).
\end{equation*}

We are now ready for the proof of Theorem~\ref{12c-dgbnbc5543}, which will be divided into
its specific parts.

\begin{proof}[Proof of Theorem~\ref{12c-dgbnbc5543}(i)]
We make use of Proposition~\ref{hndvcfwslop8} and obtain that, for every $h\in\R^n$,
\begin{equation}\label{rdvsslcnfofcm}
\begin{split}
\left\| \tau_h u_{s_k} -u_{s_k}   \right\|_{L^p\left(\Omega_{\left|h\right|}^-\right)}^p&=\int_{\Omega_{\left|h\right|}^{-}}\left|u_{s_k}(h+y)-u_{s_k}(y)\right|^p\,dy\\ 
&\leq  \bar{C}_{n,p}\left|h\right|^{s_k p}(1-s_k)\int_{B_{\left|h\right|}}\int_{\Omega}\frac{\left|u_{s_k}(x+y)-u_{s_k}(x)\right|^p}{\left|y\right|^{n+s_k p}}\,dx\,dy\\
&\leq \bar{C}_{n,p}\left|h\right|^{s_k p}(1-s_k)\int_{\Omega_{\left|h\right|}^+}\int_{\Omega}\frac{\left|u_{s_k}(y)-u_{s_k}(x)\right|^p}{\left|x-y\right|^{n+s_kp}}\,dx\,dy\\
&\leq \bar{C}_{n,p}\left|h\right|^{s_k p}2M,
\end{split}
\end{equation}
where $\bar{C}_{n,p}\in(0,+\infty)$ is a constant depending only on $n$ and $p$. 

As a result,
\begin{equation}\label{v654bgr32}
\begin{split}
\left\|\tau_h \bar{u}_{s_k}-\bar{u}_{s_k}\right\|_{L^p(\R^n)}^p&=\int_{\R^n}\left|\bar{u}_{s_k}(h+y)-\bar{u}_{s_k}(y)\right|^p\,dy\\
&\leq \int_{\Omega_{\left|h\right|}^{-}}\left|u_{s_k}(h+y)-u_{s_k}(y)\right|^p\,dy+2^p\int_{\Omega_{\left|h\right|}^{+}\setminus \Omega_{\left|h\right|}^{-}}\left\|u_{s_k}\right\|_{L^\infty(\R^n)}\,dy\\    
&\leq \bar{C}_{n,p}\left|h\right|^{s_k p}2M +2^p\left|\Omega_{\left|h\right|}^{+}\setminus \Omega_{\left|h\right|}^{-}\right|.
\end{split}
\end{equation}
Also, we notice that, as a consequence of the Dominated Convergence Theorem, 
\begin{equation*}
\lim_{h\to 0^+} \left|\Omega_{\left|h\right|}^{+}\setminus \Omega_{\left|h\right|}^{-}\right|=0.
\end{equation*}
Thanks to this and~\eqref{v654bgr32}, we can apply the
Fr\'echet-Kolmogorov to $\bar{u}_{s_k}$, see for instance Theorem~4.26 in~\cite{MR2759829}, and obtain that there exists~$u\in L^p(\Omega)$ such that, up to a subsequence,
\begin{equation}\label{crider234}
\lim_{k\to +\infty}\left\|u_{s_k}-u\right\|_{L^p(\Omega)}=0.
\end{equation} 
Therefore, taking the limit as $k\to +\infty$ in~\eqref{rdvsslcnfofcm}, it follows from~\eqref{crider234} that, for every $h\in\R^n$,
\begin{equation*}
\left\|\tau_h u-u\right\|_{L^p\left(\Omega_{\left|h\right|}^{-}\right)}\leq \left(2M\bar{C}_{n,p}\right)^\frac{1}{p}\left|h\right|. 
\end{equation*}
In particular, employing Proposition~9.3  in~\cite{MR2759829} we deduce that $Du\in L^p(\Omega)$ and therefore $u\in W^{1,p}(\Omega)$. 

Now we assume that~$u_{s_k}\in X_0^{s_k,p}(\Omega)$ for every $k\in\N$. Then, $u_{s_k}\equiv \bar{u}_{s_k}$. In particular, we obtain that 
\begin{equation*}
\mathcal{K}_{s_k}^p(u_{s_k},\Omega)=\frac{1}{2}\left[u_{s_k}\right]_{W^{s_k,p}(\Omega)}^p. 
\end{equation*}
It follows from this,~\eqref{kkkkfvetrre} and Proposition~\ref{hndvcfwslop8} that, for every open and bounded $\Omega'\subset\R^n$ and $h\in\R^n$, 
\begin{equation}\label{mpddmaepytduba}
\left\| \tau_h u_{s_k}-u_{s_k}  \right\|_{L^p(\Omega')} \leq \left(2\bar{C}_{n,p}M\right)^\frac{1}{p}\left|h\right|^{s_k}. 
\end{equation}

Consequently, by Theorem~4.26 in~\cite{MR2759829} it holds that $u_{s_k}\to u$ in $L^p(\R^n)$, where $u\equiv 0$ in $\R^n\setminus \Omega$. By taking the limit as~$k\to +\infty$ in~\eqref{mpddmaepytduba}, we obtain  that
\begin{equation*}
\left\| \tau_h u-u \right\|_{L^p(\Omega')}\leq \left(2\bar{C}_{n,p}M\right)^\frac{1}{p} \left|h\right|. 
\end{equation*}  
From this and Proposition~9.3 in~\cite{MR2759829} it follows that $Du\in L^p(\R^n)$. Since $u\equiv 0$ in $\R^n\setminus \Omega$ and $\partial \Omega$ is Lipschitz, we find that $u\in W_0^{1,p}(\Omega)$. This concludes the proof of Theorem~\ref{12c-dgbnbc5543}(i).
\end{proof}

\begin{proof}[Proof of Theorem~\ref{12c-dgbnbc5543}(ii)]
Since $W\in L^\infty((-1,1))$, with no loss of generality we can suppose that there exists~$M_1\in(0,+\infty)$ such that
\begin{equation}\label{ibfve54904}
\liminf_{k\to +\infty}(1-s_k)\mathcal{K}_{s_k}^p(u_{s_k},\Omega)\leq M_1.
\end{equation} 
In particular, up to a subsequence, we can assume that~\eqref{kkkkfvetrre} holds true. Hence, it follows from Theorem~\ref{12c-dgbnbc5543}(i) that $u\in W^{1,p}(\Omega)$.  

Now, let $v \in L^p(\R^n)$ and $\eta\in C_c^\infty(B_1)$ be a standard mollifier such that
\begin{equation*}
\eta\geq 0\quad\mbox{and}\quad \int_{B_1}\eta(x)\,dx=1.
\end{equation*}
Also, for $\epsilon\in (0,1)$ we consider $\eta_\epsilon(\cdot):=\eta \left(\frac{\cdot}{\epsilon}\right)$ and  we define the convolution $v_\epsilon:=v\ast \eta_\epsilon$. 

Then, we make the following claims. For every $s\in(0,1)$ and $p\in(1,+\infty)$, 
\begin{equation}\label{Claim1}
\mathcal{K}_{s}^p(v,\Omega)\geq \frac{1}{2}\int_{\Omega_{\epsilon}^{-}}\int_{\R^n}\frac{\left|v_\epsilon(x)-v_\epsilon(y)\right|^p}{\left|x-y\right|^{n+s p}}\,dy\,dx
\end{equation}
and
\begin{equation}\label{Claim2}
\frac{\left|\partial B_1\right|}{p}\liminf_{k\to +\infty} \int_{\Omega_{\epsilon}^{-}} \left|\nabla u_{s_k,\epsilon}(x)\right|^p \,dx\leq \liminf_{k\to +\infty}(1-s_k) \int_{\Omega_\epsilon^{-}}\int_{\R^n} \frac{\left|u_{s_k,\epsilon}(x)-u_{s_k,\epsilon}(y)\right|^p}{\left|x-y\right|^{n+s_kp}}\,dx\,dy. 
\end{equation}
Suppose for the moment that claims~\eqref{Claim1} and~\eqref{Claim2} hold true. Then, we obtain that 
\begin{equation}\label{copertone}
\begin{split}
\liminf_{k\to +\infty} (1-s_k)\mathcal{K}_{s_k}^p(u_{s_k},\Omega) &\geq \frac{\left|\partial B_1\right|}{2p} \liminf_{k\to +\infty}\int_{\Omega_\epsilon^{-}} \left|\nabla u_{s_k,\epsilon}\right|^p\,dx. 
\end{split}
\end{equation}
Also, we observe that $u_{s_k,\epsilon}\to u_\epsilon:=u\ast \eta_\epsilon$ in $L^p(\Omega_{\epsilon}^-)$. Now, according to~\eqref{ibfve54904} and~\eqref{copertone}, we obtain that up to a subsequence $u_{s_k,\epsilon}$ is  uniformly bounded in $W^{1,p}(\Omega_{\epsilon}^{-})$ with respect to $k$. 

For this reason, again up to a subsequence, we see that
\begin{equation}\label{galoe}
u_{s_k,\epsilon}\rightharpoonup u_{\epsilon} \quad\mbox{weakly in}\quad W^{1,p}(\Omega_{\epsilon}^{-}).  
\end{equation} 
Hence, we infer from~\eqref{copertone} and~\eqref{galoe} that 
\begin{equation*}
\liminf_{k\to +\infty} \mathcal{K}_{s_k}^p(u_{s_k},\Omega)\geq \frac{\left|\partial B_1\right|}{2p}\int_{\Omega_{\epsilon}^{-}}\left|\nabla u_{\epsilon}\right|^p\,dx. 
\end{equation*}
Since $u\in W^{1,p}(\Omega)$, taking the limit with respect to $\epsilon\to 0^+$ we obtain that 
\begin{equation}\label{arglba}
\liminf_{k\to +\infty}(1-s_k) \mathcal{K}_{s_k}^p(u_{s_k},\Omega)\geq\frac{\left|\partial B_1\right|}{2p}\int_{\Omega}\left|\nabla u(x)\right|^p\,dx=\mathcal{K}_1^p(u,\Omega).
\end{equation} 

We also notice that since $u_{s_k}\to u$ in $L^p(\Omega)$, then, up to a subsequence,
$u_{s_k}\to u$ a.e. in $\Omega$. Since $\left|u_{s_k}\right|\leq 1$ a.e. in $\Omega$ it follows that $\left|u\right|\leq 1$ a.e. in $\Omega$. Besides, making use of Fatou's Lemma and of
the lower semicontinuity of $W$ (see equation~\eqref{po12345}) we obtain that
\begin{equation}\label{iifvd43slop}
\liminf_{k\to +\infty}\int_{\Omega} W(u_{s_k}(x))\,dx\geq \int_{\Omega} \liminf_{k\to +\infty} W(u_{s_k}(x))\,dx\geq \int_{\Omega} W(u(x))\,dx.  
\end{equation}
Finally, from~\eqref{arglba} and~\eqref{iifvd43slop} we obtain~\eqref{utrc5342}. 

Hence, it is only left to show claims~\eqref{Claim1} and~\eqref{Claim2}.

To prove~\eqref{Claim1} we argue as follows.
Using Jensen inequality we obtain that 
\begin{equation*}
\begin{split}
&\frac{1}{2}\int_{\Omega_{\epsilon}^{-}}\int_{\R^n}\frac{\left|v_\epsilon(x)-v_\epsilon(y)\right|^p}{\left|x-y\right|^{n+sp}}\,dy\,dx\\
= &\frac{1}{2}\int_{\Omega_{\epsilon}^{-}}\int_{\R^n} \frac{\left|\int_{\R^n}\left(v(x-z)-v(y-z)\right)\eta_\epsilon(z)\,dz\right|^p}{\left|x-y\right|^{n+sp}}\,dx\,dy\\
\leq  &\frac{1}{2}\int_{\Omega_{\epsilon}^{-}}\int_{\R^n}\int_{\R^n} \frac{\left|v(x-z)-v(y-z)\right|^p}{\left|x-y\right|^{n+sp}}\eta_{\epsilon}(z)\,dz\,dx\,dy\\
= &\frac{1}{2}\int_{B_\epsilon} \int_{-z+\Omega_{\epsilon}^{-}}\int_{\R^n}\frac{\left|v(x)-v(y)\right|^p}{\left|x-y\right|^{n+sp}}\eta_{\epsilon}(z)\,dx\,dy\,dz\\
\leq & \frac{1}{2}\int_{B_\epsilon} \int_{\Omega}\int_{\R^n}\frac{\left|v(x)-v(y)\right|^p}{\left|x-y\right|^{n+sp}}\eta_{\epsilon}(z)\,dx\,dy\,dz\\
=&\frac{1}{2} \int_{\Omega}\int_{\Omega}\frac{\left|v(x)-v(y)\right|^p}{\left|x-y\right|^{n+sp}}\,dx\,dy+ \frac{1}{2}\int_{\Omega}\int_{\R^n\setminus \Omega}\frac{\left|v(x)-v(y)\right|^p}{\left|x-y\right|^{n+sp}}\,dx\,dy\\
\leq & \mathcal{K}_s^p(v,\Omega).
\end{split}
\end{equation*}

The proof of~\eqref{Claim1} is thereby complete and we now turn to the
proof of~\eqref{Claim2}. For this,
let $x,y\in \R^n$ and $\left|x-y\right|<1$. Using Taylor's Theorem with the Lagrange remainder, we obtain that, for some $\xi \in B_{1}(x)$,
\begin{equation}\label{tedodoro}
\begin{split}
\left|u_{s_k,\epsilon}(x)-u_{s_k,\epsilon}(y)\right|^p=&\left|\nabla u_{s_k,\epsilon} (x) \cdot(x-y)+D^2u_{s_k,\epsilon}(\xi)\cdot(x-y)^2 \right|^p\\
\geq &\left|\left|\nabla u_{s_k,\epsilon} (x) \cdot(x-y)\right|-\left|D^2u_{s_k,\epsilon}(\xi)\cdot(x-y)^2 \right|\right|^p.
\end{split}
\end{equation}
Now, we distinguish between the cases \begin{equation}\label{ca;1}
\left|\nabla u_{s_k,\epsilon} (x) \cdot(x-y)\right|\leq \left|D^2u_{s_k\epsilon}(\xi)\cdot(x-y)^2 \right|\end{equation}
and \begin{equation}\label{ca;2}
\left|\nabla u_{s_k,\epsilon} (x) \cdot(x-y)\right|\geq \left|D^2u_{s_k\epsilon}(\xi)\cdot(x-y)^2 \right|.\end{equation}

When~\eqref{ca;1} holds true, making use of the convexity of $\R^+\ni t \to t^p$ and~\eqref{tedodoro} we obtain that 
\begin{equation*}
\begin{split}
\left|u_{s_k,\epsilon}(x)-u_{s_k,\epsilon}(y)\right|^p &\geq \left|\nabla u_{s_k,\epsilon} (x) \cdot(x-y)\right|^p-\left|D^2u_{s_k,\epsilon}(\xi)\cdot(x-y)^2 \right|^p\\
&\geq \left|\nabla u_{s_k,\epsilon} (x) \cdot\frac{(x-y)}{\left|x-y\right|}\right|^p\left|x-y\right|^p-\left\|u_{s_k,\epsilon}\right\|_{C^2(\R^n)}^p\left|x-y\right|^{2p}
\end{split}
\end{equation*}
Now, since $\left\|u_{s_k}\right\|_{L^\infty(\R^n)}\leq 1$, we obtain that 
\begin{equation}\label{ohgbrelmvdfre36}
\left\|u_{s_k,\epsilon}\right\|_{C^2(\R^n)}\leq \left\|\eta_\epsilon \right\|_{W^{2,1}(\R^n)}.
\end{equation} 
As a result, if $\left|\nabla u_{s_k,\epsilon} (x) \cdot(x-y)\right|\leq \left|D^2u_{s_k,\epsilon}(\xi)\cdot(x-y)^2 \right|$ we obtain that
\begin{equation}\label{Euno1}
\begin{split}
\left|u_{s_k,\epsilon}(x)-u_{s_k,\epsilon}(y)\right|^p &\geq \left|\nabla u_{s_k,\epsilon} (x) \cdot\frac{(x-y)}{\left|x-y\right|}\right|^p\left|x-y\right|^p-\left\|\eta_\epsilon\right\|_{W^{2,1}(\R^n)}^p\left|x-y\right|^{2p}\\
&\geq \left|\nabla u_{s_k,\epsilon} (x) \cdot\frac{(x-y)}{\left|x-y\right|}\right|^p\left|x-y\right|^p-\left\|\eta_\epsilon\right\|_{W^{2,1}(\R^n)}^p\left|x-y\right|^{p+1},
\end{split}
\end{equation}as desired.

If instead~\eqref{ca;2} holds true, we use equation~\eqref{Bolo}  to estimate  
\begin{equation}\label{unrefdvgt56432}
\left|u_{s_k,\epsilon}(x)-u_{s_k,\epsilon}(y)\right|^p \geq \sum_{j=0}^{+\infty}\begin{pmatrix}p\\j \end{pmatrix}\left|\nabla u_{s_k,\epsilon} (x) \cdot(x-y)\right|^{p-j}\left|-D^2u_{s_k,\epsilon}(\xi)\cdot(x-y)^2 \right|^{j}. 
\end{equation}
Consequently,
\begin{equation}\label{iknfdcerty}
\begin{split}
&\sum_{j=0}^{+\infty} \begin{pmatrix}p\\j \end{pmatrix}\left|\nabla u_{s_k,\epsilon} (x) \cdot(x-y)\right|^{p-j}\left|-D^2u_{s_k,\epsilon}(\xi)\cdot(x-y)^2 \right|^{j}\\
\geq &\left|\nabla u_{s_k,\epsilon} (x) \cdot\frac{(x-y)}{\left|x-y\right|}\right|^p\left|x-y\right|^p-\sum_{j=1}^{[p]}\left| \begin{pmatrix}p\\j \end{pmatrix}\right|\left|\nabla u_{s_k,\epsilon} (x) \cdot(x-y)\right|^{p-j}\left|-D^2u_{s_k,\epsilon}(\xi)\cdot(x-y)^2 \right|^{j}\\
&-\sum_{j=[p]+1}^{+\infty}\left| \begin{pmatrix}p\\j \end{pmatrix}\right|\left|\nabla u_{s_k,\epsilon} (x) \cdot(x-y)\right|^{p-j}\left|-D^2u_{s_k,\epsilon}(\xi)\cdot(x-y)^2 \right|^{j}.
\end{split}
\end{equation}
Moreover, we observe that
\begin{equation}\label{KGBFTERFD}
\begin{split}
\sum_{j=1}^{[p]} &\left| \begin{pmatrix}p\\j \end{pmatrix}\right|\left|\nabla u_{s_k,\epsilon} (x) \cdot(x-y)\right|^{p-j}\left|-D^2u_{s_k,\epsilon}(\xi)\cdot(x-y)^2 \right|^{j} \\
\leq & \sum_{j=1}^{\left[p\right]} \left|\begin{pmatrix}p\\j \end{pmatrix}\right|\left|\nabla u_{s_k,\epsilon} (x) \cdot(x-y)\right|^{p-j}\left\|u_{s_k\epsilon}\right\|_{C^2(\R^n)}^{j}\left|x-y\right|^{2j}\\
\leq & \sum_{j=1}^{\left[p\right]} \left|\begin{pmatrix}p\\j \end{pmatrix}\right|\left\| u_{s_k,\epsilon}\right\|_{C^1(\R^n)}^{p-j} \left\|u_{s_k,\epsilon}\right\|_{C^2(\R^n)}^{j}\left|x-y\right|^{p+j}\\
\leq & \left\| u_{s_k,\epsilon}\right\|_{C^2(\R^n)}^{p}\left|x-y\right|^{p+1}\sum_{j=1}^{\left[p\right]} \left|\begin{pmatrix}p\\j \end{pmatrix}\right|.
\end{split}
\end{equation}
Also,
\begin{equation}\label{BBB-UJNER}
\begin{split}
&\sum_{j=\left[p\right]+1}^{+\infty} \left|\begin{pmatrix}p\\j \end{pmatrix}\right|\left|\nabla u_{s_k,\epsilon} (x) \cdot(x-y)\right|^{p-j}\left|D^2u_{s_k\epsilon}(\xi)\cdot(x-y)^2 \right|^{j}\\
\leq & \sum_{j=\left[p\right]+1}^{+\infty}\left|\begin{pmatrix}p\\j \end{pmatrix}\right|\left|D^2u_{s_k,\epsilon}(\xi)\cdot(x-y)^2 \right|^{p}\\
\leq & \left\|u_{s_k,\epsilon} \right\|_{C^2(\R^n)}^{p}\left|x-y\right|^{2p}\sum_{j=\left[p\right]+1}^{+\infty}\left|\begin{pmatrix}p\\j \end{pmatrix}\right|. 
\end{split}
\end{equation}
Therefore, if we set 
\begin{equation*}
C_{p}:=\max \left\lbrace \sum_{j=\left[p\right]+1}^{+\infty}\left|\begin{pmatrix} p\\j \end{pmatrix}\right|,\sum_{j=1}^{\left[p\right]} \left|\begin{pmatrix}p\\j \end{pmatrix}\right|\right\rbrace\quad\mbox{and}\quad C_{p,\epsilon}:=2C_{p}\left\|\eta_\epsilon \right\|_{W^{2,1}(\R^n)}^p,
\end{equation*}
making use of~\eqref{ohgbrelmvdfre36},~\eqref{iknfdcerty},~\eqref{KGBFTERFD} and~\eqref{BBB-UJNER} we deduce that 
\begin{equation*}
\begin{split}
& \sum_{j=0}^{+\infty} \begin{pmatrix}p\\j \end{pmatrix}\left|\nabla u_{s_k,\epsilon} (x) \cdot(x-y)\right|^{p-j}\left|-D^2u_{s_k,\epsilon}(\xi)\cdot(x-y)^2 \right|^{j}\\
\geq & \left|\nabla u_{s_k,\epsilon} (x) \cdot\frac{(x-y)}{\left|x-y\right|}\right|^p\left|x-y\right|^p-C_{p}\left\| u_{s_k,\epsilon}\right\|_{C^2(\R^n)}^{p} \left(\left|x-y\right|^{p+1}+\left|x-y\right|^{2p}\right)\\
\geq & \left|\nabla u_{s_k,\epsilon} (x) \cdot\frac{(x-y)}{\left|x-y\right|}\right|^p\left|x-y\right|^p-2C_{p}\left\|\eta_\epsilon \right\|_{W^{2,1}(\R^n)}^p \left|x-y\right|^{p+1}\\
=&\left|\nabla u_{s_k,\epsilon} (x) \cdot\frac{(x-y)}{\left|x-y\right|}\right|^p\left|x-y\right|^p-C_{p,\epsilon} \left|x-y\right|^{p+1}.
\end{split}
\end{equation*} 
Hence, as a consequence of~\eqref{Euno1},~\eqref{unrefdvgt56432} and the fact that $C_p\geq 1$,
we have that, for every $x,y\in\R^n$ such that $\left|x-y\right|<1$,
\begin{equation*}
\left|u_{s_k,\epsilon}(x)-u_{s_k,\epsilon}(y)\right|^p\geq \left|\nabla u_{s_k,\epsilon} (x) \cdot\frac{(x-y)}{\left|x-y\right|}\right|^p\left|x-y\right|^p-C_{p,\epsilon} \left|x-y\right|^{p+1}.
\end{equation*}
Therefore, for every $x\in \Omega_{\epsilon}^{-}$,
\begin{equation*}
\begin{split}
(1-s_k)\int_{\R^n} \frac{\left|u_{s_k,\epsilon}(x)-u_{s_k,\epsilon}(y)\right|^p}{\left|x-y\right|^{n+s_k p}}\,dy \geq &  (1-s_k) \int_{B_1(x)} \frac{\left|u_{s_k,\epsilon}(x)-u_{s_k,\epsilon}(y)\right|^p}{\left|x-y\right|^{n+s_k p}}\,dy\\
\geq & (1-s_k)\int_{B_1(x)}\left|\nabla u_{s_k,\epsilon} (x) \cdot\frac{(x-y)}{\left|x-y\right|}\right|^p\left|x-y\right|^{p(1-s_k)-n}\,dy\\
&-(1-s_k)C_{p,\epsilon}\int_{B_1(x)}\frac{dy}{\left|x-y\right|^{n+s_k p-p-1}}. 
\end{split}
\end{equation*} 

We notice that, for every $z \in\partial B_1$,
\begin{equation*}
\int_{\partial B_1}\left|z\cdot \omega\right|^p\,dH_\omega^{n-1}= \left|z\right|^p\int_{\partial B_1}\left|\omega_1\right|^p\,dH_{\omega}^{n-1}=\left|z\right|^p K_{n,p}. 
\end{equation*}
We thereby deduce that 
\begin{equation*}
\begin{split}
&\int_{B_1(x)}\left|\nabla u_{s_k,\epsilon} (x) \cdot\frac{(x-y)}{\left|x-y\right|}\right|^p\left|x-y\right|^{p(1-s_k)-n}\,dy\\
=&\int_{0}^1\int_{\partial B_1} \left|\nabla u_{s_k,\epsilon} (x) \cdot \omega \right|^p t^{p(1-s_k)-1}\,dH_{\omega}^{n-1}\,dt\\
=&K_{n,p}\left|\nabla u_{s_k,\epsilon}(x)\right|^p\int_{0}^1 t^{p(1-s_k)-1}\,dt\\
=&\frac{K_{n,p}}{p(1-s_k)}\left|\nabla u_{s_k,\epsilon}(x)\right|^p.
\end{split}
\end{equation*}
Thus, for every $x\in \Omega_{\epsilon}^{-}$,
\begin{equation*}
(1-s_k)\int_{\R^n} \frac{\left|u_{s_k,\epsilon}(x)-u_{s_k,\epsilon}(y)\right|^p}{\left|x-y\right|^{n+s_k p}}\,dy\geq \frac{K_{n,p}}{p}\left|\nabla u_{s_k,\epsilon}(x)\right|^p-(1-s_k)C_{p,\epsilon}\frac{\left|\partial B_1\right|}{p-s_k p+1}.
\end{equation*}
We thus integrate both sides of the inequality in $x\in\Omega_{\epsilon}^-$ and we take the $\liminf$  to obtain 
\begin{equation*}
\liminf_{k\to +\infty} (1-s_k)\int_{\Omega_{\epsilon}^{-}} \int_{\R^n}\frac{\left|u_{s_k,\epsilon}(x)-u_{s_k,\epsilon}(y)\right|^p}{\left|x-y\right|^{n+s_kp}}\,dy\,dx\geq \liminf_{k\to +\infty}\frac{K_{n,p}}{p}\int_{\Omega_{\epsilon}^{-}} \left|\nabla u_{s_k,\epsilon}(x)\right|^p\,dx
\end{equation*}
and this concludes the proof of~\eqref{Claim2}.\qedhere 
\end{proof}

\begin{proof}[Proof of Theorem~\ref{12c-dgbnbc5543}(iii)]
We assume with no loss of generality that $u\in W^{1,p}(\Omega)$. Being $\Omega$ Lipschitz, we extend $u$ to a function $\tilde{u}\in W^{1,p}(\R^n)$ such that $\tilde{u}=0$ in $\R^n\setminus B_{\frac{R}{3}}$ where $R\in (0,+\infty)$ is chosen such that $\Omega\subset B_{\frac{R}{3}}$. Then, we define 
\begin{equation*}
v:=\min\left\lbrace \max\left\lbrace \tilde{u},-1 \right\rbrace, 1  \right\rbrace. 
\end{equation*}
Notice that, by Stampacchia's Theorem,
$v\in W^{1,p}(\R^n)$. Also, $v=0$ in $B_{\frac{R}{3}}$. Then, $v$ satisfies all the assumptions of Proposition~\ref{prop-cont}. Furthermore, we have that $v\in X^{s,p}(\Omega)$ for every $s\in(0,1)$.    

Therefore, we can choose $u_{s_k}\equiv v$ and, thanks to Proposition~\ref{prop-cont}, we obtain that  
\begin{equation}\label{gddgtre5t4b}
\lim_{k\to +\infty}(1-s_k)\int_{\Omega}\int_{\R^n} \frac{\left|v(x)-v(y)\right|^p}{\left|x-y\right|^{n+s_kp}}\,dy\,dx =\frac{K_{n,p}}{p}\int_{\Omega}\left|\nabla v(x)\right|^p\,dx.
\end{equation}
Furthermore, since $u_{s_k}\equiv u$ in $\Omega$ for every $k$, we obtain 
\begin{equation}\label{oinalrev43}
\int_{\Omega}W(u_{s_k}(x))\,dx=\int_{\Omega} W(u(x))\,dx.
\end{equation}
From~\eqref{gddgtre5t4b} and~\eqref{oinalrev43} we conclude the proof of~Theorem~\ref{12c-dgbnbc5543}(iii). 
\end{proof}

As a consequence of the density estimates in Theorem~\ref{th:fracp>=2} and the $\Gamma$-convergence of the functional $\mathcal{E}_s^p$ to the free energy $\mathcal{E}_1^p$, we obtain density estimates for $\mathcal{E}_1^p$ as precisely stated in Theorem~\ref{viredghloprec45}.

\begin{proof}[Proof of Theorem~\ref{viredghloprec45}]
According to Theorem~\ref{12c-dgbnbc5543}(iii), there exists an extension $v:\R^n\to [-1,1]$ belonging to $X^{s,p}(\Omega)$ for every $s\in(0,1)$, such that
\begin{equation}\label{servcloybegr}
v\equiv u \quad\mbox{  in  }\quad \Omega
\end{equation}
and for every $\left\lbrace s_k \right\rbrace_{k\in\N}\subset (0,1)$  such that $s_k\to 1$ it holds that
\begin{equation}\label{erre0}
\mathcal{E}_1^p(u,\Omega)=\limsup_{k\to+\infty}\mathcal{E}_{s_k}^p(v,\Omega).   
\end{equation} 
From now on we let $\left\lbrace s_{k} \right\rbrace_{k\in\N}\subset(0,1)$ satisfying~$s_k\to 1$ be fixed. Moreover, up to a subsequence, we can assume that $\left\lbrace s_k \right\rbrace_{k\in\N}$ is increasing. 

Then, we claim that for every for every $\epsilon\in (0,+\infty)$ there exists a subsequence of $\left\lbrace s_k \right\rbrace_{k\in\N}$, which will not be renamed, such that, for every $h\in X_v^{s_k,p}$, 
\begin{equation}\label{onvegdret54}
\mathcal{E}_{s_k}^p(v,\Omega)\leq \epsilon+\mathcal{E}_{s_k}^p(h,\Omega). 
\end{equation}
To show~\eqref{onvegdret54}, given $k\in\N$ we consider $w_{s_k}\in X_v^{s_k,p}(\Omega)$ such that, for every $h \in X_v^{s_k,p}(\Omega)$,
\begin{equation}\label{minofwsk}
\mathcal{E}_{s_k}^p(w_{s_k},\Omega)\leq \mathcal{E}_{s_k}^p(h,\Omega).
\end{equation}
The existence of $w_{s_k}$ follows from the direct method. In particular, from this and~\eqref{erre0} it follows that 
\begin{equation}\label{feclopdgbetw6}
\mathcal{E}_{1}^p(u,\Omega)\geq \limsup_{k\to+\infty} \mathcal{E}_{s_k}^p(w_{s_k},\Omega).
\end{equation}
From this last inequality and~Theorem~\ref{12c-dgbnbc5543}(i)~and~(ii) we deduce that there exists~$w \in X^{1,p}(\Omega)$ such that $w_{s_k}\to w$ in $L^p(\Omega)$ and 
\begin{equation}\label{gliorte}
\mathcal{E}_1^p(w,\Omega)\leq \liminf_{k\to+\infty}\mathcal{E}_{s_k}^p(w_{s_k},\Omega). 
\end{equation}
Now, employing~\eqref{feclopdgbetw6} and~\eqref{gliorte} we deduce that 
\begin{equation}\label{revcghftr}
\mathcal{E}_1^p(u,\Omega)\geq \mathcal{E}_1^p(w,\Omega). 
\end{equation}
Also, as a consequence of Theorem~\ref{12c-dgbnbc5543}(i), we infer that
\begin{equation}\label{boundarycond}
u-w\in W_0^{1,p}(\Omega).
\end{equation}
From~\eqref{boundarycond},~\eqref{revcghftr} and the minimality of $u$ we obtain that 
\begin{equation*}
\mathcal{E}_1^p(u,\Omega)=\mathcal{E}_1^p(w,\Omega).
\end{equation*} 
Making use of this identity,~\eqref{erre0} and~\eqref{gliorte} we obtain that 
\begin{equation*}
\limsup_{k\to+\infty}\mathcal{E}_{s_k}^p(v,\Omega)\leq \liminf_{k\to+\infty} \mathcal{E}_{s_k}^p(w_{s_k},\Omega). 
\end{equation*}
As a consequence, bearing in mind~\eqref{minofwsk}, we deduce that~\eqref{onvegdret54} holds true and this concludes the proof of claim~\eqref{onvegdret54}. 

Now, being $u$ a minimizer for $\mathcal{E}_1^p$  in $X^{1,p}(\Omega)$, we can apply Theorem~3.1 in~\cite{MR0666107} and Theorem~6.1 in~\cite{MR0244627} and obtain that
\begin{equation}\label{Hold-54tgdbfe}
u\in C_{\textit{loc}}^\alpha(\Omega),
\end{equation}
for some $\alpha\in (0,1)$ depending only on $n$, $p$ and $m$.

Since $\Omega\neq \R^n$, we have that $C_{\textit{loc}}^\alpha(\Omega)=C_B^\alpha(\Omega)$. From this,~\eqref{Hold-54tgdbfe} and~\eqref{servcloybegr} we obtain that 
\begin{equation}\label{loperto}
v\in C_B^\alpha(\Omega)\quad\mbox{with $\alpha\in(0,1)$ depending only on $n$, $p$ and $m$.} 
\end{equation}

Given~$s_0\in\left(\frac{1}{p},1\right)$, we set 
\begin{equation*}
k_0:=\min \left\lbrace k\in\N\mbox{  s.t.  } s_k>s_0 \right\rbrace. 
\end{equation*}
Notice that being $\left\lbrace s_k \right\rbrace_{k\in\N}$ increasing, we have that~$s_k>s_0$
for every $k\geq k_0$. 
Then, we choose a subsequence of $\left\lbrace  s_k  \right\rbrace_{k\geq k_0}$ such that~\eqref{onvegdret54} is satisfied for 
$\epsilon:=\frac{\delta}{2}$, where $\delta:=\delta_{s_0,n,p,mc_1,\theta_*,c_0,\alpha, \left[v\right]_{C_B^\alpha(\Omega)}}\in(0,+\infty)$ is the constant in equation~\eqref{cala}. 

Therefore, by~\eqref{xldvaxknterv543},~\eqref{servcloybegr},~\eqref{loperto},~\eqref{onvegdret54} and the choice of $\epsilon$ and $\delta$, we obtain that all the hypothesis to apply Theorem~\ref{CapVieABall} are satisfied. In virtue of~\eqref{benz.1} we obtain that there exist $\tilde{c}:=\tilde{c}_{s_0,n,p,m,\Lambda,c_1,\theta_*,r_0,c_0,\alpha,\left[u\right]_{B,\alpha}}\in (0,1)$ and $R^{*}:=R_{s_0,n,m,p,c_1,\theta_*,c_0}^{*}\in \left[r_0,+\infty\right)$ such that, for every $r\in \left[R^*,+\infty \right)$ satisfying $B_{4r}\subset \Omega$,
\begin{equation}\label{StWr}
c_{m,\theta_*}\int_{B_r\cap \left\lbrace\theta_*<v\leq \theta^*\right\rbrace}\left|1+v(x)\right|^m\,dx+\mathcal{L}^n(B_r\cap \left\lbrace v>\theta_2 \right\rbrace)\geq \tilde{c}r^n. 
\end{equation}
To conclude, we also notice that,
since $\alpha$ depends only on $n$, $p$ and $m$,  
\begin{equation*}
\left[v\right]_{C_B^\alpha(\Omega)}\quad\mbox{depends only on $n$, $p$, $m$ and $\Omega$}.   
\end{equation*} On this account, recalling~\eqref{loperto}, we can rewrite $\delta$ and $\tilde{c}$ as
\begin{equation*}
\delta=\delta_{s_0,n,p,m,c_1,\theta_*,c_0}\quad\mbox{and}\quad \tilde{c}=\tilde{c}_{s_0,n,p,m,\Lambda,c_1,\theta_*,r_0,c_0}.  
\end{equation*}  Hence,
owing to~\eqref{StWr} and the fact that the choice of $s_0\in \left(\frac{1}{p},1\right)$ is arbitrary, we conclude the proof of Theorem~\ref{viredghloprec45}. 
\end{proof}

\section{Proof of Theorem~\ref{qteanldeeo}}\label{motifenzi}
The following result is based on Theorem~1.3 in~\cite{MR3133422}. The main difference is that here we consider the case~$p\neq 2$ and we write explicitly the dependence of the constants with respect to $s\in(0,1)$, which will be important
to obtain stable estimates as~$s\to1$.  Also, for every $t\in\R$ we denote $t^{+}:=\max \left\lbrace t,0 \right\rbrace$.

\begin{proof}[Proof of Theorem~\ref{qteanldeeo}]
In what follows, given~$A,B\in(0,+\infty)$, we adopt the notation 
\begin{equation*}
A\lesssim B \quad\mbox{if}\quad A\leq C_{n,p,m,\Lambda}B
\end{equation*}
for a suitable constant $C_{n,p,m,\Lambda}\in (0,+\infty)$.
 
Moreover, for every measurable sets $E,F\subset \R^n$  and every
function $f:\R^n\to \R$, we set
\begin{equation*}
f(E,F):=\int_{E}\int_{F}\frac{\left|f(x)-f(y)\right|^p}{\left|x-y\right|^{n+sp}}\,dx\,dy.
\end{equation*}
We also define the barrier 
\begin{equation}\label{bar}
\psi(x):=-1+2\min \left\lbrace \left(\left|x\right|-R-1\right)^{+},1 \right\rbrace
\end{equation}
and the function
\begin{equation}\label{CliopCliop}
d(x):=\max\left\lbrace  R-\left|x\right|,1 \right\rbrace.
\end{equation}
Then, we claim that  
\begin{equation}\label{chelodfve6ujg3wm}
\mathcal{E}_s^p(\psi,B_{R+2})\lesssim  \begin{dcases}
\frac{R^{n-sp}}{s(1-sp)}\quad &\mbox{if}\quad s\in \left(0,\frac{1}{p}\right)\\
\frac{R^{n-1}\log(R)}{s}\quad &\mbox{if}\quad s=\frac{1}{p}\\
\frac{R^{n-1}}{s(sp-1)}\quad &\mbox{if}\quad s\in \left(\frac{1}{p},1\right).
\end{dcases}
\end{equation}
The proof of the claim is postponed for the moment, and we show first that if~\eqref{chelodfve6ujg3wm} holds true then~\eqref{BFAOTE} follows. To do so, we consider 
\begin{equation*}
v:=\min \left\lbrace u,\psi \right\rbrace. 
\end{equation*}
We notice that if we define 
$A:=\left\lbrace v=\psi \right\rbrace\cap B_{R+2}$ then we have that $B_{R+1}\subset A$. Also, it clearly holds that $u=v$ in $A^c$. Hence, we use the minimality of $u$ in $B_{R+2}$ and henceforth in $A$ to obtain that
\begin{equation}\label{JSTFC}
\begin{split}
\mathcal{E}_s^p(u,A)= & (1-s)\left(\frac{1}{2}u(A,A)+u(A,A^c)\right)+\int_{A}W(u)\,dx\\ 
\leq & (1-s)\left(\frac{1}{2}v(A,A)+v(A,A^c)\right)+\int_{A}W(v)\,dx\\
= &\mathcal{E}_s^p(v,A). 
\end{split}
\end{equation} 
Now, if $x\in A$ and $y\in A^c$ then $v(x)=\psi(x)\leq u(x)$ and $v(y)=u(y)\leq \psi(y)$, and thus 
\begin{equation*}
\left|v(x)-v(y)\right|\leq \max\left\lbrace \left|u(x)-u(y)\right|,\left|\psi(x)-\psi(y)\right| \right\rbrace.
\end{equation*}
On this account, we find that 
\begin{equation*}
v(A,A^c)\leq u(A,A^c)+\psi(A,A^c). 
\end{equation*}
Thus, making use of~\eqref{JSTFC}, we deduce that 
\begin{equation}\label{zona}
\begin{split}
(1-s)\frac{1}{2}u(A,A)+\int_{A}W(u)\,dx &\leq (1-s)\left(\frac{1}{2}\psi(A,A)+\psi(A,A^c)\right)+\int_{A}W(\psi)\,dx\\
&=\mathcal{E}_s^p(\psi,A)\\
&\leq \mathcal{E}_s^p(\psi,B_{R+2}). 
\end{split}
\end{equation}
In particular, since $B_{R+1}\subset A\subset B_{R+2}$, using~\eqref{zona} we infer that 
\begin{equation}\label{coag-098lut}
\begin{split}
\mathcal{E}_s^p(u,B_R)&=(1-s)\mathcal{K}_s^p(u,B_R)+\int_{B_R} W(u)\,dx\\
&\leq (1-s)\left(\frac{1}{2}u(B_{R+1},B_{R+1}) +u(B_{R},B_{R+1}^c)\right) +\int_{B_R} W(u)\,dx\\
&\leq (1-s)\left(\frac{1}{2}u(A,A) +u(B_{R},B_{R+1}^c)\right) +\int_{A} W(u)\,dx\\
&\leq  \mathcal{E}_s^p(\psi,B_{R+2})+(1-s)u(B_{R},B_{R+1}^c).
\end{split}
\end{equation}
 
Furthermore, we notice that 
\begin{equation}\label{fropppitre}
\int_{0}^{R-1}t^{n-1}(R-t)^{-sp}\leq \begin{dcases}
\frac{R^{n-1}}{1-sp} \left(R^{1-sp}-1\right)\quad &\mbox{if}\quad s\in \left(0,\frac{1}{p}\right)\\
R^{n-1} \left(\log(R)-\log(1)\right)\quad &\mbox{if}\quad s=\frac{1}{p}\\
\frac{R^{n-1}}{sp-1}(1-R^{1-sp})\quad &\mbox{if}\quad s\in \left(\frac{1}{p},1\right).
\end{dcases}
\end{equation}
As a result, we have that
\begin{equation*}
\begin{split}
(1-s)u(B_R,\R^{n}\setminus B_{R+1})&=(1-s)\int_{B_{R}}\int_{\R^n\setminus B_{R+1}}\frac{\left|u(x)-u(y)\right|^p}{\left|x-y\right|^{n+sp}}\,dy\,dx\\
&\leq 2^p(1-s)\int_{B_R}\int_{\R^n\setminus B_{R+1}(x)}\frac{dy}{\left|y\right|^{n+sp}}\,dx\\
&\leq 2^{p}(1-s)\int_{B_R}\int_{\R^n\setminus B_{d(x)}}\frac{dy}{\left|y\right|^{n+sp}}\,dx\\  
&=2^{p}(1-s) \omega_{n-1} \int_{B_R} d(x)^{-sp}\,dx\\
&=2^{p}(1-s) \omega_{n-1}^2\left(\int_{0}^{R-1}t^{n-1}\left(R-t\right)^{-sp}\,dt+\int_{R-1}^R t^{n-1}\,dt\right)\\
&\leq 2^{p}(1-s)\omega_{n-1}^{2}\begin{dcases}
\frac{R^{n-1}}{1-sp} \left(R^{1-sp}-1\right)+R^{n-1}\quad &\mbox{if}\quad s\in \left(0,\frac{1}{p}\right)\\
R^{n-1} \left(\log(R)-\log(1)\right)+R^{n-1}\quad &\mbox{if}\quad s=\frac{1}{p}\\
\frac{R^{n-1}}{sp-1}(1-R^{1-sp})+R^{n-1}\quad &\mbox{if}\quad s\in \left(\frac{1}{p},1\right)
\end{dcases}\\
&\leq C(1-s)\begin{dcases}
\frac{R^{n-sp}}{1-sp} &\mbox{if}\quad s\in \left(0,\frac{1}{p}\right)\\
R^{n-1}\log(R)&\mbox{if}\quad s=\frac{1}{p}\\
\frac{R^{n-1}}{sp-1}\quad &\mbox{if}\quad s\in \left(\frac{1}{p},1\right).
\end{dcases}\\
\end{split}
\end{equation*}
Making use of this,~\eqref{coag-098lut} and claim~\eqref{chelodfve6ujg3wm} we conclude the proof of~\eqref{BFAOTE}.

It remains however to check claim~\eqref{chelodfve6ujg3wm}.
To this end, we first show that
\begin{equation}\label{pnlfvbtdcsfclm}
\left|\psi(x)-\psi(y)\right|\leq \begin{dcases}
2d(x)^{-1}\left|x-y\right|\quad &\mbox{if}\quad \left|x-y\right|<d(x)\\
2\quad &\mbox{if}\quad \left|x-y\right|\geq d(x).
\end{dcases}
\end{equation}
To prove this, we notice that, for every $x,y\in\R^n$,
\begin{equation}\label{tgbrfvedc4443}
\left|\psi(x)-\psi(y)\right|=2\left|\min\left\lbrace \left(\left|x\right|-R-1\right)^{+}, 1 \right\rbrace-\min\left\lbrace \left(\left|y\right|-R-1\right)^{+}, 1\right\rbrace\right|.
\end{equation}
Also, 
\begin{equation}\label{ted-hyt-trgh}
\min\left\lbrace \left(\left|x\right|-R-1\right)^{+}, 1   \right\rbrace=\begin{dcases}
0\quad &\mbox{if}\quad x\in B_{R+1}\\
\left|x\right|-R-1\quad &\mbox{if}\quad x\in B_{R+2}\setminus B_{R+1}\\
1\quad &\mbox{if}\quad x\in \R^{n}\setminus B_{R+2}. 
\end{dcases}
\end{equation}
Now, to prove~\eqref{pnlfvbtdcsfclm}, we distinguish between the cases \begin{equation}\label{ted-hyt-trgh:CASO1}
\left|x-y\right|<d(x)\end{equation} and \begin{equation}\label{ted-hyt-trgh:CASO2}\left|x-y\right|\geq d(x).\end{equation} 

Suppose first that~\eqref{ted-hyt-trgh:CASO1} holds true.
Thus, if $x\in B_{R-1}$ we deduce from~\eqref{CliopCliop} that $d(x)=R-\left|x\right|$ and accordingly $\left|x-y\right|<R-\left|x\right|$,
which in turn gives that~$y\in B_R$. 

From this,~\eqref{ted-hyt-trgh} and~\eqref{tgbrfvedc4443} it follows that  
\begin{equation}\label{lopwrfv}
\left|\psi(x)-\psi(y)\right|=0. 
\end{equation}

Also, when $x\in B_{R+1}\setminus B_{R-1}$ we infer from~\eqref{CliopCliop} that 
$d(x)=1$ and thus $\left|x-y\right|<1$. It thereby follows that $y\in B_{R+2}\setminus B_{R-1}$ and then, in light of~\eqref{tgbrfvedc4443} and~\eqref{ted-hyt-trgh}, 
\begin{equation}\label{reveleoep7}
\begin{split}
\left|\psi(x)-\psi(y)\right| & = \begin{dcases}
 0             \quad &\mbox{if}\quad y\in B_{R+1}  \setminus  B_{R-1}\\
 2\left(\left|y\right|-R-1\right) \quad &\mbox{if}\quad y\in B_{R+2}\setminus B_{R+1}
\end{dcases}\\
& \leq \begin{dcases}
 0             \quad &\mbox{if}\quad y\in B_{R+1}  \setminus  B_{R-1}\\
2\left(\left|y\right|-\left|x\right|\right) \quad &\mbox{if}\quad y\in B_{R+2}\setminus B_{R+1}
\end{dcases}\\
&\leq \begin{dcases}
 0             \quad &\mbox{if}\quad y\in B_{R+1}  \setminus  B_{R-1}\\
2\left|x-y\right| \quad &\mbox{if}\quad y\in B_{R+2}\setminus B_{R+1}
\end{dcases}\\
&=\begin{dcases}
 0             \quad &\mbox{if}\quad y\in B_{R+1}  \setminus  B_{R-1}\\
2d^{-1}(x) \left|x-y\right|\quad &\mbox{if}\quad y\in B_{R+2}\setminus B_{R+1}.
\end{dcases}
\end{split}
\end{equation}

When instead~$x\in B_{R+2}\setminus B_{R+1}$, we infer from~\eqref{CliopCliop} that $d(x)=1$ and thus $\left|x-y\right|<1$. This yields that $y\in B_{R+3}\setminus B_R$. Then, by virtue of~\eqref{tgbrfvedc4443}  and~\eqref{ted-hyt-trgh}, we see that
\begin{equation}\label{jjjyyytge}
\begin{split}
\left|\psi(x)-\psi(y)\right|& = \begin{dcases} 
2\left(\left|x\right|-R-1\right)\quad &\mbox{if}\quad y\in B_{R+1}\setminus B_R\\
2\left|\left|x\right|-\left|y\right|\right| \quad &\mbox{if}\quad y\in B_{R+2}\setminus B_{R+1}\\  
2\left|\left|x\right|-R-2\right|\quad &\mbox{if}\quad y\in B_{R+3}\setminus B_{R+2}
\end{dcases}\\
& \leq \begin{dcases} 
2\left(\left|x\right|-\left|y\right|\right)\quad &\mbox{if}\quad y\in B_{R+1}\setminus B_R\\
2\left|x-y\right| \quad &\mbox{if}\quad y\in B_{R+2}\setminus B_{R+1}\\  
2\left(R+2-\left|x\right|\right)\quad &\mbox{if}\quad y\in B_{R+3}\setminus B_{R+2}
\end{dcases}\\
&\leq \begin{dcases} 
2\left|x-y\right|\quad &\mbox{if}\quad y\in B_{R+1}\setminus B_R\\
2d^{-1}(x) \left|x-y\right|\quad &\mbox{if}\quad y\in B_{R+2}\setminus B_{R+1}\\  
2\left(\left|y\right|-\left|x\right|\right)\quad &\mbox{if}\quad y\in B_{R+3}\setminus B_{R+2}
\end{dcases}\\
&=\begin{dcases} 
2d^{-1}(x) \left|x-y\right|\quad &\mbox{if}\quad y\in B_{R+1}\setminus B_R\\
2d^{-1}(x) \left|x-y\right|\quad &\mbox{if}\quad y\in B_{R+2}\setminus B_{R+1}\\  
2d^{-1}(x) \left|x-y\right|\quad &\mbox{if}\quad y\in B_{R+3}\setminus B_{R+2}. 
\end{dcases}
\end{split}
\end{equation}

If instead $x\in \R^{n}\setminus B_{R+2}$, we use that $\left|x-y\right|<d(x)=1$
and accordingly~$y\in \R^n\setminus B_{R+1}$. From this, we arrive at
\begin{equation*}
\begin{split}
\left|\psi(x)-\psi(y)\right|&=\begin{dcases}
2\left(R+2-\left|y\right|\right)\quad &\mbox{if}\quad y\in B_{R+2}\setminus B_{R+1}\\
0\quad &\mbox{if}\quad y\in \R^n\setminus B_{R+2}
\end{dcases}\\
&\leq \begin{dcases}
2\left(\left|x\right|-\left|y\right|\right)\quad &\mbox{if}\quad y\in B_{R+2}\setminus B_{R+1}\\
0\quad &\mbox{if}\quad y\in \R^n\setminus B_{R+2}
\end{dcases}\\
&\leq \begin{dcases}
2\left|x-y\right|\quad &\mbox{if}\quad y\in B_{R+2}\setminus B_{R+1}\\
0\quad &\mbox{if}\quad y\in \R^n\setminus B_{R+2}
\end{dcases}\\
&=\begin{dcases}
2d^{-1}(x) \left|x-y\right|\quad &\mbox{if}\quad y\in B_{R+2}\setminus B_{R+1}\\
0\quad &\mbox{if}\quad y\in \R^n\setminus B_{R+2}.
\end{dcases}\\
\end{split}
\end{equation*}

By combining the latter estimate with~\eqref{lopwrfv},~\eqref{reveleoep7} and~\eqref{jjjyyytge}, we deduce that~\eqref{pnlfvbtdcsfclm} holds true in this case.

We now focus on the case given by~\eqref{ted-hyt-trgh:CASO1}.
In this situation, it follows from~\eqref{tgbrfvedc4443} that 
\begin{equation*}
\left|\psi(x)-\psi(y)\right|\leq 2
\end{equation*}
and this concludes the proof of~\eqref{pnlfvbtdcsfclm}.

Now, using~\eqref{pnlfvbtdcsfclm}, given $x\in\R^n$ we see that
\begin{equation*}
\begin{split}
&\int_{\R^n}\frac{\left|\psi(x)-\psi(y)\right|^p}{\left|x-y\right|^{n+sp}}\,dy\\
\lesssim & \left(\int_{B_{d(x)}(x)} \frac{d(x)^{-p}}{\left|x-y\right|^{n+sp-p}}\,dy+\int_{\R^n\setminus B_{d(x)}(x)} \frac{dy}{\left|x-y\right|^{n+sp}}\right)\\
\lesssim & \left( d(x)^{-p} \int_{0}^{d(x)}r^{-1-p+sp}\,dr+\int_{d(x)}^{+\infty} r^{-1-sp}\,dr\right) \\
=& \frac{d(x)^{-sp}}{p-sp} + \frac{d(x)^{-sp}}{sp} \\  
=& \frac{d(x)^{-sp}}{sp(1-s)},
\end{split}
\end{equation*}
From this, integrating with respect to $x\in B_{R+2}$, we obtain that
\begin{equation*}
\begin{split}
\mathcal{K}_s^p(\psi,B_{R+2}) &\leq \int_{B_{R+2}}\int_{\R^n}\frac{\left|\psi(x)-\psi(y)\right|^p}{\left|x-y\right|^{n+sp}}\,dy\,dx\\
&\lesssim\frac{1}{s(1-s)}\int_{B_{R+2}}d(x)^{-sp}\,dx\\
&= \frac{1}{s(1-s)}\left(\left|B_{R+2}\setminus B_{R-1}\right| +\int_{B_{R-1}}\left(R-\left|x\right|\right)^{-sp}\,dx\right)\\
&\lesssim \frac{1}{s(1-s)}\left( R^{n-1}+ \int_{0}^{R-1}t^{n-1}\left(R-t\right)^{-sp}\,dt\right).
\end{split}
\end{equation*}

Therefore, recalling~\eqref{fropppitre}, we conclude  that, for every $R\geq 2$, 
\begin{equation}\label{kb-0954g6}
\mathcal{K}_s^p(\psi,B_{R+2})\lesssim \frac{1}{s(1-s)}\begin{dcases}
\frac{R^{n-1}}{1-sp} \left(R^{1-sp}-1\right)\quad &\mbox{if}\quad s\in \left(0,1/p\right)\\
R^{n-1} \left(\log(R)-\log(1)\right)\quad &\mbox{if}\quad s=1/p\\
\frac{R^{n-1}}{sp-1}(1-R^{1-sp})\quad &\mbox{if}\quad s\in \left(1/p,1\right).
\end{dcases}
\end{equation} 
Finally, using~\eqref{potential},~\eqref{bar} and~\eqref{ted-hyt-trgh} we obtain that 
\begin{equation}\label{qt-ujgtre}
\begin{split}
\int_{B_{R+2}}W(\psi(x))\,dx &=\int_{B_{R+2}\setminus B_{R+1}}W(\psi(x))\,dx\\
&\lesssim  \int_{B_{R+2}\setminus B_{R+1}}\left|1+\psi(x)\right|^m\,dx\\
&\lesssim R^{n-1}.
\end{split}
\end{equation}
Therefore, from~\eqref{kb-0954g6} and~\eqref{qt-ujgtre} we obtain~\eqref{chelodfve6ujg3wm}, as desired.\end{proof}

\begin{appendix}
\begin{center}\section*{Appendices}
\end{center}

\section{A recursive formula}
For $x\in [-1,1]$, $m\in(1,+\infty)$ and $k\in \left\lbrace 1,\dots, [m] \right\rbrace$  we define the polynomials 
\begin{equation*}
\begin{dcases}
P_m^1(x):=-2xm\\
P_{m}^k(x):=P_{m-k+1}^{1}(x)P_{m}^{k-1}(x)+(1-x^2){P_{m}^{k-1}}'(x).
\end{dcases}
\end{equation*}
Then, we have the following result: 
\begin{prop}\label{colllise}
Let $m\in (1,+\infty)$ and $W_m(x):=(1-x^2)^m$ with $x\in[-1,1]$. Then, for every $k\in \left\lbrace 1,\dots, [m] \right\rbrace$,
\begin{equation}\label{telenious}
W_m^{(k)}(x)=W_{m-k}(x)P_m^k(x).
\end{equation}

Also, for every $m\in (1,+\infty)$ and $c\in (0,1)$ there exists~$\hat{q}_{m,c}\in (0,1)$ such that, for every $k\in \left\lbrace 1,\dots, \left[m\right] \right\rbrace$ and $x\in [-1,-1+\hat{q}_{m,c}]$, 
\begin{equation}\label{term-0987}
P_m^k(x)\geq c.
\end{equation}
In particular, $W_m(x)$ satisfies~\eqref{potntiald2} for some $c_1\in (0,+\infty)$ and $q\in (0,1)$. 
\end{prop}

\begin{proof} 
We notice that, for every $m\in (1,+\infty)$,
\begin{equation}\label{klkp.mtv5}
W_m'(x)=-2xm(1-x^2)^{m-1}=P_m^1(x)W_{m-1}(x).
\end{equation}

Let us now proceed recursively and assume that $m\in [2,+\infty)$ and that~\eqref{telenious} holds for every $j\leq k-1<\left[m\right]$. Then, 
\begin{equation*}
\begin{split}
W_m^{(k)}(x)&=\left(W_m^{(k-1)}(x)\right)'\\
&=\left(W_{m-k+1}(x)P_m^{k-1}(x)\right)'\\
&=W_{m-k+1}'(x)P_m^{k-1}(x)+W_{m-k+1}(x){P_m^{k-1}}'(x)\\
&=P_{m-k+1}^1(x)W_{m-k}(x)P_m^{k-1}(x)+W_{m-k+1}(x){P_m^{k-1}}'(x)\\
&=W_{m-k}(x)\left(P_{m-k+1}^1(x)P_m^{k-1}(x)+(1-x^2){P_m^{k-1}}'(x)\right)\\
&=W_{m-k}(x)P_m^k(x).
\end{split}
\end{equation*}
This concludes the proof of~\eqref{telenious}.

Now we prove~\eqref{term-0987}. For this, we consider $0<q_{\left[m\right]}\leq \dots\leq q_1<1$ which will be chosen as follows. 

For any $m\in (1,+\infty)$, we choose $q_1:=1-\frac{c}{2m}$, from which it follows that, for every $x\in \left[-1,-1+q_1\right]$,
\begin{equation*}
P_m^1(x)=-2xm\geq -2m(-1+q_1)=2m(1-q_1)=c.
\end{equation*}
Now we assume that $P_m^j(x)\geq c$ for every $j\leq k-1<\left[m\right]$ and $x\in \left[-1,-1+q_j\right]$. We need to show that there exists~$q_k\in\left(0, q_{k-1}\right)$ such that $P_m^k(x)\geq c$ for every $x\in [-1,-1+q_k]$. In order to do so, we let 
\begin{equation}
q_k:=\min \left\lbrace q_{k-1}, \frac{2c(m-k+1)-c}{2c(m-k+1)+2\left\|{P_m^{k-1}}'\right\|_{L^\infty([-1,0])}} \right\rbrace
\end{equation}
and, for every $x\in [-1,-1+q_k]\subset [-1,-1+q_{k-1}]$,
\begin{equation*}
\begin{split}
P_{m}^k(x)&=P_{m-k+1}^1(x)P_m^{k-1}(x) +(1-x^2){P_m^{k-1}}'(x)\\
&\geq P_{m-k+1}^1(x)P_m^{k-1}(x)-(1-x^2)\left\|{P_m^{k-1}}'\right\|_{L^\infty([-1,0])}\\
& \geq 2c(1-q_k)(m-k+1)-2q_k\left\|{P_m^{k-1}}'\right\|_{L^\infty([-1,0])}\\
&\geq c.
\end{split}
\end{equation*}
In particular, by choosing $\hat{q}_{c,m}:=q_{\left[m\right]}$ we establish the validity of~\eqref{term-0987}, as desired.

Thus, as a consequence of~\eqref{telenious} and~\eqref{term-0987},
we find that 
\begin{equation*}
\begin{split}
W_m(t)-W_m(r)&=\int_{r}^tW_m'(\xi)\,d\xi\\
&=\sum_{j=1}^{k-1} \frac{W_m^{(j)}(r)}{j!}(t-r)^j+\int_{r}^t\dots \int_{r}^l W_m^{(k)}(z)\,dz\dots\,d\xi\\
&\geq W_m^{(1)}(r)(t-r)  + \frac{c}{k!}(t-r)^k\\
&=(1-r^2)^{m-1}P_1^m(r)(t-r)+\frac{c}{k!}(t-r)^k\\
&\geq c(1-r)^{m-1}(1+r)^{m-1}(t-r)+\frac{c}{k!}(t-r)^k\\
&\geq c(2-\hat{q}_{c,m})^{m-1}(1+r)^{m-1}(t-r)+\frac{c}{k!}(t-r)^k,
\end{split}
\end{equation*}
as long as $m\in \N$, $c\in (0,1)$, $-1\leq r\leq t\leq -1+\hat{q}_{c,m}$ and $k=m$.

This shows that $W_m(x)$ satisfies~\eqref{potntiald2} for every $m\in\N$  with $q:=\hat{q}_{c,m}$ and 
\begin{equation}\label{c_1}
c_1:=\min\left\lbrace  c(2-q)^{m-1} ,\frac{c}{k!}    \right\rbrace.
\end{equation}

Additionally, if $m\in (1,+\infty)\setminus \N$ and $c\in (0,1)$ we define $q_1\in (0,1)$ as the unique solution to 
\begin{equation*}
2c(1-q_1)(m-\left[m\right])(2 q_1)^{m-\left[m\right]-1} =c+ (2 q_1)^{m-\left[m\right]}\left\|{P_m^{\left[m\right]}}'\right\|_{L^\infty([-1,0])}.
\end{equation*}
We observe that, for every $t\in (0,q_1)$, 
\begin{equation}\label{jvecogdv4}
2c(1-t)(m-\left[m\right])(2 t)^{m-\left[m\right]-1}>c+ (2 t)^{m-\left[m\right]}\left\|{P_m^{\left[m\right]}}'\right\|_{L^\infty([-1,0])}.
\end{equation}
Then, we set  
\begin{equation}\label{eq!}
q:=\min \left\lbrace \hat{q}_{c,m}, q_1  \right\rbrace
\end{equation}
and we deduce from~\eqref{jvecogdv4},~\eqref{telenious} and~\eqref{term-0987} that,
when $k=\left[m\right]+1$,
\begin{equation*}
\begin{split}
W_m^{(k)}(x)&=\left(W_m^{([m])}(x)\right)'\\
&=\left((1-x^2)^{m-\left[m\right]}P_m^{\left[m\right]}(x)\right)'\\
&=-2x(m-\left[m\right])\left(1-x^2\right)^{m-\left[m\right]-1}P_m^{\left[m\right]}(x)+ (1-x^2)^{m-\left[m\right]}{P_m^{\left[m\right]}}'(x)\\
&\geq 2c(1-\tilde{q})(m-\left[m\right])(2\tilde{q})^{m-\left[m\right]-1}-(2\tilde{q})^{m-\left[m\right]}\left\|{P_m^{\left[m\right]}}'\right\|_{L^\infty([-1,0])}\\
&\geq c,
\end{split}
\end{equation*}
for every $x\in [-1,-1+q]$

Because of this, we have that, for every $-1\leq r\leq t\leq -1+q$,
\begin{equation*}
\begin{split}
W_m(t)-W_m(r)&=\int_{r}^tW_m'(\xi)\,d\xi\\
&=\sum_{j=1}^{k-1} \frac{W_m^{(j)}(r)}{j!}(t-r)^j+\int_{r}^t\dots \int_{r}^l W_m^{(k)}(z)\,dz\dots\,d\xi\\
&\geq W_m^{(1)}(r)(t-r)  + \frac{c}{k!}(t-r)^k\\
&=(1-r^2)^{m-1}P_1^m(r)(t-r) \frac{c}{k!}(t-r)^k\\
&\geq c(1-r)^{m-1}(1+r)^{m-1}(t-r)+\frac{c}{k!}(t-r)^k\\
&\geq c(2-q)^{m-1}(1+r)^{m-1}(t-r)+\frac{c}{k!}(t-r)^k.
\end{split}
\end{equation*}
This proves that $W_{m}$ satisfies~\eqref{potntiald2} for every $m\in (1,+\infty)\setminus \N$ with $c_1$ and $q$ respectively as in~\eqref{c_1} and~\eqref{eq!}.  
\end{proof}

\section{Proof of Theorem~\ref{SymverofTh}}\label{Crocodilewalk}
In this section we prove Theorem~\ref{SymverofTh}. 
The estimate in Theorem~\ref{SymverofTh} is actually a consequence of the following integral inequality: 
\begin{prop}\label{AmGiDiMe}
Let $R,\sigma\in (0,+\infty)$. Then, there exists a constant $c_{n,\sigma}\in (0,+\infty)$
such that, for every $x\in B_{R}$, 
\begin{equation}\label{hloiacrca}
\int_{\R^n\setminus B_R}\frac{dy}{\left|x-y\right|^{n+\sigma}}\geq \frac{c_{n,\sigma}}{\sigma}(R-\left|x\right|)^{-\sigma}.
\end{equation}
Furthermore, if $\sigma\in(0,M)$ for some $M\in(0,+\infty)$, it holds that $c_{n,\sigma}\geq c_{n,M}$.
\end{prop}

\begin{proof}
To show~\eqref{hloiacrca}, we begin by observing that, for every $x\in B_R$,
\begin{equation}\label{symmetrymotherfucker}
\int_{\R^n\setminus B_R} \frac{dy}{\left|x-y\right|^{n+\sigma}}= \int_{\R^n\setminus B_R} \frac{dy}{\left|\left|x\right|e_1-y\right|^{n+\sigma}}.  
\end{equation}
Also, given $m>0$, we define
\begin{equation}\label{Tiemme}
T_m(\left|x\right|):=\left\lbrace z\in\R^n|\, z_1\in (R,+\infty) \mbox{  and  }  -m  (z_1-\left|x\right|)\leq z_{i} \leq m (z_1-\left|x\right|) \mbox{  for all  }i=2,\dots,n\right\rbrace,     
\end{equation}
and we notice that
\begin{equation}\label{sod}
T_m(\left|x\right|)\subset \R^{n}\setminus B_R.     
\end{equation}
Accordingly, thanks to~\eqref{symmetrymotherfucker},~\eqref{Tiemme} and~\eqref{sod} we obtain that 
\begin{equation*}
\begin{split}
\int_{\R^n\setminus B_R} \frac{dy}{\left|x-y\right|^{n+\sigma}}&\geq \int_{T_m(\left|x\right|)}\frac{dz}{\left|e_1\left|x\right|-z\right|^{n+\sigma}}\\
&= \int_{R}^{+\infty}\int_{-m(z_1-\left|x\right|)}^{m(z_1-\left|x\right|)}\dots\int_{-m(z_1-\left|x\right|)}^{m(z_1-\left|x\right|)}\frac{dz_{n}\dots d z_2}{\left|\left|x\right|e_1-z\right|^{n+\sigma}}\,dz_1.
\end{split}
\end{equation*}

Besides, we observe that, for every $z\in T_m(\left|x\right|)$,
\begin{equation}\label{cagaioredgh}
\begin{split}
\left| \left|x\right|e_1-z \right|^2&=\left(\left|x\right|-z_1\right)^2+z_2^2\dots +z_{n}^2\\
&\leq \left(\left|x\right|-z_1\right)^2+m^2\left(z_1-\left|x\right|\right)^{2}+\dots +m^2\left(1-z_1\right)^2 \\  
&= ((n-1)m^2+1)\left(z_1-\left|x\right|\right)^2.
\end{split}
\end{equation}
Thus, making use of~\eqref{sod} and~\eqref{cagaioredgh} we evince that  
\begin{equation}\label{sordina}
\begin{split}
\int_{\R^n\setminus B_R} \frac{dy}{\left|x-y\right|^{n+\sigma}}&\geq \frac{1}{((n-1)m^2+1)^{\frac{n+\sigma}{2}}}\int_{R}^{+\infty}\int_{-m(z_1-\left|x\right|)}^{m(z_1-\left|x\right|)}\dots\int_{-m(z_1-\left|x\right|)}^{m(z_1-\left|x\right|)}\frac{dz_{n}\dots d z_2}{\left|z_1-\left|x\right|\right|^{n+\sigma}}\,dz_1 \\ 
&= \frac{(2m)^{n-1}}{((n-1)m^2+1)^{\frac{n+1+s}{2}}}\int_{R}^{+\infty}\frac{(z_1-\left|x\right|)^{n-1}}{\left|z_1-\left|x\right|\right|^{n+\sigma}}\,dz_1 \\ 
&=\frac{2^{n} m^n}{(n m^2+1)^{\frac{n+1+s}{2}}}\int_{R}^{+\infty}\frac{d z_1}{\left|z_1-\left|x\right|\right|^{1+\sigma}}\,dz_1 \\ 
&=\frac{(2m)^{n-1}}{\sigma((n-1)m^2+1)^{\frac{n+\sigma}{2}}}(R-\left|x\right|)^{-s}\\
&\geq \frac{c_{n,\sigma}}{\sigma}(R-\left|x\right|)^{-s}.
\end{split}
\end{equation}
Also, we notice that if $\sigma\in (0,M)$ for some $M\in(0,+\infty)$, then $c_{n,\sigma}\geq c_{n,M}$. 
\end{proof}

Making use of Proposition~\ref{AmGiDiMe}, we now state and prove a lower bound  for $L(B_1,\R^n\setminus B_{1+r})$ when $r\in (0,+\infty)$ is small enough. Then, Theorem~\ref{SymverofTh} will follow by scaling from Lemma~\ref{gligliredght}.
    
\begin{lem}\label{gligliredght}
Let $p\in (1,+\infty)$, $s\in (0,1)$ and $\delta:=\min\left\lbrace 2^{-n}, 2^{-n+2-p} \right\rbrace$. 

Then, for every $r\in \left(0,\delta\right)$,
\begin{equation*}
L(B_1,\R^n\setminus B_{r+1})\geq\begin{dcases}
\delta\quad &\mbox{if}\quad s\in \left(0,\frac{1}{p}\right)\\
\delta\ln\left(\frac{1}{r}\right)\quad &\mbox{if}\quad s=\frac{1}{p}\\
\delta r^{1-sp}\quad &\mbox{if}\quad s\in \left(\frac{1}{p},1\right)
\end{dcases}  
\end{equation*}
\end{lem}

\begin{proof}
By~\eqref{hloiacrca}, we have that, for every $x\in B_1$, 
\begin{equation*}
\int_{\R^n\setminus B_{r+1}}\frac{dy}{\left|x-y\right|^{n+sp}}\geq \frac{c_{n,p}}{s} (r+1-\left|x\right|)^{-sp}. 
\end{equation*}
Hence, integrating both sides of the above inequality in $B_1$ we obtain that 
\begin{equation}\label{fcmlahnftrgflopt55554}
\begin{split}
L(B_1,\R^n\setminus B_{1+r})&=\int_{B_1}\int_{\R^n\setminus B_{1+r}}\frac{dy\,dx}{\left|x-y\right|^{n+sp}}\\
&\geq \frac{c_{n,p}}{s} \int_{B_1} (r+1-\left|x\right|)^{-sp}\,dx\\ 
&=\left|\partial B_1\right|\frac{c_{n,p}}{s}\int_{0}^1\frac{t^{n-1}}{(r+1-t)^{sp}}\,dt.
\end{split}
\end{equation}
Now, we need to distinguish the following three cases:
\begin{equation}\label{asdf-sp-capom-sE1}
sp\in(0,1),\end{equation}
\begin{equation}\label{asdf-sp-capom-sE2}
sp=1\end{equation}
and
\begin{equation}\label{asdf-sp-capom-sE3}sp\in (1,p).
\end{equation}

When~\eqref{asdf-sp-capom-sE1} holds true, we have that 
\begin{equation*}
\begin{split}
\int_{0}^1t^{n-1}(r+1-t)^{-sp}\,dt&\geq \int_{\frac{1}{2}}^1 t^{n-1}(r+1-t)^{-sp}\,dt\\
&\geq \left(\frac{1}{2}\right)^{n-1}\int_{\frac{1}{2}}^1 (r+1-t)^{-sp}\,dt\\
&=\left(\frac{1}{2}\right)^{n-1}\frac{1}{sp-1}(r+1-t)^{-sp+1}\bigg|_{\frac{1}{2}}^1\\
&=\left(\frac{1}{2}\right)^{n-1}\frac{1}{1-sp}\left[\left(r+\frac{1}{2}\right)^{1-sp}-r^{1-sp}\right]\\
&\geq \left(\frac{1}{2}\right)^{n-1}\frac{1}{1-sp}\left[1-\left(\frac{1}{2}\right)^{1-sp}\right]\\
&=\left(\frac{1}{2}\right)^{n-1}\frac{1-2^{-(1-sp)}}{1-sp}\\
&\geq \left(\frac{1}{2}\right)^n. 
\end{split}
\end{equation*}

Instead, when~\eqref{asdf-sp-capom-sE2} holds true, we obtain that 
\begin{equation*}
\begin{split}
\int_{0}^1t^{n-1}(r+1-t)^{-1}\,dt&\geq \int_{\frac{1}{2}}^1 t^{n-1}(r+1-t)^{-1}\,dt\\
&\geq \left(\frac{1}{2}\right)^{n-1}\int_{\frac{1}{2}}^1 (r+1-t)^{-1}\,dt\\
&=\left(\frac{1}{2}\right)^{n-1}\left|\ln(r+1-t)\right|\bigg|_{\frac{1}{2}}^1\\
&=\left(\frac{1}{2}\right)^{n-1}\left[\left|\ln(r)\right|-\left|\ln\left(r+1/2\right)\right|\right]\\
&=\left(\frac{1}{2}\right)^{n}2\ln\left(1+\frac{1}{2r}\right)\\
&\geq \left(\frac{1}{2}\right)^{n}\ln\left(\frac{1}{r}\right).
\end{split}
\end{equation*}

Finally, in case~\eqref{asdf-sp-capom-sE3},
\begin{equation*}
\begin{split}
\int_{0}^1t^{n-1}(r+1-t)^{-sp}\,dt&\geq \int_{\frac{1}{2}}^1 t^{n-1}(r+1-t)^{-sp}\,dt\\
&\geq \left(\frac{1}{2}\right)^{n-1}\int_{\frac{1}{2}}^1 (r+1-t)^{-sp}\,dt\\
&=\left(\frac{1}{2}\right)^{n-1}\frac{1}{sp-1}(r+1-t)^{-sp+1}\bigg|_{\frac{1}{2}}^1\\
&=\left(\frac{1}{2}\right)^{n-1}\frac{1}{sp-1}\left[\frac{1}{r^{sp-1}}-\frac{1}{\left(r+1/2\right)^{sp-1}}\right]\\
&\geq \left(\frac{1}{2}\right)^{n-1}\frac{1}{sp-1}\left[\frac{1}{r^{sp-1}}-\frac{1}{\left(2r\right)^{sp-1}}\right]\\
&=\frac{1}{2^{sp+n-2}}\left(\frac{2^{sp-1}-1}{sp-1}\right)r^{1-sp}\\
&\geq \frac{1}{2^{p+n-2}} r^{1-sp}.
\end{split}
\end{equation*}
These three cases complete the proof of the desired result.
\end{proof}

With this preliminary work, we are in a position to complete the proof of Theorem~\ref{SymverofTh}.

\begin{proof}[Proof of Theorem~\ref{SymverofTh}]
To prove~\eqref{Sonounoskianto} it is enough to rescale $L(B_R,\R^n\setminus B_{R+r})$. By doing so, we can apply Lemma~\ref{gligliredght} as follows:
\begin{equation*}
\begin{split}
L(B_R,\R^n \setminus B_{R+r})&=R^{n-sp} L(B_1,\R^n\setminus B_{1+\frac{r}{R}})\\
&\geq \begin{dcases}
\delta R^{n-sp}\quad &\mbox{if}\quad s\in\left(0,\frac{1}{p}\right)\\
\delta R^{n-1}  \ln \left(\frac{R}{r}\right)\quad &\mbox{if}\quad s=\frac{1}{p}\\
\delta R^{n-sp} \left(\frac{r}{R}\right)^{1-sp}\quad &\mbox{if}\quad s\in\left(\frac{1}{p},1\right),
\end{dcases}
\end{split}
\end{equation*}yielding the desired result.
\end{proof}

\section{Some regularity results}

In what follows we consider $w$ as given in Theorem~3.1 in~\cite{DFVERPP} with $\tau:=\mu$, where $\mu$ has been defined in equation~\eqref{muuuuuuu}. Also, we let $\bar{R}_\mu$ be as in~\eqref{latuse-invefr43}. Moreover, given any $f\in C_{\textit{loc}}^\alpha(\Omega)$, $K\subset\subset\Omega$ such that $\bigcup_{x_0\in K}B_2(x_0)\subset\Omega$ we denote 
\begin{equation*}
\sup_{x_0\in K} \left[f\right]_{C^\alpha(B_1(x_0))}:=\left[f\right]_{K,\alpha}. 
\end{equation*} 
Then, with these choices we have the following result:   

\begin{prop}\label{hiolevcgfdb4t54}
Let $s_0\in (0,1)$, $s\in [s_0,1)$, $\alpha\in (0,1)$ and $u\in C_{\textit{loc}}^\alpha(\Omega)$. Moreover, we let $v$ be defined in~\eqref{vdefini-tion} and $R\in \left[\bar{R}_{\mu},+\infty\right)$ satisfying $B_{R+2}\subset \Omega$. Then, $u-v\in C_{c}^\alpha(\Omega)$ and there exists~$C:=C_{s_0,n,m,p,c_1}\in (0,+\infty)$ such that, for every $x_0\in \R^n$, 
\begin{equation}\label{tebegdfjuybgf}
\left[u-v\right]_{C^\alpha(B_1(x_0))}\leq  4 \left[u\right]_{K,\alpha}+2C.
\end{equation} 
\end{prop}

\begin{proof}
First, we claim that there exists~$C_{s_0,n,m,p,c_1}\in (0,+\infty)$ such that, for a.e. $x\in\R^n$,
\begin{equation}\label{dcmbterfdbh}
\left|\nabla w(x)\right|\leq C_{s_0,n,m,p,c_1}. 
\end{equation}
To show~\eqref{dcmbterfdbh}, we recall that according to~(3.40) in~\cite{DFVERPP}, we have that 
\begin{equation}\label{zenolo}
w(x)=(2-\beta)h\left(\frac{\left|x\right|}{C_0}\right)+\beta-1,
\end{equation}
where the function $h$ and the constants $C_0$ and $\beta$ are given respectively in~(3.8) and~(3.39) in~\cite{DFVERPP}. Also, for every $R\in\left[\bar{R}_\mu,+\infty\right)$ according to~(3.39) and~(3.41) in~\cite{DFVERPP} it holds that
\begin{equation}\label{pinrepintre}
\beta\leq 2. 
\end{equation} 
Furthermore, according to~(3.39) in~\cite{DFVERPP} there exists a constant $c_{s_0,n,m,p,c_1}\in (0,+\infty)$  such that 
\begin{equation}\label{sfin}
c_{s_0,n,m,p,c_1}\leq C_0.
\end{equation}
Also, if $\hat{r},r\in (0,+\infty)$ are given respectively as in~(B.25) and~(2.10) in~\cite{DFVERPP}, then, for every $t\in \left(0,\frac{r}{2}\right]\cup \left(\hat{r},+\infty\right)$,
\begin{equation}\label{hlitrvb}
h'(t)=0. 
\end{equation}
Moreover, according to equations~(3.8),~(3.25),~(A.14) and Proposition~A.6 in~\cite{DFVERPP}, it holds that, for every $t \in  \left(\frac{r}{2},\hat{r}\right)$,
\begin{equation}\label{jijitref654321}
h'(t)\leq C_{p,m} (r-t)^{-qs-1}\leq C_{p,m} (r-\hat{r})^{-qs-1}\leq C_{p,m}\left(\widehat{C}_{p,m}\right)^{-qs-1}\leq \tilde{C}_{p,m}. 
\end{equation}
Therefore, putting together~\eqref{zenolo},~\eqref{pinrepintre},~\eqref{sfin},~\eqref{hlitrvb} and~\eqref{jijitref654321}, we deduce that, for a.e. $x\in\R^n$,
\begin{equation*}
\begin{split}
\left|\nabla w(x)\right|=\frac{2-\beta}{C_0}\left|h'\left(\frac{\left|x\right|}{C_0}\right)\right| \leq \frac{4}{c_{s_0,n,m,p,c_1}} \tilde{C}_{p,m}=:C_{s_0,n,m,p,c_1}. 
\end{split}
\end{equation*}  
This concludes the proof of claim~\eqref{dcmbterfdbh}. 

Furthermore, since $u\in C_{\textit{loc}}^\alpha(\Omega)$ and $B_{R+2}\subset\Omega$, we have that,
for every $x_0\in \overline{B}_R$,
\begin{equation}\label{horoitvcfe}
\left[u\right]_{C^\alpha(B_1(x_0))}\leq \left[u\right]_{C^\alpha\left(B_{R+\frac{3}{2}}\right)}.
\end{equation} 

Also, we recall that $v:=\min\left\lbrace u,w \right\rbrace$ and it is easy to verify that 
\begin{equation*}
\left|v(x)-v(y)\right|\leq \max\left\lbrace \left|u(x)-u(y)\right|, \left|w(x)-w(y)\right|\right\rbrace.
\end{equation*}
From this and~\eqref{dcmbterfdbh} it follows that, for every $x_0\in \overline{B}_R$,
\begin{equation}\label{hliotevgfb}
\left[v\right]_{C^{\alpha}(B_1(x_0))}\leq \max \left\lbrace  \left[u\right]_{C^\alpha\left(B_{R+\frac{3}{2}}\right)}, C_{s_0,n,m,p,c_1} \right\rbrace.
\end{equation}
Hence, recalling~\eqref{hliotevgfb}, we obtain that, for every $x_0\in \overline{B}_R$,
\begin{equation}\label{vnvcd-67123}
\left[u-v\right]_{C^\alpha(B_1(x_0))}\leq \left[u\right]_{C^\alpha(B_1(x_0))} +\left[v\right]_{C^\alpha(B_1(x_0))}\leq 2 \left[u\right]_{C^\alpha\left(B_{R+\frac{3}{2}}\right)}+C_{s_0,n,m,p,c_1}.
\end{equation}

Now, for every $x\in \R^n\setminus B_R$ we use the notation $x':=\frac{x}{\left|x\right|}R$. Also, in what follows we take $x_0\in\R^n\setminus \overline{B}_R$. Then, employing~\eqref{horoitvcfe},~\eqref{hliotevgfb},~\eqref{vnvcd-67123} and recalling that the support of~$u-v$ is contained in~$B_R$, we conclude that
\begin{equation*}
\begin{split}
\left[u-v\right]_{C^\alpha\left(B_1(x_0)\right)} :=&\sup_{x,y\in B_1(x_0)}\frac{\left|(u-v)(x)-(u-v)(y)\right|}{\left|x-y\right|^\alpha}\\
\leq & \sup_{x,y\in B_R\cap B_1(x_0)}\frac{\left|(u-v)(x)-(u-v)(y)\right|}{\left|x-y\right|^\alpha}+\sup_{\overset{x\in B_R\cap B_1(x_0)}{y\in B_1(x_0)\setminus B_R}}\frac{\left|(u-v)(x)-(u-v)(y)\right|}{\left|x-y\right|^\alpha}\\
\leq &\sup_{x,y\in B_R\cap B_1(x_0)}\frac{\left|u(x)-u(y)\right|}{\left|x-y\right|^\alpha}+\sup_{x,y\in B_R\cap B_1(x_0)}\frac{\left|v(x)-v(y)\right|}{\left|x-y\right|^\alpha}\\
&+ \sup_{\overset{x\in B_R\cap B_1(x_0)}{y\in B_1(x_0)\setminus B_R}}\frac{\left|(u-v)(x)-(u-v)(y)\right|}{\left|x-y\right|^\alpha}\\
\leq & \left[u\right]_{C^\alpha(B_1(x_0'))}+ \left[v\right]_{C^\alpha(B_1(x_0'))}+\sup_{\overset{x\in B_R\cap B_1(x_0)}{y\in B_1(x_0)\setminus B_R}}\frac{\left|(u-v)(x)-(u-v)(y')\right|}{\left|x-y'\right|^\alpha}\\
\leq &\left[u\right]_{C^\alpha(B_1(x_0'))}+\left[v\right]_{C^\alpha(B_1(x_0'))}+\left[u-v\right]_{C^\alpha(B_1(x_0'))}\\
\leq & 4\left[u\right]_{C^\alpha\left(B_{R+\frac{3}{2}}\right)}+ 2C_{s_0,n,m,p,c_1}.\qedhere
\end{split}
\end{equation*}

\end{proof}

\section{Technical results}
Here we collect some technical results that will be used in the proof of the $\Gamma$-convergence stated in Theorem~\ref{12c-dgbnbc5543}. 

First we point out a useful variation of a celebrated result in~\cite{MR3586796}. 
\begin{prop}\label{prop-cont}
Let $p\in (1,+\infty)$ and $\Omega\subset \R^n$ be bounded, open and smooth. Also, let $R\in (0,+\infty)$ be such that $\Omega\subset B_{\frac{R}{3}}$. 

Then, for every  $v\in W^{1,p}(\Omega)\cap L^\infty(\R^n\setminus B_R)\cap L^p(B_R)$,
\begin{equation}\label{zenonen}
\lim_{s\to 1}(1-s)\int_{\Omega}\int_{\R^n} \frac{\left|v(x)-v(y)\right|^p}{\left|x-y\right|^{n+sp}}\,dy\,dx =\frac{K_{n,p}}{p}\int_{\Omega}\left|\nabla v(x)\right|^p\,dx.
\end{equation}
\end{prop}

\begin{proof}
For any $\epsilon\in (0,1)$ small enough we define the sets $\Omega_{\epsilon}^{-}$ and $\Omega_{\epsilon}^{+}$ as in~\eqref{setsepsilon}. Also, without loss of generality we assume that $0\in \Omega$. 

According to~Corollary~2 in~\cite{MR3586796}, we have that, for every $u\in W^{1,p}(\Omega)$,
\begin{equation*}
\lim_{s\to 1}\int_{\Omega}\int_{\Omega} \frac{\left|v(x)-v(y)\right|^p}{\left|x-y\right|^{n+sp}}\,dy\,dx= \frac{K_{n,p}}{p}\int_{\Omega}\left|\nabla v\right|^p\,dx.
\end{equation*}

As a consequence, in order to prove~\eqref{zenonen} it is only left to show that 
\begin{equation}\label{holig6trec}
\lim_{s\to 1} (1-s)\int_{\Omega}\int_{\R^n\setminus \Omega} \frac{\left|v(x)-v(y)\right|^p}{\left|x-y\right|^{n+sp}}\,dy\,dx=0.
\end{equation}
To show~\eqref{holig6trec}, we first observe that  
\begin{equation}\label{Cbvdkbgrvc5663}
\begin{split}
&\int_{\Omega}\int_{\R^n\setminus \Omega} \frac{\left|v(x)-v(y)\right|^p}{\left|x-y\right|^{n+sp}}\,dy\,dx\\
=& \int_{\Omega_{\epsilon}^{-}}\int_{\R^n\setminus\Omega} \frac{\left|v(x)-v(y)\right|^p}{\left|x-y\right|^{n+sp}}\,dy\,dx+ \int_{\Omega\setminus \Omega_{\epsilon}^{-}}\int_{\R^n\setminus\Omega} \frac{\left|v(x)-v(y)\right|^p}{\left|x-y\right|^{n+sp}}\,dy\,dx\\
=&\int_{\Omega_{\epsilon}^{-}}\int_{\R^n\setminus\Omega_{\epsilon}^{+}} \frac{\left|v(x)-v(y)\right|^p}{\left|x-y\right|^{n+sp}}\,dy\,dx+ \int_{\Omega_{\epsilon}^{-}}\int_{\Omega_{\epsilon}^{+}\setminus \Omega} \frac{\left|v(x)-v(y)\right|^p}{\left|x-y\right|^{n+sp}}\,dy\,dx\\
+& \int_{\Omega\setminus \Omega_{\epsilon}^{-}}\int_{\R^n\setminus\Omega_{\epsilon}^{+}} \frac{\left|v(x)-v(y)\right|^p}{\left|x-y\right|^{n+sp}}\,dy\,dx+\int_{\Omega\setminus \Omega_{\epsilon}^{-}}\int_{\Omega_{\epsilon}^{+}\setminus \Omega } \frac{\left|v(x)-v(y)\right|^p}{\left|x-y\right|^{n+sp}}\,dy\,dx.
\end{split}
\end{equation}  

Now we define 
\begin{equation*}
\begin{split}
A&:= \int_{\Omega_{\epsilon}^{-}}\int_{\R^n\setminus\Omega_{\epsilon}^{+}} \frac{\left|v(x)-v(y)\right|^p}{\left|x-y\right|^{n+sp}}\,dy\,dx,\quad B:=\int_{\Omega_{\epsilon}^{-}}\int_{\Omega_{\epsilon}^{+}\setminus \Omega} \frac{\left|v(x)-v(y)\right|^p}{\left|x-y\right|^{n+sp}}\,dy\,dx,\\
C&:=\int_{\Omega\setminus \Omega_{\epsilon}^{-}}\int_{\R^n\setminus\Omega_{\epsilon}^{+}} \frac{\left|v(x)-v(y)\right|^p}{\left|x-y\right|^{n+sp}}\,dy\,dx\quad\mbox{and}\quad D:=\int_{\Omega\setminus \Omega_{\epsilon}^{-}}\int_{\Omega_{\epsilon}^{+}\setminus \Omega } \frac{\left|v(x)-v(y)\right|^p}{\left|x-y\right|^{n+sp}}\,dy\,dx.
\end{split}
\end{equation*}
Then, if $R\in (0,+\infty)$ is such that $\Omega_{\epsilon}^{+}\subset B_{\frac{R}{2}}$, we see that, for every $x\in \Omega_{\epsilon}^{-}$ and $y\in \R^n\setminus B_{R}$,
\begin{equation*}
\left|x-y\right|^{n+sp}\geq \left(\left|y\right|-\left|x\right|\right)^{n+sp}\geq \left(\frac{\left|y\right|}{2}\right)^{n+sp}.
\end{equation*} 
Hence, we obtain the existence of a constant $C_{n,p}$ such that
\begin{equation}\label{yecmlap}
\begin{split}
A =&\int_{\R^n\setminus B_{R}}\int_{\Omega_{\epsilon}^{-}}\frac{\left|v(x)-v(y)\right|^p}{\left|x-y\right|^{n+sp}}dx\,dy+\int_{B_{R}\setminus \Omega_{\epsilon}^{+}}\int_{\Omega_{\epsilon}^{-}}\frac{\left|v(x)-v(y)\right|^p}{\left|x-y\right|^{n+sp}}dx\,dy\\
\leq & \int_{\R^n\setminus B_{R}}\int_{\Omega_{\epsilon}^{-}}\frac{\left|v(x)-v(y)\right|^p}{\left|y/2\right|^{n+sp}}dx\,dy+\int_{B_{R}\setminus \Omega_{\epsilon}^{+}}\int_{\Omega_{\epsilon}^{-}}\frac{\left|v(x)-v(y)\right|^p}{(2\epsilon)^{n+sp}}dx\,dy\\
\leq & 2^{p-1}\int_{\R^n\setminus B_{R}}\int_{\Omega_{\epsilon}^{-}}\frac{\left|v(x)\right|^p+\left|v(y)\right|^p}{\left|y/2\right|^{n+sp}}dx\,dy+2^{p-1}\int_{B_{R}\setminus \Omega_{\epsilon}^{+}}\int_{\Omega_{\epsilon}^{-}}\frac{\left|v(x)\right|^p+\left|v(y)\right|^p}{(2\epsilon)^{n+sp}}dx\,dy\\ 
\leq & 2^{p-1}\int_{\R^n\setminus B_{R}}\frac{\left\|v\right\|_{L^p(\Omega_{\epsilon}^-)}^p+\left|\Omega_{\epsilon}^-\right|\left\|v\right\|_{L^\infty(\R^n\setminus B_R)}^p}{\left|y/2\right|^{n+sp}}dx\,dy+2^{p-1}\int_{B_{R}\setminus \Omega_{\epsilon}^{+}}\frac{\left\|v\right\|_{L^p(\Omega_{\epsilon}^-)}^p+\left|\Omega_{\epsilon}^-\right|\left|v(y)\right|^p}{(2\epsilon)^{n+sp}}\,dy\\ 
\leq & C_{n,p}\left(\left\|v\right\|_{L^p(\Omega_{\epsilon}^-)}^p+\left|\Omega_{\epsilon}^-\right|\left\|v\right\|_{L^\infty(\R^n\setminus B_R)}^p\right)\frac{R^{-sp}}{sp}+ C_{n,p} \epsilon^{-n-sp} \left(\left|B_R\setminus \Omega_{\epsilon}^{+}\right|\left\|v\right\|_{L^p(\Omega_{\epsilon}^-)}^p+\left|\Omega_{\epsilon}^-\right|\left\|v\right\|_{L^p(B_R\setminus \Omega_{\epsilon}^+)}^p\right).
\end{split}
\end{equation} 

Similarly, 
\begin{equation}\label{oldfvert}
\begin{split}
B \leq C_{n,p}\epsilon^{-n-sp}\left(\left|\Omega_{\epsilon}^-\right|\left\|v\right\|_{L^p(\Omega_{\epsilon}^+\setminus \Omega)}+\left|\Omega_{\epsilon}^+\setminus \Omega\right|\left\|v\right\|_{L^p(\Omega_{\epsilon}^-)}                 \right).
\end{split}
\end{equation}
 
Also, we can estimate $C$ as 
\begin{equation}\label{kelorfce}
\begin{split}
C\leq & C_{n,p}\left(\left\|v\right\|_{L^p(\Omega\setminus \Omega_{\epsilon}^-)}^p+\left|\Omega_{\epsilon}^-\right|\left\|v\right\|_{L^\infty(\R^n\setminus B_R)}^p\right)\frac{R^{-sp}}{sp}\\
&+ C_{n,p} \epsilon^{-n-sp} \left(\left|B_R\setminus \Omega_{\epsilon}^{+}\right|\left\|v\right\|_{L^p(\Omega\setminus \Omega_{\epsilon}^-)}^p+\left|\Omega\setminus \Omega_{\epsilon}^-\right|\left\|v\right\|_{L^p(B_R\setminus \Omega_{\epsilon}^+)}^p\right).
\end{split}
\end{equation}

Finally, we estimate $D$. To do so, we observe that 
\begin{equation}\label{guglielmo}
\begin{split}
D&=\int_{\Omega\setminus \Omega_{\epsilon}^{-}}\int_{\Omega_{\epsilon}^{+}\setminus \Omega } \frac{\left|v(x)-v(y)\right|^p}{\left|x-y\right|^{n+sp}}\,dy\,dx\\
&=\int_{\Omega\setminus \Omega_{\epsilon}^{-}}\int_{B_1(x)\cap \left(\Omega_{\epsilon}^{+}\setminus \Omega\right)} \frac{\left|v(x)-v(y)\right|^p}{\left|x-y\right|^{n+sp}}\,dy\,dx+ \int_{\Omega\setminus \Omega_{\epsilon}^{-}}\int_{B_1^c(x)\cap \left(\Omega_{\epsilon}^{+}\setminus \Omega\right)} \frac{\left|v(x)-v(y)\right|^p}{\left|x-y\right|^{n+sp}}\,dy\,dx.
\end{split}
\end{equation} 
Moreover, for a.e. $x,y\in \R^n$ it holds that 
\begin{equation*}
\frac{\left|v(x)-v(y)\right|}{\left|x-y\right|} \leq \int_{0}^1\left|\nabla v(y+t(x-y))\right|\,dt 
\end{equation*}

Therefore,
\begin{equation}\label{yety}
\begin{split}
&\int_{\Omega\setminus \Omega_{\epsilon}^{-}}\int_{B_1(x)\cap \left(\Omega_{\epsilon}^{+}\setminus \Omega\right)} \frac{\left|v(x)-v(y)\right|^p}{\left|x-y\right|^{n+sp}}\,dy\,dx\\
=&\int_{\Omega\setminus \Omega_{\epsilon}^{-}}\int_{B_1(x)\cap \left(\Omega_{\epsilon}^{+}\setminus \Omega\right)} \frac{\left|v(x)-v(y)\right|^p}{\left|x-y\right|^{p}}\frac{1}{\left|x-y\right|^{n+sp-p}}\,dy\,dx\\
\leq & \int_{\Omega\setminus \Omega_{\epsilon}^{-}}\int_{B_1(x)\cap \left(\Omega_{\epsilon}^{+}\setminus \Omega\right)} \left( \int_{0}^1\left|\nabla v(y+t(x-y))\right|\,dt     \right)^p\frac{1}{\left|x-y\right|^{n+sp-p}}\,dy\,dx\\
\leq &\int_{\Omega\setminus \Omega_{\epsilon}^{-}}\int_{B_1(x)\cap \left(\Omega_{\epsilon}^{+}\setminus \Omega\right)} \int_{0}^1\frac{\left|\nabla v(y+t(x-y))\right|^p}{\left|x-y\right|^{n+sp-p}}\,dt \,dy\,dx\\
\leq & \int_{\R^n}\int_{B_1} \int_{0}^1\frac{\left|\nabla v(x+(1-t)z)\right|^p}{\left|z\right|^{n+sp-p}}\chi_{\Omega\setminus \Omega_{\epsilon}^{-}}(x)\,dt \,dz\,dx\\
=&\int_{B_1} \int_{0}^1\int_{\R^n}\frac{\left|\nabla v(x)\right|^p}{\left|z\right|^{n+sp-p}}\chi_{\Omega\setminus \Omega_{\epsilon}^{-}}(x-(1-t)z)\,dx\,dt \,dz\\
=&\int_{B_1} \int_{0}^1\int_{\R^n}\frac{\left|\nabla v(x)\right|^p}{\left|z\right|^{n+sp-p}}\chi_{\Omega\setminus \Omega_{\epsilon}^{-}+(1-t)z}(x)\,dx\,dt \,dz\\
=&\int_{B_1} \int_{0}^1\int_{\R^n}\frac{\left\|\nabla v\right\|_{L^p(\Omega\setminus \Omega_{\epsilon}^{-}+(1-t)z)}^p}{\left|z\right|^{n+sp-p}}\,dt \,dz.
\end{split}
\end{equation}
Now, the function $f:\overline{B}_1\to \R$ defined by
$$f(q):=\left\|\nabla v\right\|_{L^p(\Omega\setminus \Omega_{\epsilon}^{-}+q)}^p$$
is continuous and therefore there exists~$\tilde{q}_\epsilon\in \overline{B}_1$ such that, for every $t\in (0,1)$ and $z\in B_1$,
\begin{equation*}
\left\|\nabla v \right\|_{L^p(\Omega\setminus \Omega_{\epsilon}^{-}+(1-t)z)}^p\leq f(\tilde{q}_\epsilon).
\end{equation*}

Hence, it follows from~\eqref{yety} that 
\begin{equation}\label{cresto}
\begin{split}
&\int_{\Omega\setminus \Omega_{\epsilon}^{-}}\int_{B_1(x)\cap \left(\Omega_{\epsilon}^{+}\setminus \Omega\right)} \frac{\left|v(x)-v(y)\right|^p}{\left|x-y\right|^{n+sp}}\,dy\,dx\\
\leq & \int_{B_1} \frac{\left\|\nabla v\right\|_{L^p(\Omega\setminus \Omega_{\epsilon}^{-}+\tilde{q}_\epsilon)}^p}{\left|z\right|^{n+sp-p}}\,dz\\
=& \frac{\left|\partial B_1\right|}{p-sp}\left\|\nabla v\right\|_{L^p(\Omega\setminus \Omega_{\epsilon}^{-}+\tilde{q}_\epsilon)}^p
\end{split}
\end{equation}

We also observe that
\begin{equation*}
\begin{split}
&\int_{\Omega\setminus \Omega_{\epsilon}^{-}}\int_{B_1^c(x)\cap \left(\Omega_{\epsilon}^{+}\setminus \Omega\right)} \frac{\left|v(x)-v(y)\right|^p}{\left|x-y\right|^{n+sp}}\,dy\,dx\\
\leq &\int_{\Omega\setminus \Omega_{\epsilon}^{-}}\int_{B_1^c(x)\cap \left(\Omega_{\epsilon}^{+}\setminus \Omega\right)}\left|v(x)-v(y)\right|^p\,dy\,dx\\
\leq & C_{n,p} \left\|v\right\|_{L^p(\Omega_{\epsilon}^{+})}, 
\end{split}
\end{equation*}
for a suitable constant $C_{n,p}$.  

Consequently, taking into account~\eqref{Cbvdkbgrvc5663},~\eqref{yecmlap},~\eqref{oldfvert},~\eqref{kelorfce},~\eqref{guglielmo}
and~\eqref{cresto}, we infer that 
\begin{equation*}
\begin{split}
(1-s)&\int_{\Omega}\int_{\R^n\setminus \Omega} \frac{\left|v(x)-v(y)\right|^p}{\left|x-y\right|^{n+sp}}\,dy\,dx\\ 
\leq (1-s) & C_{n,p}\left(\left\|v\right\|_{L^p(\Omega_{\epsilon}^-)}^p+\left|\Omega_{\epsilon}^-\right|\left\|v\right\|_{L^\infty(\R^n\setminus B_R)}^p\right)\frac{R^{-sp}}{sp}\\
+(1-s)& C_{n,p} \epsilon^{-n-sp} \left(\left|B_R\setminus \Omega_{\epsilon}^{+}\right|\left\|v\right\|_{L^p(\Omega_{\epsilon}^-)}^p+\left|\Omega_{\epsilon}^-\right|\left\|v\right\|_{L^p(B_R\setminus \Omega_{\epsilon}^+)}^p\right)\\
+(1-s) &C_{n,p}\epsilon^{-n-sp}\left(\left|\Omega_{\epsilon}^-\right|\left\|v\right\|_{L^p(\Omega_{\epsilon}^+\setminus \Omega)}+\left|\Omega_{\epsilon}^+\setminus \Omega\right|\left\|v\right\|_{L^p(\Omega_{\epsilon}^-)}                 \right)\\
+(1-s)& C_{n,p}\left(\left\|v\right\|_{L^p(\Omega\setminus \Omega_{\epsilon}^-)}^p+\left|\Omega_{\epsilon}^-\right|\left\|v\right\|_{L^\infty(\R^n\setminus B_R)}^p\right)\frac{R^{-sp}}{sp}\\
+(1-s)& C_{n,p} \epsilon^{-n-sp} \left(\left|B_R\setminus \Omega_{\epsilon}^{+}\right|\left\|v\right\|_{L^p(\Omega\setminus \Omega_{\epsilon}^-)}^p+\left|\Omega\setminus \Omega_{\epsilon}^-\right|\left\|v\right\|_{L^p(B_R\setminus \Omega_{\epsilon}^+)}^p\right)\\
+(1-s) & \frac{\left|\partial B_1\right|}{p(1-s)}\left\|\nabla v\right\|_{L^p(\Omega\setminus \Omega_{\epsilon}^{-}+\tilde{q}_\epsilon)}^p\\
+ (1-s) & C_{n,p} \left\|v\right\|_{L^p(\Omega_{\epsilon}^{+})}.
\end{split}
\end{equation*}
In particular, we deduce that 
\begin{equation*}
\limsup_{s\to 1}(1-s)\int_{\Omega}\int_{\R^n\setminus \Omega} \frac{\left|v(x)-v(y)\right|^p}{\left|x-y\right|^{n+sp}}\,dy\,dx\leq  \frac{\left|\partial B_1\right|}{p}\left\|\nabla v\right\|_{L^p(\Omega\setminus \Omega_{\epsilon}^{-}+\tilde{q}_\epsilon)}^p.
\end{equation*}
Finally, we notice that the right-hand side is small with $\epsilon$, which concludes the proof of Proposition~\ref{prop-cont}.
\end{proof}

The following result has been proved for $L^1(\R^n)$ in~\cite{MR2765717} and for $L^2(\R^n)$ in~\cite{MR4544090}. We provide here a proof for~$L^p(\R^n)$ for every $p \in (1,+\infty)$. 

\begin{prop}\label{colieradgber}
Let $p\in (1,+\infty)$. Then, there exists~$C_{n,p}\in (0,1)$ such that, for every $v\in L^p(\R^n)$, every~$h\in \R^n$, every
open bounded set $\Omega' \subset\R^n$ and every~$\rho\in (0,\left|h\right|]$,
\begin{equation}\label{slogbfvdtre54}
\left\| \tau_h v-v   \right\|_{L^p(\Omega')}\leq C_{n,p}\frac{\left|h\right|^p}{\rho^{n+p}}\int_{B_\rho}\left\|\tau_h v-v\right\|_{L^p(\Omega_{\left|h\right|}^{'})}\,dy.
\end{equation}
\end{prop}

\begin{proof}
Let $\phi\in C_c^1(B_1)$ be such that 
\begin{equation*}
\phi\geq 0 \quad\mbox{and}\quad \int_{B_1}\phi(x)\,dx=1.
\end{equation*}
For every $\rho\in (0,+\infty)$ we define the following two functions 
\begin{equation*}
U_\rho(x):=\frac{1}{\rho^n}\int_{B_\rho}v(x+y)\phi\left(\frac{y}{\rho}\right)\,dy\quad\mbox{and}\quad V_\rho(x):=\frac{1}{\rho^n}\int_{B_\rho} (v(x)-v(x+y))\phi\left(\frac{y}{\rho}\right)\,dy.
\end{equation*}
We notice that 
\begin{equation*}
v(x)=U_\rho(x)+V_\rho(x).
\end{equation*}

We also have that
\begin{equation}\label{klorfecg}
\begin{split}
\left|\tau_h v(x)-v(x)\right|^p&=\left|U_\rho(x+h)+V_\rho(x+h)-U_\rho(x)-V_\rho(x)\right|^p\\
&\leq C_p\left(\left|U_\rho(x+h)-U_\rho(x)\right|^p+\left|V_\rho(x+h)\right|^p+\left|V_\rho(x)\right|^p\right),
\end{split}
\end{equation}
for a suitable constant $C_p>0$.

By the  Jensen inequality, for every $x\in\R^n$,
\begin{equation}\label{nhbcefdrtlgy7}
\left|V_\rho(x)\right|^p\leq \frac{\left\|\phi\right\|_{L^\infty(B_1)}}{\rho^n}\int_{B_\rho}\left|v(x)-v(x+y)\right|^p\,dy.
\end{equation}
Moreover,
\begin{equation*}
\begin{split}
\nabla U_\rho(x)&=\frac{\nabla}{\rho^n}\int_{B_\rho}v(x+y)\phi\left(\frac{y}{\rho}\right)\,dy\\
&=\frac{\nabla}{\rho^n}\int_{B_\rho}v(z)\phi\left(\frac{z-x}{\rho}\right)\,dy\\
&=-\frac{1}{\rho^{n+1}}\int_{B_\rho}\left(v(z)-v(x)\right)\nabla \phi\left(\frac{z-x}{\rho}\right)\,dy\\
&=-\frac{1}{\rho^{n+1}}\int_{B_\rho}\left(v(x+y)-v(x)\right)\nabla \phi\left(\frac{y}{\rho}\right)\,dy.
\end{split}
\end{equation*}

Accordingly, using the fundamental Theorem of Calculus and the Jensen inequality we obtain that
\begin{equation*}
\begin{split}
\left|U_\rho(x+h)-U_\rho(x)\right|^p&=\left|\int_{0}^1 \nabla U_{\rho}(x+th)\cdot h\,dt\right|^p\\
&\leq \left|h\right|^p\int_{0}^1 \left|\nabla U_{\rho}(x+th)\right|^p\,dt\\
&\leq \omega_n^p\frac{\left|h\right|^p}{\rho^{n+p}}\left\|\nabla \phi\right\|_{L^\infty(B_1)}^p\int_{0}^1\int_{B_{\rho}}\left|\tau_y v(x+th)-v(x+th)\right|^p\,dy\,dt
\end{split}
\end{equation*}

On this account, using~\eqref{klorfecg} and~\eqref{nhbcefdrtlgy7}, we obtain that 
\begin{equation*}
\begin{split}
\left|\tau_hv(x)-v(x)\right|^p\leq C_p\bigg( & \omega_n^p\frac{\left|h\right|^p}{\rho^{n+p}}\left\|\nabla \phi\right\|_{L^\infty(B_1)}^p\int_{0}^1\int_{B_{\rho}}\left|\tau_y v(x+th)-v(x+th)\right|^p\,dy\,dt\\
&+\frac{\left\|\phi\right\|_{L^\infty(B_1)}}{\rho^n}\int_{B_\rho}\left|v(x)-v(x+y)\right|^p\,dy\\
&+\frac{\left\|\phi\right\|_{L^\infty(B_1)}}{\rho^n}\int_{B_\rho}\left|v(x+h)-v(x+h+y)\right|^p\,dy\bigg).
\end{split}
\end{equation*}
Integrating both sides in $\Omega'$ and taking $\rho\leq \left|h\right|$ we conclude.
\end{proof}

As a consequence of this result and Lemma~2.5 in~\cite{MR4544090} we obtain the following result. 
\begin{prop}\label{hndvcfwslop8}
There exists a constant $\bar{C}_{n,p}\in(0,+\infty)$ such that for every $v\in L^p(\R^n)$, every bounded open set $\Omega'\subset\R^n$, every $s\in(0,1)$ and every $h\in\R^n$, 
\begin{equation*}
\left\|\tau_h v-v\right\|_{L^p\left(\Omega'\right)}^p\leq \bar{C}_{n,p}\left|h\right|^{sp}(1-s)\int_{B_{\left|h\right|}} \frac{\left\|\tau_y v-v\right\|_{L^p\left(\Omega_{\left|h\right|}'\right)}^p}{\left|y\right|^{n+sp}}\,dy.
\end{equation*}  
\end{prop}

\begin{proof}
For every fixed $v\in L^p(\R^n)$, we define the function $g_v:[0,\left|h\right|]\to \R$ as 
\begin{equation*}
g_v(t):=\int_{\partial B_t} \left\|\tau_y v-v\right\|_{L^p\left(\Omega_{\left|h\right|}'\right)}^p\,dH_{y}^{n-1}.
\end{equation*} 
In particular, writing formula~\eqref{slogbfvdtre54} in polar coordinates we obtain that 
\begin{equation}\label{hutrgbchyr}
\left\|\tau_h v-v\right\|_{L^p(\Omega')}^p\leq C_{n,p} \frac{\left|h\right|^p}{\rho^{n+p}} \int_{0}^\rho g_v(t)\,dt. 
\end{equation}

We multiply both sides of~\eqref{hutrgbchyr} by $\rho^{p-sp-1}$ and integrate in the interval $\left[0,\left|h\right|\right]$. In this way, we find that
\begin{equation*}
 \left\|\tau_h v-v\right\|_{L^p(\Omega')}^p\leq p(1-s)C_{n,p} \int_{0}^{\left|h\right|}\frac{\left|h\right|^{sp}}{\rho^{n+sp+1}}\int_{0}^\rho g_v(t)\,dt\,d\rho
\end{equation*} 

Then, applying Lemma~2.5 in~\cite{MR4544090}, we arrive at
\begin{equation*}
 \left\|\tau_h v-v\right\|_{L^p(\Omega')}^p\leq p(1-s)C_{n,p} \left|h\right|^{sp}\int_{0}^{\left|h\right|}\frac{g_v(t)}{t^{n+sp}}\,dt.
\end{equation*}
Recalling the definition of $g_v$, the desired result plainly follows. 
\end{proof}

In the proof of Theorem~\ref{12c-dgbnbc5543} we will also make use of the following binomial Theorem. 

\begin{prop}[Newton's generalized binomial Theorem]
Let $a>b>0$ and $p\in (1,+\infty)$. Then, we have that 
\begin{equation}\label{Bolo}
(a-b)^p =\sum_{k=0}^{+\infty}\begin{pmatrix}
p\\k
\end{pmatrix}a^{p-k}(-b)^k.
\end{equation}
\end{prop}

\end{appendix}

\end{document}